\documentclass[a4paper, 11pt, twoside]{scrartcl}
\pdfoutput=1

\title{Numerical Methods for the Magnetic Induction Equation with Hall Effect and
       Projections onto Divergence-Free Vector Fields}
\author{Hendrik Ranocha, Katharina Ostaszewski, Philip Heinisch}
\date{2nd October 2018}
\titlehead{\includegraphics[width=0.5\textwidth]{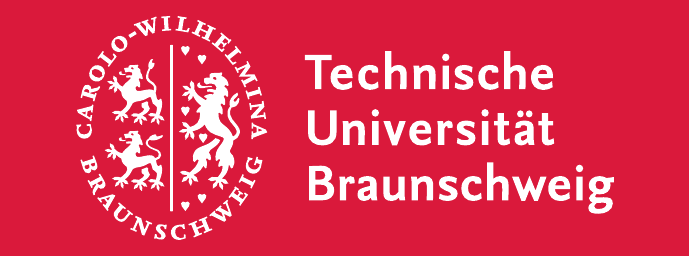}}

\usepackage[utf8]{luainputenc}
\usepackage[UKenglish]{babel}
\usepackage{csquotes}

\usepackage[a4paper, top=0.51in, bottom=0.79in, left=0.999in, right=0.99in]{geometry}

\usepackage[plainpages=false,pdfpagelabels,hidelinks,unicode]{hyperref}
\makeatletter
\hypersetup{pdfauthor={\@author}}
\hypersetup{pdftitle={\@title}}
\makeatother

\usepackage[%
  backend=biber,
  style=numeric-comp,
  giveninits=true, uniquename=init, 
  natbib=true,
  url=true,
  doi=true,
  isbn=false,
  backref=false,
  maxnames=99,
  ]{biblatex}
\usepackage{amsmath}
\usepackage{amssymb}
\usepackage{commath}
\usepackage{mathtools}
\usepackage{bbm}
\usepackage{nicefrac}

\usepackage{siunitx}

\usepackage{amsthm}
\usepackage{thmtools}
\usepackage{etoolbox}
\makeatletter
\patchcmd{\thmt@setheadstyle}
 {\bgroup\thmt@space}
 {\thmt@space}
 {}{}
\patchcmd{\thmt@setheadstyle}
 {\egroup\fi}
 {\fi}
 {}{}
\makeatother
\declaretheoremstyle[
  bodyfont=\normalfont\itshape,
  headformat=\NAME\ \NUMBER\NOTE,
]{myplain}
\declaretheoremstyle[
  headformat=\NAME\ \NUMBER\NOTE,
]{mydefinition}
\newcommand{\envqed}{{\lower-0.3ex\hbox{$\triangleleft$}}}
\declaretheorem[style=myplain,numberwithin=section]{theorem}
\declaretheorem[style=myplain,numberlike=theorem]{lemma}

\declaretheorem[style=myplain,numberlike=theorem]{proposition}

\declaretheorem[style=mydefinition,numberlike=theorem,qed=\envqed]{remark}
\declaretheorem[style=mydefinition,numberlike=theorem,qed=\envqed]{example}

\usepackage{bbm} 

\usepackage{ifluatex}
\ifluatex
  \usepackage[no-math]{fontspec}
\else
  \usepackage[T1]{fontenc}
\fi
\usepackage{newpxtext,newpxmath}

\usepackage{color}
\usepackage{graphicx}
\usepackage[small]{caption}
\usepackage{subcaption}

\usepackage{pgfplots}
\usepackage{tikzscale}
\pgfplotsset{compat=1.13}
\usepgfplotslibrary{groupplots}
\usepgfplotslibrary{external}
\begingroup\expandafter\expandafter\expandafter\endgroup
\expandafter\ifx\csname pdfsuppresswarningpagegroup\endcsname\relax
\else
\fi

\usepackage{booktabs}
\usepackage{rotating}
\usepackage{multirow}

\makeatletter
\newcommand\mynewtag[2]{#1\def\@currentlabel{#1}\label{#2}}
\makeatother

\usepackage{multicol}
\usepackage{enumitem}

\usepackage{calc}
\usepackage{xparse}

\newcommand{\appendixref}[1]{\autoref{#1}}

\NewDocumentCommand{\mat}{mo}{%
  \IfValueTF{#2}{%
    \underline{\underline{#1}}{#2}
  }{%
    \underline{\underline{#1}}\,
  }%
}
\newcommand{\diag}[1]{\operatorname{diag}\left(#1\right)}

\AtBeginDocument{%
  \renewcommand{\div}{\operatorname{div}}
}
\newcommand{\grad}{\operatorname{grad}}
\newcommand{\curl}{\operatorname{curl}}

\newcommand{\scp}[2]{\left\langle{#1,\, #2}\right\rangle}

\DeclarePairedDelimiterX\newset[1]\lbrace\rbrace{\setaux #1||\endsetaux}
\def\setaux#1|#2|#3\endsetaux{\if\relax\detokenize{#2}\relax #1 \else #1 \;\delimsize\vert\; #2 \fi}
\renewcommand{\set}[1]{\newset*{#1}}

\newcommand{\vect}[1]{\begin{pmatrix} #1 \end{pmatrix}}

\newcommand{\I}{\operatorname{I}}

\newcommand{\fextj}{f^{\mathrm{ext},j}}

\newcommand{\cfl}{\mathrm{cfl}}

\newcommand{\VOL}{\mathrm{VOL}}
\newcommand{\SURF}{\mathrm{SURF}}

\AtBeginDocument{%
  \let\rho\varrho
  \let\phi\varphi
  \let\epsilon\varepsilon
}
\newcommand{\N}{\mathbb{N}}
\newcommand{\R}{\mathbb{R}}

\newsavebox{\DelimiterBox}
\newlength{\DelimiterHeight}
\newlength{\DelimiterDepth}
\newsavebox{\ArgumentBox}
\newlength{\ArgumentHeight}
\newlength{\ArgumentDepth}
\newlength{\ResizedDelimiterHeight}
\newlength{\ResizedDelimiterDepth}

\newcommand{\encloseby}[3]{%
  \savebox{\ArgumentBox}{$\displaystyle #1$}%
  \settoheight{\ArgumentHeight}{\usebox{\ArgumentBox}}%
  \settodepth{\ArgumentDepth}{\usebox{\ArgumentBox}}%
  \savebox{\DelimiterBox}{#2}%
  \settoheight{\DelimiterHeight}{\usebox{\DelimiterBox}}%
  \settodepth{\DelimiterDepth}{\usebox{\DelimiterBox}}%
  \setlength{\ResizedDelimiterHeight}{%
    \maxof{1.2\ArgumentHeight}{\DelimiterHeight}%
  }
  \setlength{\ResizedDelimiterDepth}{%
    \maxof{1.2\ArgumentDepth}{\DelimiterDepth}%
  }
  \raisebox{-\ResizedDelimiterDepth}{%
    \resizebox{\width}{\ResizedDelimiterHeight+\ResizedDelimiterDepth}{%
      \raisebox{\DelimiterDepth}{#2}%
    }%
  }
  #1
  \raisebox{-\ResizedDelimiterDepth}{%
    \resizebox{\width}{\ResizedDelimiterHeight+\ResizedDelimiterDepth}{%
      \raisebox{\DelimiterDepth}{#3}%
    }%
  }
}

\ifluatex
  \newcommand{\mean}[1]{\encloseby{#1}{$\{\mkern-5mu\{$}{$\}\mkern-5mu\}$}}
  \newcommand{\jump}[1]{\encloseby{#1}{$[\mkern-4mu[$}{$]\mkern-4mu]$}}
\else
  \newcommand{\mean}[1]{\encloseby{#1}{$\{\mkern-6mu\{$}{$\}\mkern-6mu\}$}}
  \newcommand{\jump}[1]{\encloseby{#1}{$[\mkern-3mu[$}{$]\mkern-3mu]$}}
\fi

\newcommand{\1}[1]{\mathbbm{1}_{\set{#1}}}
\newcommand{\Bnum}{B^\mathrm{num}}

\begin{document}

\maketitle

\begin{abstract}
  The nonlinear magnetic induction equation with Hall effect can be used to model
magnetic fields, e.g. in astrophysical plasma environments. In order to give
reliable results, numerical simulations should be carried out using effective
and efficient schemes. Thus, high-order stable schemes are investigated here.

Following the approach provided recently by Nordström (J Sci Comput 71.1,
pp. 365--385, 2017), an energy analysis for both the linear and the nonlinear
induction equation including boundary conditions is performed at first. Novel
outflow boundary conditions for the Hall induction equation are proposed,
resulting in an energy estimate.
Based on an energy analysis of the initial boundary value problem at the continuous
level, semidiscretisations using summation by parts (SBP) operators and simultaneous
approximation terms are created. Mimicking estimates at the continuous level,
several energy stable schemes are obtained in this way and compared in numerical
experiments.
Moreover, stabilisation techniques correcting errors in the numerical divergence
of the magnetic field via projection methods are studied from an energetic
point of view in the SBP framework. In particular, the treatment of boundaries
is investigated and a new approach with some improved properties is proposed.

\end{abstract}

\section{Introduction}
\label{sec:introduction}

Numerical plasma simulations have many applications not only in space physics,
but also in engineering. In recent years increasingly powerful computers have
caused a quick adoption of different numerical models to simulate the interaction
of the solar wind with different celestial objects \citep{koenders2015innerComa,
huang2018hall}, space weather \citep{toth2005SWMF}, and the performance of plasma
engines \citep{Boyd2006,Nishida2012}. While advances in computational power and
available memory have allowed for increasingly accurate physical models with
higher spatial and temporal accuracy, the numerical methods used to solve the
underlying equations have started to become a limiting factor. In the past
numerical instabilities and other artefacts were commonly small compared to the
errors introduced by insufficient physical modelling, but with increased model
quality and accuracy, shortcomings in the numerical methods became noticeable.

This article is concerned with the numerical treatment of one of the primary
equations behind all numeric plasma models: the magnetic induction equation. It
is widely used in different models, except some completely kinetic approaches, and is
usually written in non-dimensional form as
\begin{equation}
\label{eq:induction-full}
  \partial_t B
  =
  \underbrace{\nabla \times (u \times B)}_{\text{transport term}}
  \underbrace{- \nabla \times \biggl( \frac{\nabla \times B}{\rho} \times B \biggr)}_{\text{Hall term}},
\end{equation}
where $B$ is the magnetic field, $u$ the particle velocity, $\rho$ the particle
charge density, and $\nabla \times B$
the curl of $B$. If not mentioned otherwise, all functions depend on time $t \in (0,T)$
and space $x = (x_1,x_2,x_3) \in \Omega \subseteq \R^3$. The first term on the
right hand side of \eqref{eq:induction-full} is often called \emph{transport term}
and the second one is the \emph{Hall term}.
In general, the induction equation \eqref{eq:induction-full} is supplemented with
the divergence constraint $\div B = 0$ on the magnetic field. Of course, suitable
initial and boundary conditions have to be given.

The induction equation \eqref{eq:induction-full} can be used as part of larger
physical models, in particular magnetohydrodynamics (MHD).
Then, there are additional equations determining the particle charge density and
velocity. Considering the induction equation \eqref{eq:induction-full} as a
model on its own, the quantities $u$ and $\rho$ are given data.
While the Hall term on the right hand side of \eqref{eq:induction-full} can be dropped for some applications,
many MHD models require it to accurately describe processes such as the evolution of the protostellar disk \citep{Ebrahimi2011,Simon2013} or the comet solar wind interaction \citep{huang2018hall}.
Other terms may also be added to extend the model and describe additional physical processes governed primarily by resistive or electron inertia effects.
Using $\div (\nabla \times \cdot) = 0$, the divergence constraint $\div B = 0$
will be automatically fulfilled if the initial condition $B^0$ satisfies it,
all functions are sufficiently smooth, and boundaries are ignored.

The magnetic induction equation \eqref{eq:induction-full} has been considered
in different forms in the literature. In \cite{fuchs2009stable, fuchs2009splitting,
koley2009higher, mishra2010stability}, only the transport term has been considered.
A linear resistive term has been added in \cite{koley2011higher}. Another variant
of the induction equation with Hall effect without discussion of boundary conditions
has been investigated in \cite{corti2012stable}.

The fundamental technique used in this article is the energy method, cf.
\cite[chapters 8 and 11]{gustafsson2013time}. Physically, it can be motivated as
follows. The magnetic energy is proportional to $\abs{B}^2$ and fulfils a
secondary balance law that can be determined using the induction equation
\eqref{eq:induction-full}, since $\partial_t \abs{B}^2 = 2 B \cdot \partial_t B$
for sufficiently smooth solutions. Boundary conditions have to be given such that
the magnetic energy remains bounded and can be estimated by given initial and
boundary data. This behaviour should hold for both the partial differential equation
(PDE) at the continuous level and the discrete variant.

Following the approach of \citet{nordstrom2017roadmap}, in order
to know what should be mimicked by the discretisation, the initial boundary value
problem (IBVP) will be investigated at first at the continuous level. Results
obtained there are also useful on their own and can be applied
to different discretisations, not only the ones considered in this article. Boundary
conditions will be imposed both strongly (i.e. by enforcing given boundary values
exactly) and weakly (i.e. by adding an appropriate penalty term to the PDE).

In order to mimic estimates obtained from the energy method semidiscretely,
summation by parts (SBP) derivative operators will be used \cite{kreiss1974finite,
strand1994summation}. The weak imposition of boundary conditions is mimicked via
simultaneous approximation terms (SATs) \cite{carpenter1994time, carpenter1999stable}.
Further information can be found in the review articles \cite{svard2014review,
fernandez2014review} and references cited therein.
One method to obtain energy estimates for problems with varying coefficients or
nonlinear ones is the application of certain splittings. Such techniques have been
used successfully in the literature, cf. \cite{olsson1994energy, yee2000entropy,
sandham2002entropy, nordstrom2006conservative, kennedy2008reduced, morinishi2010skew,
sjogreen2010skew, kopriva2014energy, gassner2016well, wintermeyer2017entropy,
ranocha2018generalised, ranocha2017extended, sjogreen2017skew, sjogreen2018high}.

Although SBP operators have been developed in the context of finite difference (FD)
methods and this setting will be used here, they can also be found in various other
frameworks including finite volume (FV) \cite{nordstrom2001finite, nordstrom2003finite},
discontinuous Galerkin (DG) \cite{gassner2013skew, fernandez2014generalized}, and
the recent flux reconstruction/correction procedure via reconstruction schemes
\cite{huynh2007flux, huynh2014high, ranocha2016summation}. Thus, basic results
about energy estimates and boundary conditions obtained here can also be applied
to these schemes.

This article is structured as follows. At first, the linear magnetic induction
equation using only the transport term is investigated in section~\ref{sec:transport-term}.
After the derivation of energy estimates and admissible boundary conditions,
the concept of SBP operators is briefly reviewed and applied to obtain stable
semidiscretisations. Afterwards, the nonlinear induction equation with Hall effect
is considered in section~\ref{sec:Hall-term}. Using the same basic approach,
energy stable outflow boundary conditions are proposed and studied both at the
continuous and the semidiscrete level. In section~\ref{sec:div-constraint}, the
focus lies on the divergence constraint. Since it has been used widely, the projection
method enforcing this constraint is studied from the point of view of the
energy method and corresponding boundary conditions are investigated.
Thereafter, results of numerical experiments are presented in section~\ref{sec:numerical-results}. Finally, a summary and discussion of the obtained results is given in section~\ref{sec:summary}.

\section{Linear Magnetic Induction Equation}
\label{sec:transport-term}

This section is focused on the transport term of the magnetic induction equation
\eqref{eq:induction-full}. Thus, the linear equation
\begin{equation}
\label{eq:induction-transport-conservative}
  \partial_t B
  =
  \nabla \times (u \times B)
\end{equation}
with divergence constraint $\div B = 0$ and suitable initial and boundary
conditions will be investigated. Therefore, the PDE will be rewritten using the
divergence constraint such that the energy rate can be calculated via
\begin{equation}
  \od{}{t} \norm{B}_{L^2(\Omega)}^2
  =
  \od{}{t} \int_\Omega \abs{B}^2
  =
  2 \int_\Omega B \cdot \partial_t B
\end{equation}
and inserting the PDE, leading to admissible boundary conditions. Implementing
these in a weak form yields an energy estimate involving given initial and boundary data.
Finally, using summation by parts operators, a semidiscretisation mimicking these
properties will be constructed.

\subsection{Continuous Setting}

The $i$-th component of the transport term $\nabla \times (u \times B)$ can be
written using the totally antisymmetric Levi-Civita symbol $\epsilon_{ijk}$ as
\begin{equation}
  \left[ \nabla \times (u \times B) \right]_i
  =
  \epsilon_{ijk} \partial_j \bigl( \epsilon_{klm} u_l B_m \bigr)
  =
  \partial_j \bigl( u_i B_j - u_j B_i \bigr),
\end{equation}
where summation over repeated indices is implied. In order to obtain an energy
estimate, the product rule can be used to rewrite the transport term as
\begin{equation}
\label{eq:induction-transport-energy-rate-1}
  \partial_j \bigl( u_i B_j - u_j B_i \bigr)
  =
  \underbrace{u_i \partial_j B_j}_{(i)}
  \underbrace{+ B_j \partial_j u_i
  - \frac{1}{2} B_i \partial_j u_j}_{(ii)}
  \underbrace{- \frac{1}{2} u_j \partial_j B_i
  - \frac{1}{2} \partial_j (u_j B_i)}_{(iii)}.
\end{equation}
If $\div B = 0$, the first term $(i)$ on the right hand side vanishes. Dropping it can in
general be interpreted as adding $- u \div B$ to the right hand side of the magnetic
induction equation \eqref{eq:induction-transport-conservative}, as studied
in \cite{godunov1972symmetric, powell1994approximate, powell1999solution} for the
MHD equations. Investigations in the context of numerical schemes for the
induction equation can be found in \cite{fuchs2009stable, koley2009higher,
mishra2010stability}. Without dropping the term $(i)$, the system is not symmetric.
Moreover, it cannot be symmetrised, and the energy method cannot be applied, cf.
\cite{mishra2010stability} for the two-dimensional case.

The terms $(ii)$ of \eqref{eq:induction-transport-energy-rate-1} contain no
derivatives of the magnetic field and can be interpreted as source terms describing
the influence of the particles on the magnetic field. The remaining terms $(iii$) contain
derivatives of $B$. Multiplying by the magnetic field and integrating over a
volume $\Omega$ such that the divergence theorem can be used yields
\begin{equation}
\begin{aligned}
  \int_\Omega B \cdot \partial_t B
  &=
  \int_\Omega \left(
      B_i B_j \partial_j u_i
    - \frac{1}{2} B_i B_i \partial_j u_j
    - \frac{1}{2} B_i u_j \partial_j B_i
    - \frac{1}{2} B_i \partial_j (u_j B_i)
  \right)
  \\
  &=
  \int_\Omega \left( B_i B_j \partial_j u_i - \frac{1}{2} B_i B_i \partial_j u_j \right)
  - \int_{\partial\Omega} \frac{1}{2} B_i B_i u_j \nu_j,
\end{aligned}
\end{equation}
where $\nu = (\nu_j)_j$ is the outward unit normal at $\partial\Omega$. This
proves
\begin{lemma}
\label{lem:energy-rate-transport}
  If the linear induction equation \eqref{eq:induction-transport-conservative}
  is written in the form
  \begin{equation}
  \label{eq:induction-transport-split}
    \partial_t B_i
    =
      B_j \partial_j u_i
    - \frac{1}{2} B_i \partial_j u_j
    - \frac{1}{2} u_j \partial_j B_i
    - \frac{1}{2} \partial_j (u_j B_i),
  \end{equation}
  the energy rate can be obtained using only integration by parts via
  \begin{equation}
  \label{eq:energy-rate-transport}
    \int_\Omega B \cdot \partial_t B
    =
    \int_\Omega \left( B_i B_j \partial_j u_i - \frac{1}{2} B_i B_i \partial_j u_j \right)
    - \int_{\partial\Omega} \frac{1}{2} B_i B_i u_j \nu_j.
  \end{equation}
\end{lemma}

Following classical arguments for linear PDEs, boundary conditions should be given
such that an energy estimate can be obtained, cf. \cite{nordstrom2017roadmap}.
Concentrating on the surface term in \eqref{eq:energy-rate-transport}, an energy
growth can only occur if $u_j \nu_j = u \cdot \nu < 0$. Since $\nu$ is the outward
normal and $u$ the particle velocity, this corresponds exactly to the case of an inflow,
in accordance with physical intuition and the frozen in theorem \citep{alfven1942existence}.
Thus, the initial boundary value problem for the induction equation
\eqref{eq:induction-transport-split} becomes
\begin{equation}
\label{eq:induction-transport-strong}
\begin{aligned}
  \partial_t B_i
  &=
    B_j \partial_j u_i
  - \frac{1}{2} B_i \partial_j u_j
  - \frac{1}{2} u_j \partial_j B_i
  - \frac{1}{2} \partial_j (u_j B_i),
  && \text{in } (0,T) \times \Omega,
  \\
  B(t,x) &= B^b(t,x),
  && \text{if } u(t,x) \cdot \nu(x) < 0 \text{ on } \partial\Omega,
  \\
  B(0,x) &= B^0(x),
  && x \in \Omega,
\end{aligned}
\end{equation}
where $B^0$ and $B^b$ are given initial and boundary data. Using these supplementary
conditions, the magnetic energy can be estimated as follows.
\begin{lemma}
\label{lem:induction-transport-strong}
  A sufficiently smooth solution $B$ of the linear induction equation
  \eqref{eq:induction-transport-strong} fulfils
  \begin{equation}
  \label{eq:induction-transport-strong-energy-rate}
    \od{}{t} \norm{B(t)}_{L^2(\Omega)}^2
    =
    2 \int_\Omega B \cdot \partial_t B
    \leq
    9 \norm{\nabla u(t)}_{L^\infty(\Omega)} \norm{B(t)}_{L^2(\Omega)}^2
    + \norm{u(t)}_{L^\infty(\partial\Omega)} \norm{B^b(t)}_{L^2(\partial\Omega)}^2
  \end{equation}
  and
  \begin{equation}
  \label{eq:induction-transport-strong-energy-estimate}
    \norm{B(t)}_{L^2(\Omega)}^2
    \leq
    \exp\bigl( 9 \norm{\nabla u}_\infty t \bigr)
    \left(
      \norm{B^0}_{L^2(\Omega)}^2
      + \int_0^t \norm{u(t)}_{L^\infty(\partial\Omega)}
        \norm{B^b(t)}_{L^2(\partial\Omega)}^2 \dif t
    \right).
  \end{equation}
\end{lemma}
\begin{proof}
  If the particle velocity $u$ and its partial derivatives are bounded, the energy
  rate can be estimated using \eqref{eq:energy-rate-transport} via
  \begin{equation}
  \begin{aligned}
    \int_\Omega B \cdot \partial_t B
    &=
    \int_\Omega \left( B_i B_j \partial_j u_i - \frac{1}{2} B_i B_i \partial_j u_j \right)
    - \int_{\partial\Omega} \frac{1}{2} B_i B_i u_j \nu_j
    \\
    &\leq
    \norm{\nabla u(t)}_{L^\infty(\Omega)} \sum_{i,j} \int_\Omega \abs{B_i} \abs{B_j}
    + \frac{3}{2} \norm{\nabla u(t)}_{L^\infty(\Omega)} \int_\Omega \abs{B}^2
    + \frac{1}{2} \norm{u(t)}_{L^\infty(\partial\Omega)} \int_{\partial\Omega} \abs{B^b}^2.
  \end{aligned}
  \end{equation}
  For this estimate, the boundary $\partial\Omega$ has been divided into two parts:
  the inflow part $\partial\Omega_\mathrm{in}$ (where $u \cdot \nu < 0$) and the outflow
  part $\partial\Omega_\mathrm{out}$ (where $u \cdot \nu \geq 0$). On $\partial\Omega_\mathrm{in}$,
  the boundary condition $B = B^b$ has been inserted. The integral over
  $\partial\Omega_\mathrm{out}$ is non-positive, since $B_i B_i u_j \nu_j = \abs{B}^2 u \cdot \nu
  \geq 0$ there. The infinity norm is $\norm{\nabla u(t)}_{L^\infty(\Omega)} = \max_{i,j}
  \norm{\partial_j u_i(t)}_{L^\infty(\Omega)}$. Using
  \begin{equation}
    \sum_{i,j} \int_\Omega \abs{B_i} \abs{B_j}
    \leq
    \sum_{i,j} \int_\Omega \frac{1}{2} \Bigl( \abs{B_i}^2 + \abs{B_j}^2 \Bigr)
    =
    3 \norm{B}_{L^2(\Omega)}^2
  \end{equation}
  yields
  \begin{equation}
    \int_\Omega B \cdot \partial_t B
    \leq
    \frac{9}{2} \norm{\nabla u(t)}_{L^\infty(\Omega)} \norm{B(t)}_{L^2(\Omega)}^2
    + \frac{1}{2} \norm{u(t)}_{L^\infty(\partial\Omega)} \norm{B^b(t)}_{L^2(\partial\Omega)}^2.
  \end{equation}
  Abbreviating $\norm{\nabla u}_\infty = \max_{i,j}
  \norm{\partial_j u_i}_{L^\infty((0,T) \times \Omega)}$, the energy estimate
  \eqref{eq:induction-transport-strong-energy-estimate} follows due to Grönwall's
  inequality.
\end{proof}

Instead of the strong implementation of the boundary conditions as in
\eqref{eq:induction-transport-strong}, the boundary conditions can also be implemented
in a weak form. Since this form is related directly to semidiscretisations using
SBP operators and SATs, it will be used in the following. Therefore, a lifting
operator $L$ is used. Similar to a Dirac measure concentrated on the boundary
$\partial\Omega$, it fulfils
\begin{equation}
  \int_\Omega u \cdot L(\psi) = \int_{\partial\Omega} \phi \cdot \psi
\end{equation}
for smooth (and possibly vector valued) functions $\phi,\psi$, cf. \cite{arnold2002unified,
wang2009unifying, nordstrom2017roadmap}. In the semidiscrete setting, such a lifting
operator is mainly given by a multiplication by the inverse grid size as described
in the following subsection. Imposing the boundary data $B^b$ weakly yields the IBVP
\begin{equation}
\label{eq:induction-transport-weak}
\begin{aligned}
  \partial_t B_i
  &=
    B_j \partial_j u_i
  - \frac{1}{2} B_i \partial_j u_j
  - \frac{1}{2} u_j \partial_j B_i
  - \frac{1}{2} \partial_j (u_j B_i)
  \\
  &\phantom{=} + L\bigl( \1{u \cdot \nu < 0} (u \cdot \nu) (B_i - B^b_i) \bigr),
  && \text{in } (0,T) \times \Omega,
  \\
  B(0,x) &= B^0(x),
  && x \in \Omega,
\end{aligned}
\end{equation}
where $\1{u \cdot \nu < 0}$ is one where $u \cdot \nu < 0$ and zero elsewhere.
Similar to the strong form of the boundary conditions, this yields
\begin{lemma}
\label{lem:induction-transport-weak}
  A sufficiently smooth solution $B$ of the linear induction equation
  \eqref{eq:induction-transport-weak} with weak implementation of the boundary
  condition satisfies the energy estimate for the strong implementation given in
  Lemma~\ref{lem:induction-transport-strong}. If the boundary condition is not
  fulfilled exactly, there is an additional dissipative term.
\end{lemma}
\begin{proof}
  As in the proof of Lemma~\ref{lem:induction-transport-strong}, the energy rate
  can be estimated using \eqref{eq:energy-rate-transport} via
  \begin{equation}
    \int_\Omega B \cdot \partial_t B
    =
    \int_\Omega \left( B_i B_j \partial_j u_i - \frac{1}{2} B_i B_i \partial_j u_j \right)
    - \int_{\partial\Omega} \frac{1}{2} B_i B_i u_j \nu_j
    + \int_\Omega B_i L\bigl( \1{u \cdot \nu < 0} u_j \nu_j (B_i - B^b_i) \bigr).
  \end{equation}
  Only the last term on the right hand side is new and can be rewritten as
  \begin{equation}
    \int_\Omega B_i L\bigl( \1{u \cdot \nu < 0} u_j \nu_j (B_i - B^b_i) \bigr)
    =
    \int_{\partial\Omega} \1{u \cdot \nu < 0} u_j \nu_j B_i (B_i - B^b_i).
  \end{equation}
  Thus, the surface terms are
  \begin{equation}
    - \int_{\partial\Omega} \left(
      \frac{1}{2} B_i B_i u_j \nu_j
      - \1{u \cdot \nu < 0} u_j \nu_j B_i (B_i - B^b_i)
    \right)
    =
    - \int_{\partial\Omega} u_j \nu_j \left(
      \frac{1}{2} B_i B_i - \1{u \cdot \nu < 0} (B_i B_i - B_i B^b_i)
    \right).
  \end{equation}
  The integrand is the same as for the strong implementation of the boundary
  conditions where $u \cdot \nu \geq 0$, i.e. $- u_j \nu_j B_i B_i \leq 0$.
  Elsewhere, the integrand is
  \begin{equation}
    - u_j \nu_j \left(
      \frac{1}{2} B_i B_i - (B_i B_i - B_i B^b_i)
    \right)
    =
    - \frac{1}{2} u_j \nu_j B^b_i B^b_i
    + \frac{1}{2} \underbrace{u_j \nu_j}_{< 0} (B_i - B^b_i) (B_i - B^b_i)
    \leq
    - \frac{1}{2} u_j \nu_j B^b_i B^b_i.
  \end{equation}
  Hence, an additional dissipative term
  \begin{equation}
    - \int_{\partial\Omega} \1{u \cdot \nu < 0} u_j \nu_j (B_i - B^b_i) (B_i - B^b_i)
    \leq
    0
  \end{equation}
  appears in the estimate of the energy rate $\od{}{t} \norm{B(t)}_{L^2(\Omega)}^2$
  compared to the strong form of the boundary condition.
\end{proof}

\subsection{Summation by Parts Operators}
\label{sec:SBP}

Using the formulation \eqref{eq:induction-transport-weak} of the magnetic induction
equation with weak implementation of the boundary condition, the estimates of the
energy rate \eqref{eq:induction-transport-strong-energy-rate} and of the energy
\eqref{eq:induction-transport-strong-energy-estimate} have been obtained using
only integration by parts and properties of the lifting operator $L$. Thus, these
have to be mimicked discretely in order to obtain similar estimates at the semidiscrete
level. Summation by parts operators and simultaneous approximation terms are these
discrete analogues.

Before presenting a semidiscretisation of \eqref{eq:induction-transport-weak},
the concept of SBP operators will be described briefly. Since finite difference
methods on Cartesian grids will be used in the following, the one dimensional
setting is described at first.

The given domain $\Omega = [x_L, x_R]$ is discretised as a uniform grid with nodes
$x_L = x_1 < x_2 < \dots < x_N = x_R$. A function $u$ is represented discretely as
a vector $(u^{(a)})_a$, where the components are the values at the grid nodes, i.e.
$u^{(a)} = u(x_a)$. Nonlinear operations are performed componentwise. Thus, the product
of two functions $u$ and $v$ is represented by the Hadamard product of the corresponding
vectors, i.e. $(u v)^{(a)} = u^{(a)} v^{(a)}$. By a slight abuse of notation, $u$
may represent the vector of coefficients or the diagonal multiplication matrix $\diag{u}$,
performing this multiplication of discretised functions.

Since summation by parts should mimic integration by parts, derivatives
and integrals have to be discretised. Therefore, the derivative operator is represented
by a matrix $D$, i.e. $D u \approx \partial_x u$. The integral over $\Omega$ is
interpreted as the $L^2$ scalar product and represented by a symmetric and positive
definite norm/mass matrix\footnote{The name ``mass matrix'' is common for finite
element methods such as discontinuous Galerkin methods, while ``norm matrix'' is
more common in the finite difference community. Here, both names will be used
equivalently.} $M$, i.e.
\begin{equation}
  u^T M v
  =
  \scp{u}{v}_M
  \approx
  \scp{u}{v}_{L^2(\Omega)}
  =
  \int_\Omega u \cdot v.
\end{equation}
Since boundary nodes are included, integration with respect to the outer unit normal
$\nu$ at $\partial\Omega$ as in the divergence theorem is given by the difference of
boundary values. This bilinear form is represented by the matrix $E = \diag{-1,0,\dots,0,1}$.
Together, these operators mimic integration by parts discretely via
\begin{gather}
\label{eq:SBP-IBP}
  \begin{array}{ccc}
    \underbrace{
      u^T M D v
      + u^T D^T M v
    }
    & = &
    \underbrace{
      u^T E v,
    }
    \\
    \rotatebox{90}{$\!\approx\;$}
    &&
    \rotatebox{90}{$\!\!\approx\;$}
    \\
    \overbrace{
      \int_{x_L}^{x_R} u \, (\partial_x v)
      + \int_{x_L}^{x_R} (\partial_x u) \, v
    }
    & = &
    \overbrace{
      u \, v \big|_{x_L}^{x_R}
    },
  \end{array}
\end{gather}
if the SBP property
\begin{equation}
\label{eq:SBP}
  M D + D^T M = E
\end{equation}
is fulfilled. Finally, the discrete version of the lifting operator $L$ is
$M^{-1} \abs{E}$, since
\begin{equation}
  \int_\Omega u \cdot L(v)
  \approx
  u^T M \left( M^{-1} \abs{E} v \right)
  =
  u^T \abs{E} v
  =
  \bigl( u^{(N)} \cdot v^{(N)} + u^{(1)} \cdot v^{(1)} \bigr)
  \approx
  \int_{\partial\Omega} u \cdot v.
\end{equation}

Here, only diagonal norm SBP operators are considered, i.e. those SBP operators
with diagonal mass matrices $M$. In this case, discrete integrals are evaluated
using the quadrature provided by the weights of the diagonal mass matrix. For
classical diagonal norm SBP operators, the order of accuracy is $2p$ in the
interior and $p$ at the boundaries, allowing a global convergence order of
$p+1$ for hyperbolic problems \cite{svard2006order, svard2014note}. Here, SBP
operators will be referred to by their interior order of accuracy $2p$.

\begin{example}
\label{ex:SBP-2}
  The classical second order accurate SBP operators are
  \begin{equation}
  \label{eq:SBP-2}
    D
    =
    \frac{1}{2 \Delta x}
    \begin{pmatrix}
      -2 & 2 \\
      -1 & 0 & 1 \\
      & \ddots & \ddots & \ddots \\
      && -1 & 0 & 1 \\
      &&& -2 & 2
    \end{pmatrix},
    \qquad
    M
    =
    \Delta x
    \begin{pmatrix}
      \frac{1}{2} \\
      & 1 \\
      && \ddots \\
      &&& 1 \\
      &&&& \frac{1}{2}
    \end{pmatrix},
  \end{equation}
  where $\Delta x$ is the grid spacing.
  Thus, the first derivative is given by the standard second order central derivative
  in the interior and by one sided derivative approximations at the boundaries.
\end{example}

In multiple space dimensions, tensor product operators will be used, i.e. the one
dimensional SBP operators are applied accordingly in each dimension. Thus, they
are of the form
\begin{equation}
  D_1 = D_x \otimes \I_y \otimes \I_z, \quad
  D_2 = \I_x \otimes D_y \otimes \I_z, \quad
  D_3 = \I_x \otimes \I_y \otimes D_z,
\end{equation}
where $\I_{x,y,z}$ are identity matrices and $D_{x,y,z}$ are one dimensional SBP
derivative operators in the corresponding coordinate directions. The boundary
operators are
\begin{equation}
  E_1 = E_x \otimes M_y \otimes M_z, \quad
  E_2 = M_x \otimes E_y \otimes M_z, \quad
  E_3 = M_x \otimes M_y \otimes E_z.
\end{equation}
They fulfil $u^T E_i v \approx \int_{\partial\Omega} u \cdot v  \, \nu_i$. Sometimes,
the boundary integral operator
\begin{equation}
  E = \abs{E_1} + \abs{E_2} + \abs{E_3}
\end{equation}
will be used. Finally, the mass matrix is $M = M_x \otimes M_y \otimes M_z$.
\begin{remark}
  The standard tensor product discretisations of the divergence and curl operators
  given above satisfy $\div \curl = 0$, since the discrete derivative operators
  commute, i.e. $D_j D_i = D_i D_j$. However, the imposition of boundary conditions
  has to be taken into account. Thus, the discrete divergence of the magnetic field
  will not remain zero, even if the initial data are discretely divergence free.
  Hence, even the direct discretisation of the conservative form of
  $\nabla \times (u \times B)$ without source term will not result in discretely
  divergence free magnetic fields. Thus, the divergence constraint will be considered
  in more detail in section~\ref{sec:div-constraint}.
\end{remark}

\subsection{Semidiscrete Setting}
\label{sec:transport-term-semidiscrete}

Replacing derivatives $\partial_j$  by SBP operators $D_j$ and the lifting operator of
terms multiplied by $\nu_j$ by $M^{-1} E_j$ results in the following semidiscretisation
of the linear induction equation \eqref{eq:induction-transport-weak} with weak
implementation of the boundary conditions.
\begin{equation}
\label{eq:induction-transport-semidiscrete}
\begin{aligned}
  \partial_t B_i
  &=
    B_j D_j u_i
  - \frac{1}{2} B_i D_j u_j
  - \frac{1}{2} u_j D_j B_i
  - \frac{1}{2} D_j (u_j B_i)
  \\
  &\phantom{=} + M^{-1} E_j \bigl( \1{u \cdot \nu < 0} u_j (B_i - B^b_i) \bigr),
  && \text{for } t \in (0,T),
  \\
  B(0) &= B^0.
\end{aligned}
\end{equation}

\begin{remark}
  The surface term in \eqref{eq:induction-transport-semidiscrete} can also be written
  using numerical fluxes as in finite volume and discontinuous Galerkin methods.
  Indeed,
  \begin{equation}
    \1{u \cdot \nu < 0} u_j (B_i - B^b_i)
    =
    u_j B_i - u_j \Bnum_i,
  \end{equation}
  where $\Bnum$ is the upwind numerical flux, i.e. $\Bnum = B^b$ where $u \cdot \nu < 0$
  and $\Bnum = B$ elsewhere.
\end{remark}

The semidiscrete energy rate can be obtained analogously to the one in the continuous
setting. Indeed, the calculations leading to Lemma~\ref{lem:energy-rate-transport}
are mimicked as follows. The semidiscrete energy rate is
\begin{multline}
  \od{}{t} \norm{B}_M^2
  =
  2 B_i^T M \partial_t B_i
  =
    2 B_i^T M B_j D_j u_i
  - B_i^T M B_i D_j u_j
  - B_i^T M u_j D_j B_i
  - B_i^T M D_j (u_j B_i)
  \\
  + 2 B_i^T E_j \bigl( \1{u \cdot \nu < 0} u_j (B_i - B^b_i) \bigr).
\end{multline}
Since multiplication is performed componentwise and the mass matrix is diagonal,
\begin{equation}
  B_i^T M B_j D_j u_i = B_i^T B_j^T M D_j u_i,
\end{equation}
where $B_j = B_j^T$ is the diagonal multiplication matrix containing the coefficients
of $B_j$ on the diagonal. Thus, using the SBP property \eqref{eq:SBP},
\begin{multline}
  \od{}{t} \norm{B}_M^2
  =
    2 B_i^T M B_j D_j u_i
  - B_i^T M B_i D_j u_j
  - B_i^T M u_j D_j B_i
  - B_i^T E_j (u_j B_i)
  + B_i^T D_j^T M (u_j B_i)
  \\
  + 2 B_i^T E_j \bigl( \1{u \cdot \nu < 0} u_j (B_i - B^b_i) \bigr).
\end{multline}
Since $B_i^T M u_j D_j B_i = B_i^T D_j^T M (u_j B_i)$, this can be rewritten as
\begin{equation}
\label{eq:induction-transport-semidiscrete-energy-rate-1}
  \od{}{t} \norm{B}_M^2
  =
    2 B_i^T M B_j D_j u_i
  - B_i^T M B_i D_j u_j
  - B_i^T E_j (u_j B_i)
  + 2 B_i^T E_j \bigl( \1{u \cdot \nu < 0} u_j (B_i - B^b_i) \bigr).
\end{equation}
The first two terms on the right hand side mimic the volume terms
$\int_\Omega \left( 2 B_i B_j \partial_j u_i - B_i B_i \partial_j u_j \right)$
as in \eqref{eq:energy-rate-transport} and the other two terms mimic the surface
terms appearing for the weak implementation of the boundary condition. Thus,
an analogous estimate can be obtained. Indeed, rewriting the surface term as in
the proof of Lemma~\ref{lem:induction-transport-weak},
\begin{equation}
  \od{}{t} \norm{B}_M^2
  \leq
    2 \norm{D u(t)}_{\ell^\infty} \sum_{i,j} \abs{B_i}^T M \abs{B_j}
  + 3 \norm{D u(t)}_{\ell^\infty} \abs{B_i}^T M \abs{B_i}
  + \norm{u(t)}_{\ell^\infty} \abs{B^b_i}^T E \abs{B^b_i}.
\end{equation}
Here, the absolute value $\abs{B_i}$ should be considered componentwise and the
discrete $\ell^\infty$ norm is $\norm{D u(t)}_{\ell^\infty} = \max_{i,j}
\norm{D_j u_i}_{l^\infty}$.
Proceeding as in the proof of Lemma~\ref{lem:induction-transport-weak} results in
\begin{lemma}
\label{lem:induction-transport-semidiscrete}
  A sufficiently smooth solution $B$ of the semidiscrete linear induction equation
  \eqref{eq:induction-transport-semidiscrete} satisfies
  \begin{equation}
  \label{eq:induction-transport-semidiscrete-energy-rate}
    \od{}{t} \norm{B(t)}_M^2
    \leq
    9 \norm{D u(t)}_{\ell^\infty} \norm{B(t)}_M^2
    + \norm{u(t)}_{\ell^\infty} \norm{B^b(t)}_E^2
  \end{equation}
  and
  \begin{equation}
  \label{eq:induction-transport-semidiscrete-energy-estimate}
    \norm{B(t)}_{M}^2
    \leq
    \exp\bigl( 9 \norm{D u}_\infty t \bigr)
    \left(
      \norm{B^0}_M^2
      + \int_0^t \norm{u(t)}_{\ell^\infty}
        \norm{B^b(t)}_{E}^2 \dif t
    \right).
  \end{equation}
  Thus, this semidiscretisation is energy stable.
\end{lemma}

\begin{remark}
  If multiple blocks/elements $\Omega^l$ are used to discretise the total domain
  $\Omega$, these blocks have to be coupled. This coupling can be done via surface
  terms, analogously to the weak imposition of boundary conditions. Suppose that
  the particle velocity $u$ is discretised as a continuous function across the boundaries,
  which seems to be quite natural if $u$ is given, e.g. in a hybrid model. Then,
  the discrete values of $u \cdot \nu$ at a point on the boundary between two blocks
  $\Omega^{l_1}, \Omega^{l_2}$ satisfy $u^{l_1} \cdot \nu^{l_1} = - u^{l_2} \cdot \nu^{l_2}$,
  since $u^{l_1} = u^{l_2}$ and $\nu^{l_1} = -\nu^{l_2}$ because of opposite outward
  unit normals. Thus, a boundary condition has to be specified at one of the two
  blocks (if $u \neq 0$) or none of them (if $u=0$). Setting the desired boundary
  value $B^b$ to the value of $B$ from the other block corresponds to the application
  of the upwind numerical flux as in finite volume or discontinuous Galerkin methods.
  This coupling of multiple
  blocks is energy stable if conforming block interfaces (i.e. matching nodes) are
  used. Although central fluxes could be used as well to give an energy estimate,
  the application of upwind numerical fluxes yields additional stabilisation and
  improved properties concerning e.g. the numerical error, cf. \cite{nordstrom2007error,
  kopriva2017error, offner2018error2}.
\end{remark}

\subsection{Different Formulations and Implementation}
\label{sec:implementation}

Discretising the split form $\frac{1}{2} \bigl( B_i \partial_j u_j + u_j \partial_j B_i
+ \partial_j (u_j B_i) \bigr)$ instead of the conservative form $\partial_j(u_j B_i)$
might seem to be computationally expensive at first. However, the loops appearing in
the (block-banded) matrix vector multiplication can be fused, resulting in less additional
cost.

Another drawback that might be attributed to a split form discretisation concerns
weak solutions. If discontinuities appear in the solution, e.g. due to nonlinearities
if the MHD equations are discretised by an operator splitting approach or the particle
velocity is obtained via a particle simulation in a hybrid model, the discretisation
should be conservative in the light of the classical Lax-Wendroff theorem
\cite{lax1960systems}. However, split form discretisations such as $\frac{1}{2} \bigl(
B_i D_j u_j + u_j D_j B_i + D_j (u_j B_i) \bigr)$ can be written in a conservative
way if classical central differences are used in periodic domains or diagonal norm
SBP operators are used in bounded domains, cf. \cite{ducros2000high,
pirozzoli2010generalized, fisher2013discretely, fisher2013highJCP, gassner2016split}.
\begin{example}
  Consider the split-form discretisation $-\frac{1}{2} \bigl(B_i D_j u_j + u_j D_j B_i
  + D_j (u_j B_i) \bigr)$ using the second order SBP operator from Example~\ref{ex:SBP-2}
  in the interior. Using upper indices to indicate the grid nodes, the derivative
  in $x_1$ direction is
  \begin{multline}
    - \frac{1}{2} B_i^{(a,b,c)} \frac{u_1^{(a+1,b,c)} - u_1^{(a-1,b,c)}}{2 \Delta x_1}
    - \frac{1}{2} u_1^{(a,b,c)} \frac{B_i^{(a+1,b,c)} - B_i^{(a-1,b,c)}}{2 \Delta x_1}
    - \frac{1}{2} \frac{(u_1 B_i)^{(a+1,b,c)} - (u_1 B_i)^{(a-1,b,c)}}{2 \Delta x_1}
    \\
    =
    - \frac{1}{\Delta x_1} \!\left(\!
      \frac{\left( u_1^{(a,b,c)} + u_1^{(a+1,b,c)} \right)\!
            \left( B_i^{(a,b,c)} + B_i^{(a+1,b,c)} \right)}{4}
      -
      \frac{\left( u_1^{(a,b,c)} + u_1^{(a-1,b,c)} \right)\!
            \left( B_i^{(a,b,c)} + B_i^{(a-1,b,c)} \right)}{4}
    \!\right).
  \end{multline}
  This discretisation is conservative with numerical flux
  \begin{equation}
    f^{\mathrm{num},1}_{m,k}
    =
    \frac{1}{4}
    \left( u_1^{(m)} + u_1^{(k)} \right)
    \left( B_i^{(m)} + B_i^{(k)} \right),
  \end{equation}
  where $m,k$ represent the Cartesian indices $(a,b,c),(a \pm 1,b,c)$. The direct
  discretisation $-D_j(u_j B_i)$ of the conservative form can be written similarly as
  \begin{multline}
    - \frac{(u_1 B_i)^{(a+1,b,c)} - (u_1 B_i)^{(a-1,b,c)}}{2 \Delta x_1}
    \\
    =
    - \frac{1}{\Delta x_1} \left(
        \frac{u_1^{(a,b,c)} B_i^{(a,b,c)} + u_1^{(a+1,b,c)} B_i^{(a+1,b,c)}}{2}
        -
        \frac{u_1^{(a,b,c)} B_i^{(a,b,c)} + u_1^{(a-1,b,c)} B_i^{(a-1,b,c)}}{2}
    \right).
  \end{multline}
  The boundary terms can be handled similarly. Thus, both discretisations are
  conservative.
\end{example}
Using symmetric numerical fluxes, high-order conservative semidiscretisations
can be obtained for conservation laws, cf. \cite{fisher2013highJCP, chen2017entropy,
ranocha2017comparison}. For the discretisations considered here, the arithmetic
mean value
\begin{equation}
  \mean{v}_{m,k} = \frac{v^{(m)} + v^{(k)}}{2}
\end{equation}
suffices to obtain the central form $D(v w)$ (via $\mean{v w}$), the split
form $\frac{1}{2} \bigl(v D w + w D v + D(v w) \bigr)$ (via $\mean{v} \mean{w}$),
and the product form $v D w + w D v$ (via $2 \mean{v}\mean{w} - \mean{v w}$).

If a source term such as $- u \div B$ is added to the induction equation as in
\eqref{eq:induction-transport-split}, symmetric numerical fluxes do not suffice
anymore to represent the semidiscretisations. Then, extended numerical fluxes
containing non-symmetric terms can be used to describe the semidiscretisations in
a unified way, cf. \cite{ranocha2017shallow}, \cite[Section 4]{bouchut2004nonlinear},
and references cited therein. Therefore, not only the mean value but also the jump
\begin{equation}
  \jump{v}_{m,k} = v_k - v_m
\end{equation}
will be used.

The general form of the semidiscretisations considered here is
\begin{equation}
\label{eq:induction-equation-transport-semidiscrete-general}
  \partial_t B_i^{(m)}
  =
  \VOL_i^{(m)} + \SURF_i^{(m)},
\end{equation}
where $\VOL_i^{(m)}$ is a discretisation of the volume term, i.e. an approximation
of $\partial_j(u_i B_j - u_j B_i)$, and $\SURF_i^{(m)}$ is a surface term, i.e.
the SAT in \eqref{eq:induction-transport-semidiscrete} that is nonzero only at
the boundary nodes,
\begin{equation}
\label{eq:induction-equation-transport-semidiscrete-surface}
  \SURF
  =
  M^{-1} E_j \bigl( \1{u \cdot \nu < 0} u_j (B - B^b) \bigr).
\end{equation}
\begin{remark}
  The classical second order SBP operator \eqref{eq:SBP-2} has a special form
  only directly at the boundary nodes. Higher order SBP operators use more nodes
  near the boundary with modified stencil. Nevertheless, the surface term $\SURF$
  is nontrivial only directly at the boundaries.
\end{remark}
The general form of the volume term is
\begin{equation}
\label{eq:induction-equation-transport-semidiscrete-volume}
  \VOL^{(m)}
  =
  \sum_{j=1}^3 \sum_k 2 (D_j)_{m,k} \fextj_{m,k},
\end{equation}
where $\fextj_{m,k}$ is an extended numerical flux in space direction $j$.
The discretisation $B_j D_j u - \frac{1}{2} B D_j u_j - \frac{1}{2} u_j D_j B
- \frac{1}{2} D_j (u_j B)$ of \eqref{eq:induction-transport-semidiscrete} is
obtained by choosing
\begin{equation}
  \fextj_{m,k}
  =
  \underbrace{2 \mean{B_j}_{m,k} \mean{u}_{m,k} - \mean{B_j u}_{m,k}}
    _{= \frac{1}{2} \bigl( B_j^{(m)} u^{(k)} + B_j^{(k)} u^{(m)} \bigr) }
  \underbrace{- \frac{1}{2} u^{(m)} \jump{B_j}_{m,k}}
    _{= - \frac{1}{2} u^{(m)} \bigl( B_j^{(k)} - B_j^{(m)} \bigr) }
  \underbrace{- \mean{u_j}_{m,k} \mean{B}_{m,k}}
    _{\mathrlap{\!\!= - \frac{1}{4} \bigl( u_j^{(m)} + u_j^{(k)} \bigr)
                                    \bigl( B^{(m)} + B^{(k)} \bigr) }}.
\end{equation}
The first two terms generate the nonconservative form $B_j D_j u + u D_j B_j$,
the third term generates the source term $- u D_j B_j$, and the last term generates
the split discretisation $-\frac{1}{2} \bigl( B D_j u_j + u_j D_j B + D_j (u_j B) \bigr)$.
Indeed,
\begin{equation}
\begin{aligned}
  &\phantom{=}
  \VOL^{(m)}
  =
  \sum_{j=1}^3 \sum_k 2 (D_j)_{m,k} \fextj_{m,k}
  \\
  &=
  \sum_{j=1}^3 \sum_k (D_j)_{m,k} \left(
    B_j^{(m)} u^{(k)} + B_j^{(k)} u^{(m)} - u^{(m)} \bigl( B_j^{(k)} - B_j^{(m)} \bigr)
    - \frac{1}{2} \bigl( u_j^{(m)} + u_j^{(k)} \bigr) \bigl( B^{(m)} + B^{(k)} \bigr)
  \right)
  \\
  &=
  \left[
    B_j D_j u - \frac{1}{2} \bigl( B D_j u_j + u_j D_j B + D_j (u_j B) \bigr)
  \right]^{(m)},
\end{aligned}
\end{equation}
where $\sum_k (D_j)_{m,k} = 0$ has been used, since $D_j$ is a consistent approximation
of the derivative. The same result can also be obtained by another choice of the
extended numerical fluxes corresponding to $\partial_j(B_j u)$ and $- u \div B$.
Indeed, both terms can be discretised as split forms via
\begin{equation}
\label{eq:extended-fluxes-product-and-source}
  \underbrace{2 \mean{B_j}_{m,k} \mean{u}_{m,k} - \mean{B_j u}_{m,k}}_{\leadsto \partial_j(B_j u)}
  \underbrace{- \frac{1}{2} u^{(m)} \jump{B_j}_{m,k}}_{\leadsto - u \partial_j B_j}
  =
  \underbrace{\mean{B_j}_{m,k} \mean{u}_{m,k}}_{\leadsto \partial_j(B_j u)}
  \underbrace{- \frac{1}{2} \mean{u}_{m,k} \jump{B_j}_{m,k}}_{\leadsto - u \partial_j B_j}.
\end{equation}

Thus, there are some obvious possibilities to discretise the volume terms of the
linear induction equation $\partial_t B_i = \partial_j (u_i B_j - u_j B_i)$,
possibly augmented with source term $-u_i \div B$, listed in Table~\ref{tab:uiBj-forms},
Table~\ref{tab:source-forms}, and Table~\ref{tab:ujBi-forms}. Besides the choice
of adding a source term or not, the forms are equivalent at the continuous level
for smooth functions due to the product rule. However, a discrete product rule is
impossible for general high order discretisations, cf. \cite{ranocha2018mimetic}.

\begin{table}[!ht]
\centering
\begin{small}
\addtolength\tabcolsep{-1.5pt}
  \caption{Different discretisations of the term $\partial_j(u_i B_j)$
           in the linear induction equation \eqref{eq:induction-transport-split}.}
  \label{tab:uiBj-forms}
  \begin{tabular*}{\linewidth}{@{\extracolsep{\fill}}*3c@{}}
    \toprule
    Form & Discretisation & Extended Numerical Flux
    \\
    \midrule
    Central
    & $D_j (u_i B_j)$
    & $\mean{u_i B_j}_{m,k}
      = \frac{1}{2} \Bigl( u_i^{(m)} B_j^{(m)} + u_i^{(k)} B_j^{(k)} \Bigr)$
    \\
    Split
    & $\frac{1}{2} \bigl( D_j(u_i B_j) + u_i D_j B_j + B_j D_j u_i \bigr)$
    & $\mean{u_i}_{m,k} \mean{B_j}_{m,k}
      = \frac{1}{4} \Bigl( u_i^{(m)} + u_i^{(k)} \Bigr) \Bigl( B_j^{(m)} + B_j^{(k)} \Bigr)$
    \\
    Product
    & $u_i D_j B_j + B_j D_j u_i$
    & $2 \mean{u_i}_{m,k} \mean{B_j}_{m,k} - \mean{u_i B_j}_{m,k}
      = \frac{1}{2} \Bigl( u_i^{(m)} B_j^{(k)} + u_i^{(k)} B_j^{(m)} \Bigr)$
    \\
    \bottomrule
  \end{tabular*}

  \bigskip

  \caption{Different discretisations of the source term $- u_i \partial_j B_j$
           for the linear induction equation \eqref{eq:induction-transport-split}.}
  \label{tab:source-forms}
  \begin{tabular*}{\linewidth}{@{\extracolsep{\fill}}*3c@{}}
    \toprule
    Form & Discretisation & Extended Numerical Flux
    \\
    \midrule
    Zero
    & $0$
    & $0$
    \\
    Central
    & $- u_i D_j B_j$
    & $- \frac{1}{2} u_i^{(m)} \jump{B_j}_{m,k}
      =  - \frac{1}{2} u_i^{(m)} \Bigl( B_j^{(k)} - B_j^{(m)} \Bigr)$
    \\
    Split
    & $- \frac{1}{2} \bigl(u_i D_j B_j + D_j (u_i B_j) - B_j D_j u_i \bigr)$
    & $- \frac{1}{2} \mean{u_i}_{m,k} \jump{B}_{m,k}
      = - \frac{1}{4} \Bigl( u_i^{(m)} + u_i^{(k)} \Bigr) \Bigl( B_j^{(k)} - B_j^{(m)} \Bigr)$
    \\
    \bottomrule
  \end{tabular*}

  \bigskip

  \caption{Different discretisations of the term $- \partial_j(u_j B_i)$
           in the linear induction equation \eqref{eq:induction-transport-split}.}
  \label{tab:ujBi-forms}
  \begin{tabular*}{\linewidth}{@{\extracolsep{\fill}}*3c@{}}
    \toprule
    Form & Discretisation & Extended Numerical Flux
    \\
    \midrule
    Central
    & $- D_j (u_j B_i)$
    & $- \mean{u_j B_i}_{m,k}
      = - \frac{1}{2} \Bigl( u_j^{(m)} B_i^{(m)} + u_j^{(k)} B_i^{(k)} \Bigr)$
    \\
    Split
    & $- \frac{1}{2} \bigl( D_j(u_j B_i) + u_j D_j B_i + B_i D_j u_j \bigr)$
    & $- \mean{u_j}_{m,k} \mean{B_i}_{m,k}
      = - \frac{1}{4} \Bigl( u_j^{(m)} + u_j^{(k)} \Bigr) \Bigl( B_i^{(m)} + B_i^{(k)} \Bigr)$
    \\
    Product
    & $- \bigl( u_j D_j B_i + B_i D_j u_j \bigr)$
    & $- 2 \mean{u_j}_{m,k} \mean{B_i}_{m,k} + \mean{u_j B_i}_{m,k}
      = - \frac{1}{2} \Bigl( u_j^{(m)} B_i^{(k)} + u_j^{(k)} B_i^{(m)} \Bigr)$
    \\
    \bottomrule
  \end{tabular*}
\end{small}
\end{table}

\begin{remark}
  In the current article, the split form \eqref{eq:induction-transport-split} is
  used to motivate energy estimates and the consideration of different forms of
  the induction equation. Other applications of split forms can be found e.g. in \cite{olsson1994energy, nordstrom2006conservative}.
\end{remark}

\begin{remark}
  Several different split forms and source terms of the ideal MHD equations have
  been compared numerically in \cite{sjogreen2017skew}. If present, the source
  term $- u \div B$ has been discretised via the central extended flux. Different
  numerical fluxes have been used for the other terms.
\end{remark}

\begin{remark}
  Entropy conservative numerical fluxes for the ideal MHD equations can be found
  in \cite{winters2016affordable, chandrashekar2016entropy, sjogreen2018high}.
  They contain products of some averages but with additional terms compared to
  the split form fluxes for the induction equation.
\end{remark}

\begin{remark}
  Energy stability of semidiscretisations can be transferred to fully discrete
  schemes if implicit time integrators with the SBP property are used, cf.
  \cite{nordstrom2013summation, lundquist2014sbp, boom2015high, nordstrom2018energy}.
  In this article, explicit time integration schemes will be used, since they
  can be implemented efficiently and easily on modern HPC hardware such as GPUs.
  For linear problems with semibounded operators, such explicit schemes can also
  be shown to be energy stable, cf. \cite{tadmor2002semidiscrete, ranocha2018L2stability}.
\end{remark}

\subsection{Energy Stability of Other Semidiscretisations}
\label{sec:energy-stability-of-other-forms}

In \cite{mishra2010stability}, stable schemes for the linear magnetic induction
equation \eqref{eq:induction-transport-conservative} have been derived by applying
the principle of frozen coefficients to the conservative form of $-\partial_j(u_j B_i)$.
Thus, the central flux has been used for $-\partial_j (u_j B_i)$ and the
discretisation of the other terms corresponds to the choice of the extended fluxes
as in \eqref{eq:extended-fluxes-product-and-source}.
In that article, the equivalence of strong stability for semidiscretisations of
linear symmetric hyperbolic systems using the conservative and the product form
has been proven. Thus, also the application of the product flux for $-\partial_j(u_j B_i)$
yields an energy estimate.

In order to make this article sufficiently self-contained, a brief description
of the approach to energy estimates for the other forms is given in the following
for diagonal mass matrices $M$. Then, discrete integration and multiplication
commute. The key is the following result of \cite[section~2]{mishra2010stability}.
\begin{lemma}
\label{lem:product-rule-error}
  For every discrete derivative operator $D$ with diagonal mass matrix $M$, there
  is a constant $C > 0$ such that for every smooth function $v$ and any grid
  function $w$
  \begin{equation}
    \norm{D(v w) - v D w}_M \leq C \norm{\partial_x v}_{L^\infty} \norm{w}_M.
  \end{equation}
\end{lemma}

Applying the energy method to the conservative form of $-\partial_j(u_j B_i)$ yields
\begin{equation}
  - 2 B_i^T M D_j (u_j B_i)
  =
  - B_i^T E_j u_j B_i
  + B_i^T D_j^T M u_j B_i
  - B_i^T M D_j (u_j B_i).
\end{equation}
The surface term is the same as for the split form discretisation, cf.
\eqref{eq:induction-transport-semidiscrete-energy-rate-1}. The remaining volume
terms satisfy
\begin{equation}
  \abs{B_i^T D_j^T M u_j B_i - B_i^T M D_j (u_j B_i)}
  \leq
  \norm{B_i}_M \norm{u_j D_j B_i - D_j (u_j B_i)}_M
  \leq
  C \norm{\nabla u}_{L^\infty} \norm{B}_M^2
\end{equation}
for some constant $C$ due to Lemma~\ref{lem:product-rule-error}. Thus, an energy
estimate can be obtained. Similarly, applying the energy method to the product
form discretisation of $-\partial_j(u_j B_i)$ yields
\begin{equation}
\begin{aligned}
  - 2 B_i^T M u_j D_j B_i - 2 B_i^T M B_i D_j u_j
  =
  - B_i^T E_j u_j B_i
  + B_i^T u_j^T D_j^T M B_i
  - B_i^T M u_j D_j B_i
  - 2 B_i^T M B_i D_j u_j.
\end{aligned}
\end{equation}
The surface term is again the same as for the split form discretisation and
the remaining terms can be estimated using Lemma~\ref{lem:product-rule-error},
since
\begin{equation}
  \abs{B_i^T u_j^T D_j^T M B_i - B_i^T M u_j D_j B_i}_M
  \leq
  \norm{B_i}_M \norm{D_j (u_j B_i) - u_j D_j B_i}_M
  \leq
  C \norm{\nabla u}_{L^\infty} \norm{B}_M^2
\end{equation}
and
\begin{equation}
  \abs{- 2 B_i^T M B_i D_j u_j}
  \leq
  6 \norm{D u}_{\ell^\infty} \norm{B}_M^2.
\end{equation}
Thus, an energy estimate can be obtained.

It is also possible to combine the central discretisation of the source term
$- u_i \partial_j B_j$ with other forms of $\partial_j (u_i B_j)$. Indeed, applying
the energy method to the source term and the central discretisation of
$\partial_j (u_i B_j)$ yields a volume term that can be estimated as
\begin{equation}
  \abs{2 B_i^T M D_j (u_i B_j) - 2 B_i^T M u_i D_j B_j}
  \leq
  2 \norm{B}_M \norm{D_j (u_i B_j) - u_i D_j B_j}_M
  \leq
  C \norm{\nabla u}_{L^\infty} \norm{B}_M^2
\end{equation}
for some constant $C$ due to Lemma~\ref{lem:product-rule-error}. This can be
compared to the corresponding upper bound $6 \norm{D u}_{\ell^\infty} \norm{B}_M^2$
appearing in the proof of Lemma~\ref{lem:induction-transport-semidiscrete}.
Similarly, applying the split form discretisation of $\partial_j (u_i B_j)$
results in
\begin{equation}
\begin{aligned}
  &\quad\,
  \abs{B_i^T M D_j (u_i B_j) + B_i^T M u_i D_j B_j + B_i^T M B_j D_j u_i
  - 2 B_i^T M u_i D_j B_j}
  \\
  &\leq
  \abs{B_i^T M D_j (u_i B_j) - B_i^T M u_i D_j B_j}
  + \abs{B_i^T M B_j D_j u_i}
  \\
  &\leq
  C \norm{\nabla u}_{L^\infty} \norm{B}_M^2
  + 3 \norm{D u}_{\ell^\infty} \norm{B}_M^2,
\end{aligned}
\end{equation}
for some $C > 0$ due to Lemma~\ref{lem:product-rule-error} and an energy estimate
can be obtained.

Finally, the split form discretisation of the source term can also be used to
obtain an energy estimate. This is summed up in
\begin{proposition}
\label{prop:induction-transport-semidiscrete}
  The semidiscretisations \eqref{eq:induction-equation-transport-semidiscrete-general}
  of the linear induction equation \eqref{eq:induction-transport-conservative}
  using the surface terms \eqref{eq:induction-equation-transport-semidiscrete-surface}
  as SATs and the volume terms \eqref{eq:induction-equation-transport-semidiscrete-volume},
  where the extended numerical fluxes are given by any combination of terms in
  Tables~\ref{tab:uiBj-forms}, \ref{tab:source-forms}, and~\ref{tab:ujBi-forms}
  with non-zero source terms, lead to an energy estimate.
\end{proposition}

Although there are energy estimates for various types of schemes, the behaviour
of the solutions and the numerical error can be different for fixed grids. Thus,
this will be investigated in numerical experiments in section~\ref{sec:numerical-results}.

\subsection{Bounds on the Divergence}
\label{sec:bounds-on-divB}

Another motivation for adding the source term $-u \div B$ to the linear induction
equation \eqref{eq:induction-transport-conservative} is given by the following
well-known observation. Taking the divergence of the resulting PDE with source
term yields
\begin{equation}
\label{eq:transport-eq-divB}
  \partial_t \div B
  =
  \div ( \nabla \times (u \times B) - u \div B )
  =
  - \div (u \div B).
\end{equation}
Thus, the divergence of $B$ satisfies a linear transport equation with velocity
$u$ and divergence errors can possibly be transported out of the domain.

However, boundary conditions are important for the preservation of the divergence
constraint. While no detailed investigation will be conducted here, a simple example
is given in the following. Based thereon, it might seem to be questionable to
obtain bounds on the (discrete) divergence of $B$ using only bounds on the velocity
$u$ and the magnetic field itself.

\begin{example}
  Consider the velocity $u(x,y,z) = (1,0,0)$ and the initial condition $B^0 \equiv 0$
  in the domain $[0,\pi] \times [0,1]^2$. Then, a boundary condition has to be
  specified exactly at the $x=0$ boundary. Choose the boundary data $B^b_i(t) =
  \delta_{i1} \sin(n t)$
  for $n \in \N$. The solution of the IBVP for the linear induction equation
  with source term, $\partial_t B = \nabla \times (u \times B) - u \div B$,
  is given by
  \begin{equation}
    B_2 \equiv 0 \equiv B_3,
    \quad
    B_1(t,x,y,z)
    =
    \begin{cases}
      0, & x-t > 0, \\
      \sin\bigl( n (t-x) \bigr), & \text{else}.
    \end{cases}
  \end{equation}
  For $t > \pi$, the solution satisfies
  \begin{equation}
    \norm{B(t)}_{L^2(\Omega)}^2
    =
    \int_0^\pi \sin\bigl( n (t-x) \bigr)^2 \dif x
    =
    \frac{\pi}{2}
  \end{equation}
  and
  \begin{equation}
    \norm{\div B(t)}_{L^2(\Omega)}^2
    =
    \int_0^\pi \bigl( - n \cos\bigl( n (t-x) \bigr) \bigr)^2 \dif x
    =
    n^2 \frac{\pi}{2}
    \to \infty, \quad n \to \infty.
  \end{equation}
  Thus, there is a sequence of solutions with uniformly bounded norms and unbounded
  norms of the divergence, even if only the interior of $\Omega$ is considered
  for the latter.
\end{example}

\section{Nonlinear Magnetic Induction Equation}
\label{sec:Hall-term}

In this section, the nonlinear Hall magnetic induction equation
\begin{equation}
\label{eq:induction-transport-conservative-Hall}
  \partial_t B
  =
  \nabla \times (u \times B)
  - \nabla \times \biggl( \frac{\nabla \times B}{\rho} \times B \biggr)
\end{equation}
with divergence constraint $\div B = 0$ and suitable initial and boundary
conditions will be investigated, following the same principle ideas as in the
previous section. However, this problem is more complicated due to the nonlinear
second derivatives. Using the results of section~\ref{sec:transport-term}, a
source term $- u \div B$ is added to the right hand side and a splitting is used.
This yields
\begin{equation}
\label{eq:induction-transport-split-Hall}
  \partial_t B
  =
  (B \cdot \nabla) u - \frac{1}{2} \left( B (\div u) + (u \cdot \nabla) B + \div(B \otimes u) \right)
  - \nabla \times \biggl( \frac{\nabla \times B}{\rho} \times B \biggr).
\end{equation}
The investigation in this section follows basically the outline given in
\cite{nordstrom2017roadmap, nordstrom2018energy}.

\subsection{Continuous Setting}

Using the results section~\ref{sec:transport-term}, the transport term can be
handled similarly, i.e. the product rule can be used and a source term $- u \div B$
can be added to formulate the linear part in a way allowing to estimate the energy
rate. Hence, the nonlinear term has to be considered next. For a sufficiently
smooth solution, setting $A := \frac{1}{\rho} (\nabla \times B) \times B$, the
contribution of the Hall term to the energy rate can be calculated via
\begin{equation}
\label{eq:int-B-dot-curl-A}
\begin{aligned}
  - \int_\Omega B \cdot (\nabla \times A)
  =
  - \int_\Omega B_i \epsilon_{ijk} \partial_j A_k
  &=
  \int_\Omega \epsilon_{ijk} A_k \partial_j B_i
  - \int_{\partial\Omega} \epsilon_{ijk} B_i \nu_j A_k
  \\
  &=
  - \int_\Omega A \cdot (\nabla \times B)
  - \int_{\partial\Omega} (A \times B)  \cdot \nu.
\end{aligned}
\end{equation}
Here, $A \cdot (\nabla \times B) = 0$, since $A = \frac{1}{\rho} (\nabla \times B) \times B$.
Hence, the Hall term is conservative with respect to the magnetic energy and
yields the surface term
\begin{equation}
  - \int_{\partial\Omega}
    \biggl( \biggl(\frac{\nabla \times B}{\rho} \times B \biggr) \times B \biggr)
    \cdot \nu
  =
  - \int_{\partial\Omega} \left(
      \biggl( B \cdot \frac{\nabla \times B}{\rho} \biggr) B
      - \abs{B}^2 \frac{\nabla \times B}{\rho}
  \right) \cdot \nu.
\end{equation}
This term has to be added to the surface term $- \int_{\partial\Omega} \frac{1}{2}
B_i B_i u_j \nu_j$ of the linear induction equation, cf. \eqref{eq:energy-rate-transport}.
Thus, a smooth solution $B$ of \eqref{eq:induction-transport-split-Hall} satisfies
\begin{multline}
  \frac{1}{2} \od{}{t} \norm{B}_{L^2(\Omega)}^2
  =
  \int_\Omega B \cdot \partial_t B
  =
  \int_\Omega \left( B_i B_j \partial_j u_i - \frac{1}{2} B_i B_i \partial_j u_j \right)
  \\
  - \int_{\partial\Omega} \left(
      \frac{1}{2} \abs{B}^2 u
      + \biggl( B \cdot \frac{\nabla \times B}{\rho} \biggr) B
      - \abs{B}^2 \frac{\nabla \times B}{\rho}
  \right) \cdot \nu.
\end{multline}
The integrand of the surface term can also be written using $\I_3 = \diag{1,1,1}$ as
\begin{equation}
\label{eq:integrand-quadratic-form}
  \left(
    \frac{1}{2} \abs{B}^2 u
    + \biggl( B \cdot \frac{\nabla \times B}{\rho} \biggr) B
    - \abs{B}^2 \frac{\nabla \times B}{\rho}
  \right) \cdot \nu
  =
  \vect{B \\ \frac{\nabla \times B}{\rho}}^T
  \begin{pmatrix}
    \bigl( \frac{u}{2} - \frac{\nabla \times B}{\rho} \bigr) \cdot \nu \I_3
    & \frac{1}{2} B \cdot \nu \I_3 \\
    \frac{1}{2} B \cdot \nu \I_3 & 0
  \end{pmatrix}
  \vect{B \\ \frac{\nabla \times B}{\rho}}.
\end{equation}
This is a quadratic form with coefficients depending on the solution itself,
contrary to linear equations \cite{nordstrom2017roadmap}. However, the matrix is
still symmetric and therefore diagonalisable. Here, the eigenvalues are
\begin{equation}
  \lambda_\pm
  =
  \frac{1}{2} \left(
    \Bigl( \frac{u}{2} - \frac{\nabla \times B}{\rho} \Bigr) \cdot \nu
    \pm \sqrt{
      \biggl( \Bigl( \frac{u}{2} - \frac{\nabla \times B}{\rho} \Bigr) \cdot \nu \biggr)^2
      + (B \cdot \nu)^2 }
  \right),
\end{equation}
and the corresponding eigenvectors are given by
\begin{equation}
  v_\pm^1 = \vect{\lambda_\pm, 0, 0, \frac{B \cdot \nu}{2}, 0, 0}^T,
  \quad
  v_\pm^2 = \vect{0, \lambda_\pm, 0, 0, \frac{B \cdot \nu}{2}, 0}^T,
  \quad
  v_\pm^3 = \vect{0, 0, \lambda_\pm , 0, 0, \frac{B \cdot \nu}{2}}^T,
\end{equation}
if $B \cdot \nu \neq 0$. Three different cases can occur:
\begin{enumerate}
  \item $B \cdot \nu \neq 0$.\\
  In this case, each of the eigenvalues $\lambda_+ > 0$ and $\lambda_- < 0$ has
  geometric multiplicity three.

  \item $B \cdot \nu = 0$ and $\bigl( \frac{u}{2} - \frac{\nabla \times B}{\rho} \bigr) \cdot \nu \neq 0$.\\
  In this case, there are the threefold eigenvalues zero and
  $\bigl( \frac{u}{2} - \frac{\nabla \times B}{\rho} \bigr) \cdot \nu \neq 0$.

  \item $B \cdot \nu = 0$ and $\bigl( \frac{u}{2} - \frac{\nabla \times B}{\rho} \bigr) \cdot \nu = 0$.\\
  In this case, the matrix occurring in \eqref{eq:integrand-quadratic-form} is
  simply zero.
\end{enumerate}
Thus, depending on the number of negative eigenvalues, it can be expected that three
(case~1 or case~2 with $\bigl( \frac{u}{2} - \frac{\nabla \times B}{\rho} \bigr) \cdot \nu < 0$)
or zero (otherwise) boundary conditions can be imposed, cf. \cite{nordstrom2017roadmap,
nordstrom2018energy}.

In order to determine admissible forms of the boundary conditions, the integrand
\eqref{eq:integrand-quadratic-form} is rewritten by diagonalising the symmetric
matrix using $V := (v_-^1, v_-^2, v_-^3, v_+^1, v_+^2, v_+^3)$ as
\begin{equation}
  \vect{B \\ \frac{\nabla \times B}{\rho}}^T
  \begin{pmatrix}
    \bigl( \frac{u}{2} - \frac{\nabla \times B}{\rho} \bigr) \cdot \nu \I_3
    & \frac{1}{2} B \cdot \nu \I_3 \\
    \frac{1}{2} B \cdot \nu \I_3 & 0
  \end{pmatrix}
  \vect{B \\ \frac{\nabla \times B}{\rho}}
  =
  \vect{B \\ \frac{\nabla \times B}{\rho}}^T
  V
  \begin{pmatrix}
    \frac{\lambda_-}{\abs{v_-}^2} \I_3 & 0 \\
    0 & \frac{\lambda_+}{\abs{v_+}^2} \I_3
  \end{pmatrix}
  V^T
  \vect{B \\ \frac{\nabla \times B}{\rho}},
\end{equation}
where
\begin{equation}
  \abs{v_\pm}^2
  =
  \abs{v_\pm^i}^2
  =
  \lambda_\pm^2 + \frac{(B \cdot \nu)^2}{4}.
\end{equation}
Now, possible forms of boundary conditions can be determined using the characteristic
variables
\begin{equation}
\label{eq:characteristic-variable}
  V^T \vect{B \\ \frac{\nabla \times B}{\rho}}
  =
  \vect{
    \lambda_- B + \frac{B \cdot \nu}{2} \frac{\nabla \times B}{\rho}
    \\
    \lambda_+ B + \frac{B \cdot \nu}{2} \frac{\nabla \times B}{\rho}
  }.
\end{equation}
The general form of boundary conditions used also in \cite{nordstrom2017roadmap,
nordstrom2018energy} is $W_- = R W_+ + g$, where $W_-$ are the incoming variables
(corresponding to negative eigenvalues), $W_+$ the outgoing ones (corresponding
to positive eigenvalues), and $g$ are boundary data. Thus, as for linear hyperbolic
equations, the incoming variables are specified via the outgoing variables (and
an operator $R$) and boundary data $g$. Depending on the solution, there might be
no incoming or outgoing variables since the eigenvalues $\lambda_\pm$ can be zero.
However, if $B \cdot \nu \neq 0$, this general form of boundary conditions is
\begin{equation}
\label{eq:form-of-BCs}
  \underbrace{\lambda_- B + \frac{B \cdot \nu}{2} \frac{\nabla \times B}{\rho}}
    _{W_-,\, \text{incoming}}
  =
  R \underbrace{\left( \lambda_+ B + \frac{B \cdot \nu}{2} \frac{\nabla \times B}{\rho} \right)}
    _{W_+,\, \text{outgoing}}
  + g.
\end{equation}

The following general result has been obtained in \cite[section~2.3]{nordstrom2018energy}.
\begin{proposition}
\label{prop:Nordstrom-LaCognata}
  Suppose that the energy method can be applied to a given initial boundary value
  problem and yields volume terms that can be estimated and the surface term
  \begin{equation}
  \label{eq:Nordstrom-LaCognata-surface-term}
    - \int_{\partial\Omega}
      \vect{W_+ \\ W_-}^T
      \begin{pmatrix}
        \Lambda_+ & 0 \\
        0 & \Lambda_-
      \end{pmatrix}
      \vect{W_+ \\ W_-},
  \end{equation}
  where $\Lambda_\pm$ is a diagonal matrix with only positive/negative eigenvalues
  and $W_\pm$ are the outgoing/incoming variables. The boundary condition
  \begin{equation}
  \label{eq:Nordstrom-LaCognata-BC}
    W_- = R W_+ + g
  \end{equation}
  bounds the surface term \eqref{eq:Nordstrom-LaCognata-surface-term}, if
  \begin{enumerate}[label=\alph*)]
    \item
    the boundary condition \eqref{eq:Nordstrom-LaCognata-BC} is implemented strongly,
    \begin{equation}
    \label{eq:Nordstrom-LaCognata-nonhomogeneous-1}
      \Lambda_+ + R^T \Lambda_- R > 0,
    \end{equation}
    and there is a positive semi-definite matrix $\Gamma$ such that
    \begin{equation}
    \label{eq:Nordstrom-LaCognata-nonhomogeneous-2}
      - \Lambda_- + (\Lambda_- R) (\Lambda_+ + R^T \Lambda_- R)^{-1} (\Lambda_- R)^T
      \leq
      \Gamma
      <
      \infty.
    \end{equation}

    \item
    the boundary condition \eqref{eq:Nordstrom-LaCognata-BC} is implemented strongly,
    \begin{equation}
    \label{eq:Nordstrom-LaCognata-homogeneous}
      \Lambda_+ + R^T \Lambda_- R \geq 0,
    \end{equation}
    and homogeneous boundary data $g = 0$ are used.
  \end{enumerate}
  The same is true for a weak implementation of the boundary conditions using
  a penalty term described in \cite[section~2.3.2]{nordstrom2018energy}.
\end{proposition}

\subsection{Outflow Boundary Conditions}

Stable (neutral) outflow boundary conditions or ``do nothing'' boundary conditions
will be important for the envisioned use cases.
Inspired by results of \cite{dong2014robust} for the incompressible Navier-Stokes
equations, the following outflow boundary conditions for the magnetic induction
equation \eqref{eq:induction-transport-split-Hall} are proposed:
\begin{equation}
\label{eq:outflow-BCs-strong}
  \1{\bigl( \frac{u}{2} - \frac{\nabla \times B}{\rho} \bigr) \cdot \nu < 0}
  \biggl(\biggl( \frac{u}{2} - \frac{\nabla \times B}{\rho} \biggr) \cdot \nu\biggr) B
  + (B \cdot \nu) \frac{\nabla \times B}{\rho}
  =
  0.
\end{equation}
For the corresponding weak implementation, the following term is added to the PDE
\begin{equation}
\label{eq:outflow-BCs-weak}
  + L \left(
  \1{\bigl( \frac{u}{2} - \frac{\nabla \times B}{\rho} \bigr) \cdot \nu < 0}
  \biggl(\biggl( \frac{u}{2} - \frac{\nabla \times B}{\rho} \biggr) \cdot \nu\biggr) B
  + (B \cdot \nu) \frac{\nabla \times B}{\rho}
  \right).
\end{equation}
Thus, applying the energy method to \eqref{eq:induction-transport-split-Hall}
yields the volume terms of \eqref{eq:energy-rate-transport} and the surface term
\begin{equation}
\begin{aligned}
  &\phantom{=}
  - \int_{\partial\Omega} \left(
    \biggl(\biggl( \frac{u}{2} - \frac{\nabla \times B}{\rho} \biggr) \cdot \nu\biggr)
      \abs{B}^2
    + (B \cdot \nu) \frac{\nabla \times B}{\rho} \cdot B
  \right)
  \\&\phantom{=}
  + \int_{\partial\Omega} \left(
    \1{\bigl( \frac{u}{2} - \frac{\nabla \times B}{\rho} \bigr) \cdot \nu < 0}
    \biggl(\biggl( \frac{u}{2} - \frac{\nabla \times B}{\rho} \biggr) \cdot \nu\biggr) B
    + (B \cdot \nu) \frac{\nabla \times B}{\rho}
  \right) \cdot B
  \\
  &=
  - \int_{\partial\Omega}
    \biggl(\biggl( \frac{u}{2} - \frac{\nabla \times B}{\rho} \biggr) \cdot \nu\biggr)
    \abs{B}^2
    \biggl( 1 - \1{\bigl( \frac{u}{2} - \frac{\nabla \times B}{\rho} \bigr) \cdot \nu < 0} \biggr)
  \leq
  0.
\end{aligned}
\end{equation}
Hence, an energy estimate can be obtained. The different cases listed above will
be considered separately in the following with respect to the form and number of
boundary conditions.

\begin{remark}
\label{rem:hall_outflow_u_vs_u2}
  The appearance of $u/2 - (\nabla \times B)/\rho$ instead of $u - (\nabla \times B)/\rho$
  in the boundary condition \eqref{eq:outflow-BCs-strong} might seem to be
  irritating based on physical intuition at first, since the associated transport
  velocity for the magnetic field uses $u$ instead of $u/2$. However, these terms
  arise at the boundary using the energy method. It is not clear whether an energy
  estimate can be obtained using $u$ instead of $u/2$. Moreover, associated numerical
  methods behave differently, cf. section~\ref{sec:hall_outflow}.
\end{remark}

\subsubsection{Case~1: \texorpdfstring{$B \cdot \nu \neq 0$}{B⋅ν≠0} and \texorpdfstring{$\bigl( \frac{u}{2} - \frac{\nabla \times B}{\rho} \bigr) \cdot \nu < 0$}{(u/2 - ∇×B/ϱ)⋅ν<0}}

In this case, there are three incoming and three outgoing variables and the boundary
condition \eqref{eq:outflow-BCs-strong} can be written as
\begin{equation}
  \biggl(\biggl( \frac{u}{2} - \frac{\nabla \times B}{\rho} \biggr) \cdot \nu\biggr) B
  + (B \cdot \nu) \frac{\nabla \times B}{\rho}
  =
  W_- + W_+
  =
  0.
\end{equation}
Thus, the expected number of boundary conditions is imposed and given in the form
\eqref{eq:form-of-BCs} with $R = - \I_3$ and $g = 0$. The surface term resulting
from the energy method, i.e. from computing $\int_\Omega B \cdot \partial_t B$,
becomes
\begin{equation}
  - \int_{\partial\Omega} \left(
    \biggl(\biggl( \frac{u}{2} - \frac{\nabla \times B}{\rho} \biggr) \cdot \nu\biggr) \abs{B}^2
    + (B \cdot \nu) \frac{\nabla \times B}{\rho} \cdot B
  \right)
  =
  0
\end{equation}
for the strong implementation \eqref{eq:outflow-BCs-strong}. Analogously, the
resulting surface term using the weak implementation \eqref{eq:outflow-BCs-weak}
is also zero.

\subsubsection{Case~2: \texorpdfstring{$B \cdot \nu \neq 0$}{B⋅ν≠0} and \texorpdfstring{$\bigl( \frac{u}{2} - \frac{\nabla \times B}{\rho} \bigr) \cdot \nu \geq 0$}{(u/2 - ∇×B/ϱ)⋅ν≥0}}

Again, there are three incoming and outgoing variables. The boundary condition
\eqref{eq:outflow-BCs-strong} fulfils
\begin{equation}
\begin{aligned}
  (B \cdot \nu) \frac{\nabla \times B}{\rho} = 0
  &\iff
  \biggl( \lambda_- B + \frac{B \cdot \nu}{2} \frac{\nabla \times B}{\rho} \biggr)
  =
  \frac{\lambda_-}{\lambda_+}
  \biggl( \lambda_+ B + \frac{B \cdot \nu}{2} \frac{\nabla \times B}{\rho} \biggr)
  \\
  &\iff
  W_- = \frac{\lambda_-}{\lambda_+} W_+.
\end{aligned}
\end{equation}
The expected number of boundary conditions is imposed in the form \eqref{eq:form-of-BCs}
with $R = \frac{\lambda_-}{\lambda_+}$ and $g = 0$. The surface term resulting
from the energy method is
\begin{equation}
  - \int_{\partial\Omega} \left(
    \biggl(\biggl( \frac{u}{2} - \frac{\nabla \times B}{\rho} \biggr) \cdot \nu\biggr) \abs{B}^2
    + (B \cdot \nu) \frac{\nabla \times B}{\rho} \cdot B
  \right)
  =
  - \int_{\partial\Omega} \left(
    \biggl(\biggl( \frac{u}{2} - \frac{\nabla \times B}{\rho} \biggr) \cdot \nu\biggr) \abs{B}^2
  \right)
  \leq
  0
\end{equation}
for the strong implementation \eqref{eq:outflow-BCs-strong} and similarly for the
weak implementation \eqref{eq:outflow-BCs-weak}.

\subsubsection{Case~3: \texorpdfstring{$B \cdot \nu = 0$}{B⋅ν=0} and \texorpdfstring{$\bigl( \frac{u}{2} - \frac{\nabla \times B}{\rho} \bigr) \cdot \nu < 0$}{(u/2 - ∇×B/ϱ)⋅ν<0}}

In this case, the boundary condition \eqref{eq:outflow-BCs-strong} becomes
\begin{equation}
  \biggl(\biggl( \frac{u}{2} - \frac{\nabla \times B}{\rho} \biggr) \cdot \nu\biggr) B
  =
  W_-
  =
  0.
\end{equation}
Since $\lambda_+ = 0$, this is of the expected form for homogeneous boundary data
$g = 0$ and no outgoing variables $W_+$. As in Case~1, the surface term arising
from the energy method is zero for both implementations.

\subsubsection{Case~4: \texorpdfstring{$B \cdot \nu = 0$}{B⋅ν=0} and \texorpdfstring{$\bigl( \frac{u}{2} - \frac{\nabla \times B}{\rho} \bigr) \cdot \nu \geq 0$}{(u/2 - ∇×B/ϱ)⋅ν≥0}}

Now, the boundary condition \eqref{eq:outflow-BCs-strong} is simply $0 = 0$,
which is the only expected form, since there are no incoming variables ($\lambda_-=0$,
$\lambda_+ > 0$). Clearly, the surface term arising from the energy method is
non-positive.

Together with the results of Lemma~\ref{lem:induction-transport-strong} and
Lemma~\ref{lem:induction-transport-weak}, this is is summed up in
\begin{theorem}
\label{thm:outflow-BCs-strong}
  A sufficiently smooth solution $B$ of the magnetic induction equation with Hall
  effect \eqref{eq:induction-transport-split-Hall} with strong form boundary condition
  \eqref{eq:outflow-BCs-strong} or with weak implementation of the boundary condition
  due to the addition of \eqref{eq:outflow-BCs-weak} on the right hand side satisfies
  the energy rate estimate
  \begin{equation}
    \od{}{t} \norm{B}_{L^2(\Omega)}^2
    =
    2 \int_\Omega B \cdot \partial_t B
    \leq
    9 \norm{\nabla u(t)}_{L^\infty(\Omega)} \norm{B(t)}_{L^2(\Omega)}^2.
  \end{equation}
\end{theorem}

\begin{remark}
  The energy (rate) estimate can also be investigated using properties of the
  matrix $\Lambda_+ + R^T \Lambda_- R$, which should be positive semidefinite
  as in \cite{nordstrom2018energy}, cf. Proposition~\ref{prop:Nordstrom-LaCognata}.
  Here, $\Lambda_\pm = \frac{\lambda_\pm}{\abs{v_\pm}^2} \I_3$. The basic result
  (an estimate can be obtained) is the same.
\end{remark}

\begin{remark}
  One might want to specify Dirichlet boundary data of the form $B = B^b$ at an inflow
  boundary. This can be written in the form \eqref{eq:Nordstrom-LaCognata-BC} with
  $R = \I_3$ and appropriate $g$. However, it does not seem to be possible to obtain
  an energy estimate in this way, similar to the case of Dirichlet boundary conditions
  for the incompressible Navier-Stokes equations investigated in \cite{nordstrom2018energy},
  since condition \eqref{eq:Nordstrom-LaCognata-nonhomogeneous-1} is not satisfied
  (and \eqref{eq:Nordstrom-LaCognata-nonhomogeneous-2} makes no sense).
\end{remark}

\begin{remark}
  If there are negative eigenvalues, it might seem to be natural to specify boundary
  data of the form $W_- = g$, i.e. \eqref{eq:Nordstrom-LaCognata-BC} with $R = 0$.
  Then, condition \eqref{eq:Nordstrom-LaCognata-nonhomogeneous-1} can be weakened
  to \eqref{eq:Nordstrom-LaCognata-homogeneous} (which is satisfied trivially for
  $R=0$) and condition \eqref{eq:Nordstrom-LaCognata-nonhomogeneous-2} becomes
  $- \Lambda_- \leq \Gamma < \infty$. Since
  \begin{equation}
  \begin{aligned}
    - \Lambda_-
    &=
    - \frac{\lambda_-}{\abs{v_-}^2}
    =
    - \frac{\lambda_-}{\lambda_-^2 + \frac{(B \cdot  \nu)^2}{4}}
    \\
    &=
    - \frac{
       \bigl( \frac{u}{2} - \frac{\nabla \times B}{\rho} \bigr) \cdot \nu
      - \sqrt{\bigl( \bigl( \frac{u}{2} - \frac{\nabla \times B}{\rho} \bigr) \cdot \nu \bigr)^2
              + (B \cdot \nu)^2 }
    }{
      \bigl( \bigl( \frac{u}{2} - \frac{\nabla \times B}{\rho} \bigr) \cdot \nu \bigr)^2
      + (B \cdot \nu)^2
      - \bigl( \bigl( \frac{u}{2} - \frac{\nabla \times B}{\rho} \bigr) \cdot \nu \bigr)
        \sqrt{\bigl( \bigl( \frac{u}{2} - \frac{\nabla \times B}{\rho} \bigr) \cdot \nu \bigr)^2
              + (B \cdot \nu)^2 }
    },
  \end{aligned}
  \end{equation}
  $- \Lambda_- \to \infty$, e.g. for $B \cdot \nu = 0$ and
  $(u/2 - (\nabla \times B)/\rho) \cdot \nu < 0$,
  $(u/2 - (\nabla \times B)/\rho) \cdot \nu \nearrow 0$.
  Thus, it is not possible to get an energy estimate in this case using
  Proposition~\ref{prop:Nordstrom-LaCognata}.
\end{remark}

\subsection{Semidiscrete Setting}

The Hall term $- \nabla \times \Bigl( \frac{1}{\rho} (\nabla \times B) \times B \Bigr)$
can be discretised directly using SBP derivative operators. Since the energy estimate
relies solely on integration by parts, a discrete analogue holds if SBP
operators are used. As in section~\ref{sec:transport-term-semidiscrete}, the
properties of the induction equation with Hall effect and weak implementation of
the boundary conditions mentioned before remain invariant under semidiscretisation
if the components $\nu_j$ of the outer unit normal are exchanged with the corresponding
boundary matrices $E_j$. This yields
\begin{theorem}
  The semidiscretistion
  \begin{multline}
    \partial_t B_i
    =
    B_j D_j u_i
    - \frac{1}{2} B_i D_j u_j
    - \frac{1}{2} u_j D_j B_i
    - \frac{1}{2} D_j (u_j B_i)
    - D_j \biggl( \frac{(D \times B)_j}{\rho} B_i - \frac{(D \times B)_i}{\rho} B_j \biggr)
    \\
    + M^{-1} E_j \left(
      \1{\bigl( \frac{u}{2} - \frac{D \times B}{\rho} \bigr) \cdot \nu < 0}
      \biggl( \frac{u_j}{2} - \frac{(D \times B)_j}{\rho} \biggr) B_i
      + B_j \frac{(D \times B)_i}{\rho}
    \right)
  \end{multline}
  of the magnetic induction equation \eqref{eq:induction-transport-split-Hall}
  with outflow boundary condition \eqref{eq:outflow-BCs-strong} using
  $(D \times B)_i := \epsilon_{ijk} D_j B_k$ is energy stable, i.e. a
  sufficiently smooth solution satisfies
  \begin{equation}
    \od{}{t} \norm{B}_{M}^2
    =
    2 B_i^T M \partial_t B_i
    \leq
    9 \norm{D u(t)}_{\ell^\infty} \norm{B(t)}_{M}^2.
  \end{equation}
\end{theorem}

\begin{remark}
  As described in section~\ref{sec:energy-stability-of-other-forms}, the transport
  and source term can be discretised using different forms leading to an energy
  estimate. For all forms (with non-zero source term), the same boundary terms
  arise and energy estimates can be obtained.
\end{remark}

\begin{remark}
  Due to the second derivatives appearing in the Hall term, it can be expected that
  there is a time step restriction of the form $\Delta t \propto \Delta x^2$ for
  explicit time integrators. This has been mentioned in the context of the Hall
  MHD equations in \cite{huba2003hall, toth2008hall, toth2012adaptive} with some
  connections to physical waves.
\end{remark}

\begin{remark}
  The semidiscretisation can be implemented straightforwardly (e.g. using extended
  numerical fluxes) as described in section~\ref{sec:implementation} if the discrete
  current $D \times B$ is computed at first.
\end{remark}

\section{Divergence Constraint on the Magnetic Field}
\label{sec:div-constraint}

There are some possibilities to handle the divergence constraint $\div B = 0$
on the magnetic field that have been described in the articles \cite{toth2000divB,
derigs2018ideal}, e.g. the addition of nonconservative source terms
\cite{godunov1972symmetric, powell1994approximate, powell1999solution}, the
projection method \cite{brackbill1980effect}, constrained transport schemes
\cite{toth2000divB}, and generalised Lagrange multipliers or hyperbolic divergence
cleaning \cite{munz2000divergence, dedner2002hyperbolic, derigs2018ideal}.
Here, explicit divergence cleaning via the projection method will be considered
in detail and adapted to the semidiscretisations discussed in the previous
sections. In particular, the focus will be on the magnetic energy and boundary
conditions.

\subsection{Divergence Cleaning via Projection}

For plasma simulations, the projection method to enforce $\div B = 0$ has been
proposed in \cite{brackbill1980effect}. The basic idea can be described as follows.
If $\div B \neq 0$, solve the Poisson equation $-\Delta \phi = \div B$ and set
$\tilde B = B + \grad \phi$. Then, $\div \tilde B = \div B + \Delta \phi = 0$.
Although this idea seems to be pretty simple, the discretisation has to be performed
carefully. The following parts should be investigated:
\begin{itemize}
  \item
  In the derivation above, $\div \grad = \Delta$ has been used. This does not
  hold for all discretisations exactly.

  \item
  Boundary conditions have to be imposed in order to get a well-posed Poisson
  problem.

  \item
  What is the influence of the projection on the total conservation of the magnetic
  field and the magnetic energy?

  \item
  How is the resulting discrete linear equation solved?
\end{itemize}

\subsection{Continuous Setting}
\label{sec:div-constraint-continuous}

The Poisson equation $-\Delta \phi = \div B$ has to be enhanced by boundary conditions
in order to get a well-posed problem. Homogeneous Dirichlet boundary conditions
yield the problem
\begin{equation}
\label{eq:poisson}
\begin{aligned}
  - \Delta \phi &= \div B && \text{in } \Omega,
  \\
  \phi &= 0 && \text{on } \partial\Omega.
\end{aligned}
\end{equation}
Assume that $\phi$ is a sufficiently smooth (say, $C^2$) solution of \eqref{eq:poisson}.
Then, the change of the total mass of the magnetic field due to the projection
$B \mapsto B + \grad \phi$ is
\begin{equation}
\label{eq:change-total-mass}
  \int_\Omega \grad \phi
  =
  \int_{\partial\Omega} \phi \nu
  =
  0,
\end{equation}
since $\phi|_{\partial\Omega} = 0$. The total magnetic energy
$\norm{B + \grad \phi}_{L^2(\Omega)}^2$ after the projection is given by
\begin{equation}
\begin{aligned}
  \norm{B}_{L^2(\Omega)}^2
  &=
  \norm{(B + \grad \phi) - \grad \phi}_{L^2(\Omega)}^2
  \\
  &=
  \norm{B + \grad \phi}_{L^2(\Omega)}^2
  + \norm{\grad \phi}_{L^2(\Omega)}^2
  - 2 \scp{B + \grad\phi}{\grad\phi}_{L^2(\Omega)},
\end{aligned}
\end{equation}
where
\begin{equation}
\label{eq:change-total-energy}
\begin{aligned}
  -\scp{B + \grad\phi}{\grad\phi}_{L^2(\Omega)}
  &=
  -\int_\Omega (B + \grad\phi) \cdot \grad\phi
  \\
  &=
  -\int_{\partial\Omega} \phi \, (B + \grad\phi) \cdot \nu
  + \int_\Omega \phi \div(B + \grad\phi)
  =
  0,
\end{aligned}
\end{equation}
since $\phi|_{\partial\Omega} = 0$ and $\div(B + \grad\phi) = \div B + \Delta \phi = 0$.
Thus, the projection $B \mapsto B + \grad \phi$ reduces the total magnetic energy,
which can be interpreted as a desirable stability condition. This is summed up in
\begin{lemma}
  For sufficiently smooth data, the projection $B \mapsto B + \grad\phi$ where
  $\phi$ solves the Poisson equation \eqref{eq:poisson} with homogeneous Dirichlet
  boundary conditions conserves the total mass $\int_\Omega B$ of the magnetic
  field and is energy stable, i.e. it does not increase the total magnetic energy.
\end{lemma}
\begin{remark}
  Despite these ``nice'' properties, the boundary values of the magnetic field
  will be changed in general. This behaviour of the projection is similar to the
  one of modal filters in spectral (element) methods, cf. \cite{vandeven1991family,
  boyd1998two, hesthaven2008filtering}.
  If the boundary values of the magnetic field shall be preserved by the projection,
  the Poisson equation has to be enhanced by homogeneous Neumann boundary conditions.
  In this case, the two assertions given above will be false in general.
\end{remark}

Moreover, for homogeneous Dirichlet boundary conditions, the projection via \eqref{eq:poisson}
can be interpreted as least norm solution of the underdetermined linear system
$\div \beta = \div B$ that shall be solved to get the update $B \mapsto B - \beta$.
Indeed, formally and without further specification of the domains of the linear
operators, $\div^* = - \grad$ and $\div \div^* = - \Delta$ for homogeneous Dirichlet
boundary conditions due to integration by parts.
The least norm solution of $\div \beta = \div B$ is
\begin{equation}
  \beta = \div^* (\div \div^*)^{-1} \div B
  =
  - \grad (-\Delta)^{-1} \div B,
\end{equation}
where $(-\Delta)^{-1}$ is the solution operator of the Poisson equation \eqref{eq:poisson}.
Thus, $\beta = - \grad \phi$. This is the minimum norm solution of $\div\beta = \div B$.
Indeed, for every other solution $b$ with $\div b = \div B$
\begin{equation}
\label{eq:least-norm-solution-prop-1}
  \norm{b}^2
  =
  \norm{\beta}^2 + \norm{b - \beta}^2
  \geq
  \norm{\beta}^2,
\end{equation}
since
\begin{equation}
\label{eq:least-norm-solution-prop-2}
  \scp{b - \beta}{\beta}
  =
  \scp{b - \beta}{\div^* (\div \div^*)^{-1} \div B}
  =
  \scp{\div(b-\beta)}{(\div \div^*)^{-1} \div B}
  =
  0,
\end{equation}
due to $\div b = \div B = \div \beta$. Hence, the projection $B \mapsto B + \grad \phi$
with $\phi$ given by \eqref{eq:poisson} provides the least possible change of the
magnetic field that is necessary to obtain zero divergence, cf.
\cite[section 5.2]{toth2000divB}. This property will be no longer true if other
boundary conditions are used for the Poisson equation.

\begin{remark}
  This problem can be seen as an ill-posed inverse problem. In this case, it
  can be useful to apply an iterative method for the discrete system and solve it
  not to machine accuracy but to some prescribed tolerance allowing non-vanishing
  divergence of the magnetic field but possibly resulting in better numerical
  solutions, cf. \cite[section~2.4]{kaipio2005statistical} and
  \cite[section 5.4]{toth2000divB}.
\end{remark}

\begin{remark}
  Although the projection $B \mapsto \tilde B := B + \grad \phi$ does not increase
  the total magnetic energy, i.e. $\int_\Omega \abs{\tilde B}^2 \leq \int_\Omega \abs{B}^2$,
  a pointwise estimate of the form $\abs{\tilde B}^2 \leq \abs{B}^2$ can in general
  not be guaranteed. Indeed, consider the magnetic field
  \begin{equation}
    B(x,y,z)
    =
    \vect{
      x + 2x (1-y^2) (1-z^2) \\
      -y + 2y (1-x^2) (1-z^2) \\
      2z (1-x^2) (1-y^2)
    }
  \end{equation}
  with corresponding correction potential
  \begin{equation}
    \phi(x,y,z)
    =
    (1-x^2) (1-y^2) (1-z^2),
    \quad
    \grad \phi(x,y,z)
    =
    \vect{
      - 2x (1-y^2) (1-z^2) \\
      - 2y (1-x^2) (1-z^2) \\
      - 2z (1-x^2) (1-y^2)
    },
  \end{equation}
  and the divergence free projection $\tilde B(x,y,z) = (x, -y, 0)^T$ on the cube
  $\Omega = [-1,1]^3$. Then,
  \begin{equation}
    \abs{\tilde B(x,-1,0)}^2
    =
    x^2 + 1
    >
    x^2 + (2 x^2 - 1)^2
    =
    \abs{B(x,-1,0)}^2
  \end{equation}
  for $x \in (-1,1) \setminus \set{0}$.
  Considering the MHD equations, the (mathematical, convex)
  entropy (in non-dimensional units) is $U = - \rho s$, where $s$ is the (physical)
  specific entropy, given as
  \begin{equation}
    s = \log(p) - \gamma \log(\rho),
    \quad
    p = (\gamma-1) \left( \rho e - \frac{1}{2} \rho \abs{v}^2 - \frac{1}{2} \abs{B}^2 \right),
  \end{equation}
  where $p$ is the pressure, $\rho e$ the total energy, $\rho$ the density, $v$
  the velocity, and $B$ the magnetic field, cf. \cite{derigs2018ideal}. Thus,
  by choosing an appropriate distribution of the density $\rho$, the total
  (mathematical) entropy can increase during the projection of the magnetic field.
  Such an effect has been mentioned in \cite{derigs2018ideal} without description
  of an example.
\end{remark}

\subsection{Discrete Setting}
\label{sec:div-constraint-discrete}

There seem to be at least three general possibilities regarding the discretisation
of the projection $B \mapsto B + \grad \phi$ coupled with the Poisson equation
\eqref{eq:poisson}.
\begin{enumerate}
  \item
  Choose a discretisation of $\div$ and get corresponding discretisations of
  $\grad$ and $-\Delta$ with homogeneous Dirichlet boundary conditions.

  \item
  Choose a discretisation of $-\Delta$ with homogeneous Dirichlet boundary conditions
  and get corresponding discretisations of $\div$ and $\grad$.

  \item
  Choose $\div$ and $-\Delta$ with homogeneous Dirichlet boundary conditions
  independently and ignore the supposed coupling of these discretisations since
  they should be consistent.
\end{enumerate}
In general, it will not be possible to obtain $-\Delta \phi = \div B$ at every
node and $\phi = 0$ at $\partial\Omega$, since the boundary nodes are included
in the discretisation. Thus, the discrete projection will in general not enforce
$\div B = 0$ at boundary nodes. One might argue that this is no severe drawback,
since the divergence at boundary nodes is also influenced by the values at the
other side of the boundary.

Another possibility is to ignore the interpretation of the projection onto
divergence free vector fields as solving a Poisson problem and compute the least
norm solution discretely, if possible.

\subsubsection{Possibility 1: Choose \texorpdfstring{$\div$}{div} with Homogeneous
               Dirichlet Boundary Conditions}
\label{sec:choose-div}

One possibility for the discretisation of the divergence that might be considered
natural or obvious is to use the SBP derivative operators $D_i$. In this
case, the discrete divergence of the magnetic field is $D_i B_i$.
Then, a discrete solution of the Poisson equation \eqref{eq:poisson} can be
obtained by setting the boundary nodes of $\phi$ to zero and solving the
discrete Poisson equation $-D_j D_j \phi = D_i B_i$
at the interior nodes. Thereafter, the magnetic field is updated via
$B_i \mapsto B_i + D_i \phi$.

Then, the divergence of the projected magnetic field is zero at the interior nodes.
Moreover, the total mass of the magnetic field is unchanged if SBP operators
are used, since an analogue of \eqref{eq:change-total-mass} holds discretely.
Moreover, the magnetic energy can only decrease, since an analogue of
\eqref{eq:change-total-energy} holds discretely; the last integral is zero since
$D_i \left( B_i + D_i \phi \right)$ is zero at interior
nodes and $\phi$ is zero at the boundary nodes.

\begin{lemma}
  If the Poisson equation with homogeneous Dirichlet boundary conditions
  \eqref{eq:poisson} is discretised via applying the first derivative SBP operator
  twice, the total magnetic field remains constant and the magnetic energy can
  only decrease due to the projection.
\end{lemma}

\begin{example}
\label{ex:choose-div}
  Using the SBP derivative operators of Example~\ref{ex:SBP-2},
  \begin{equation}
    D
    =
    \frac{1}{2 \Delta x}
    \begin{pmatrix}
      -2 & 2 \\
      -1 & 0 & 1 \\
      & \ddots & \ddots & \ddots \\
      && -1 & 0 & 1 \\
      &&& -2 & 2
    \end{pmatrix},
    \quad
    -D^2
    =
    \frac{1}{4 \Delta 4}
    \begin{pmatrix}
      -2 & 4 & -2 \\
      -2 & 3 & 0 & -1 \\
      -1 & 0 & 2 & 0 & -1 \\
      & \ddots & \ddots & \ddots & \ddots & \ddots \\
      && -1 & 0 & 2 & 0 & -1 \\
      &&& -1 & 0 & 3 & -2 \\
      &&&& -2 & 4 & -2
    \end{pmatrix}.
  \end{equation}
  If the boundary nodes are enforced to be zero, this becomes $(-D^2)_0$.
  The part of $(-D^2)_0$ describing the interior nodes is
  \begin{equation}
    \left[(-D^2)_0\right]_{2:N-1,2:N-1}
    =
    \frac{1}{4 \Delta x^2}
    \begin{pmatrix}
      3 & 0 & -1 \\
      0 & 2 & 0 & -1 \\
      -1 & 0 & 2 & 0 & -1 \\
      & \ddots & \ddots & \ddots & \ddots & \ddots \\
      && -1 & 0 & 2 & 0 & -1 \\
      &&& -1 & 0 & 2 & 0 \\
      &&&& -1 & 0 & 3
    \end{pmatrix},
  \end{equation}
  where Matlab like notation has been used for the indices. This operator is
  symmetric and positive definite.
\end{example}

\subsubsection{Possibility 2: Choose \texorpdfstring{$-\Delta$}{-Δ} with Homogeneous
               Dirichlet Boundary Conditions}

In a periodic domain, the classical second order Laplace operator is given by
\begin{equation}
  D^{(2)}
  =
  \underbrace{
  \frac{1}{\Delta x^2}
  \begin{pmatrix}
    -2 & 1 &&& 1 \\
    1 & -2 & 1 \\
    & \ddots & \ddots & \ddots \\
    && 1 & -2 & 1 \\
    1 &&& 1 & -2
  \end{pmatrix}
  }_{\hat= \Delta}
  =
  \underbrace{
  \frac{1}{\Delta x}
  \begin{pmatrix}
    1 & && -1 \\
    -1 & 1 \\
    & \ddots & \ddots \\
    && -1 & 1
  \end{pmatrix}
  }_{\hat= \grad}
  \underbrace{
  \frac{1}{\Delta x}
  \begin{pmatrix}
    -1 & 1 \\
    & \ddots & \ddots \\
    && -1 & 1 \\
    1 &&& -1
  \end{pmatrix}
  }_{\hat= \div}.
\end{equation}
In this case, a factorisation in adjoint discrete gradient and divergence operators
exist. If this discretisation of the negative Laplace operator shall be used, the
divergence should be computed via forward differences and the gradient via backward
differences (or vice versa). However, such a factorisation does not seem to be immediate
for general discretisations of the Laplace operator with homogeneous boundary
conditions. Thus, this approach will not be pursued in the following.

\subsubsection{Possibility 3: Choose \texorpdfstring{$\div$}{div} and
               \texorpdfstring{$-\Delta$}{-Δ} with Homogeneous Dirichlet Boundary
               Conditions}
\label{sec:choose-div-and-Delta}

Another possibility is to use the standard narrow stencil second derivative operator
with homogeneous Dirichlet boundary conditions to solve the Poisson equation at
the interior nodes and use the standard SBP first derivative operator to compute
the gradient. Again, the total amount of the magnetic field is still unchanged,
as in the previous cases. If no relation between the first and second derivative operators is
known, nothing can be said about the magnetic energy, since the additional term
$-\scp{B + \grad\phi}{\grad\phi}_{L^2(\Omega)}$ \eqref{eq:change-total-energy}
has no definite sign. However, if compatible first and second derivative SBP
operators as proposed in \cite{mattsson2008stable} are used, this term can be
estimated. Indeed, these operators fulfil
\begin{equation}
  M D^{(2)}_i
  =
  - {D^{(1)}_i}^T M D^{(1)}_i + E_i S_i - R_i,
\end{equation}
where $D^{(k)}_i$ is the operator approximating the $k$-th derivative in
coordinate direction $i$, $E_i$ is the $i$-th boundary operator, $S_i$
approximates the derivative in direction $i$ at the boundary, and $R_i$ is
positive semidefinite, cf. \cite[Definition~3.1]{mattsson2008stable}. Thus, the
discrete analogue of \eqref{eq:change-total-energy} is
\begin{multline}
  -\scp{B + \grad\phi}{\grad\phi}_{L^2(\Omega)}
  \approx
  - \phi^T {D^{(1)}_i}^T M \left( B_i + D^{(1)}_i \phi \right)
  =
  - \phi^T {D^{(1)}_i}^T M B_i
  - \phi^T {D^{(1)}_i}^T M D^{(1)}_i \phi
  \\
  \stackrel{\eqref{eq:SBP}}{=}
  \underbrace{
    - \phi^T E_i B_i
  }_{=0}
  \underbrace{
    + \phi^T M D^{(1)}_i B_i
    + \sum_{i=1}^3 \phi^T M D^{(2)}_i \phi
  }_{=0}
  \underbrace{
    - \phi^T E_i S_i \phi
  }_{=0}
  \underbrace{
    + \sum_{i=1}^3 \phi^T R_i \phi
  }_{\geq 0}
  \geq 0.
\end{multline}
The first and fourth term on the right hand side vanish since $\phi$ is zero
at the boundary. The sum of the second and third term vanishes since $\phi$
is zero at the boundary and solves the discrete Poisson equation in the interior.
Finally, the remaining term is non-negative. Thus, the total magnetic energy before
the correction is given by the magnetic energy after the correction plus some
non-negative terms. Therefore, the magnetic energy can again only decrease as
in section~\ref{sec:choose-div}. Nevertheless, the discrete divergence will in
general not be zero after the correction, since $\div \grad = \Delta$ does not
hold discretely.

\begin{lemma}
  If the Poisson equation with homogeneous Dirichlet boundary conditions
  \eqref{eq:poisson} is discretised via a narrow stencil second derivative SBP
  operator that is compatible with the first derivative operator, the total magnetic
  field remains constant and the magnetic energy can only decrease due to the
  projection. However, the discrete divergence will in general not vanish after
  the correction.
\end{lemma}

\begin{remark}
  One might think that the new magnetic energy is smaller than in the case of
  section~\ref{sec:choose-div}, since the term with the scalar product in
  \eqref{eq:change-total-energy} is non-positive instead of zero. However, the
  numerical solution $\phi$ will also be different, since the Laplace
  operator is different. Thus, the new energies cannot be compared a priori
  in general.
\end{remark}

\begin{example}
\label{ex:choose-div-and-Delta}
  Using the SBP derivative operators of Example~\ref{ex:SBP-2}, a compatible
  SBP operator for the second derivative given in \cite{mattsson2008stable} is
  \begin{equation}
    D^{(2)}
    =
    \frac{1}{\Delta x^2}
    \begin{pmatrix}
      1 & -2 & 1 \\
      1 & -2 & 1 \\
      & \ddots & \ddots & \ddots \\
      && 1 & -2 & 1 \\
      && 1 & -2 & 1
    \end{pmatrix}.
  \end{equation}
  Enforcing homogeneous Dirichlet boundary conditions, the inner part of this
  operator becomes
  \begin{equation}
    \left[- D^{(2)}_0 \right]_{2:N-1,2:N-1}
    =
    \frac{1}{\Delta x^2}
    \begin{pmatrix}
      2 & -1 \\
      -1 & 2 & -1 \\
      & \ddots & \ddots & \ddots \\
      && -1 & 2 & -1 \\
      &&& -1 & 2
    \end{pmatrix},
  \end{equation}
  where Matlab like notation has been used again for the indices. This is the
  classical form of the discrete Laplace operator for homogeneous Dirichlet boundary
  conditions using second order central finite differences. Again, this operator
  is symmetric and positive definite.
\end{example}

\subsubsection{Possibility 4: Choose \texorpdfstring{$\div$}{div} and Compute
               the Least Norm Solution}
\label{sec:least-norm-solution}

Here, the concept of the least norm solution mentioned already in
section~\ref{sec:div-constraint-continuous} will be used at the discrete level.
Using a discrete divergence $\div$, the linear equation $\div \beta = d$ with
$d = - \div B$ should be solved for $\beta$. Since $d = - \div B$ is in the
range of $\div$, there is at least one solution, namely $\beta = - B$. Since the
kernel (nullspace) of $\div$ is not trivial, there are in general several
solutions. Among these, the least norm solution is given as
\begin{equation}
\label{eq:least-norm-solution}
  \beta = \div^* (\div \div^*)^{-1} d,
\end{equation}
where $(\div \div^*)^{-1} d$ is a solution $\phi$ of $(\div \div^*) \phi = d$. Indeed,
$\div \beta = \div \div^* (\div \div^*)^{-1} d = d$ and for every other vector
field $b$ with $\div b = d$, $\norm{b}^2 \geq \norm{\beta}^2$, since the equations
\eqref{eq:least-norm-solution-prop-1} and \eqref{eq:least-norm-solution-prop-2}
still hold.

The operator $(\div \div^*)$ is symmetric and positive semidefinite, since for
every discrete scalar field $\psi$, $\scp{\psi}{\div \div^* \psi} = \norm{\div^* \psi}^2
\geq 0$. Therefore, the kernel of $(\div \div^*)$ is the kernel of $\div^*$
and this kernel is in general not trivial. Nevertheless, the right hand side
$d = - \div B$ is orthogonal to this kernel, since $\scp{\psi}{\div B} =
\scp{\div^* \psi}{B} = 0$ for $\div^* \psi = 0$.

The least norm solution \eqref{eq:least-norm-solution} has the same nice properties
as the projection via the Poisson equation with homogeneous boundary conditions.
Indeed, the total magnetic field is unchanged, since
\begin{equation}
  \scp{1}{B + \beta}
  =
  \scp{1}{B + \div^* (\div \div^*)^{-1} d}
  =
  \scp{1}{B} + \scp{\div 1}{(\div \div^*)^{-1} d}
  =
  \scp{1}{B},
\end{equation}
where $1$ denotes the discrete vector field whose components are one. Moreover,
\begin{equation}
  \norm{B}^2
  =
  \norm{B + \beta - \beta}^2
  =
  \norm{B + \beta}^2 + \norm{\beta}^2 - 2 \scp{B + \beta}{\beta}
  =
  \norm{B + \beta}^2 + \norm{\beta}^2,
\end{equation}
since
\begin{equation}
\begin{aligned}
  \scp{B + \beta}{\beta}
  &=
  \scp{B - \div^* (\div \div^*)^{-1} \div B}{- \div^* (\div \div^*)^{-1} \div B}
  \\
  &=
  - \scp{\div B - \div B}{(\div \div^*)^{-1} \div B}
  =
  0.
\end{aligned}
\end{equation}
The calculations above are valid for SBP operators if the $L^2$ scalar products are
discretised via the corresponding mass matrix.
\begin{lemma}
  If the least norm solution \eqref{eq:least-norm-solution} is computed and the
  divergence is discretised via the first derivative SBP operator, the total
  magnetic field remains constant and the magnetic energy can only decrease due
  to the projection.
\end{lemma}

\begin{example}
\label{ex:least-norm-solution}
  Using again the SBP derivative operators of Example~\ref{ex:SBP-2},
  the adjoint operator used to compute $\div^*$ is due to the SBP property
  \eqref{eq:SBP}
  \begin{equation}
    D^*
    =
    M^{-1} D^T M
    =
    - D + M^{-1} E
    =
    \frac{1}{2 \Delta x}
    \begin{pmatrix}
      -2 & -2 \\
      1 & 0 & -1 \\
      & \ddots & \ddots & \ddots \\
      && 1 & 0 & -1 \\
      &&& 2 & 2
    \end{pmatrix}.
  \end{equation}
  Moreover, the operator $(\div \div^*)$ is given by
  \begin{equation}
    D D^*
    =
    D M^{-1} D^T M
    =
    - D^2 + D M^{-1} E
    =
    \frac{1}{4 \Delta x^2}
    \begin{pmatrix}
      6 & 4 & -2 & 0 \\
      2 & 3 & 0 & -1 \\
      -1 & 0 & 2 & 0 & -1 \\
      & \ddots & \ddots & \ddots & \ddots & \ddots \\
      && -1 & 0 & 2 & 0 & -1 \\
      &&& -1 & 0 & 3 & 2 \\
      &&& 0 & -2 & 4 & 6
    \end{pmatrix}.
  \end{equation}
  This operator is symmetric and positive semidefinite.
\end{example}

\begin{remark}
  In the interior, this least norm solution still solves the Poisson equation.
  However, the near boundary terms are different from the approach described in
  section~\ref{sec:choose-div}.
\end{remark}

\begin{remark}
  The projection via solution of the Poisson equation with homogeneous Dirichlet
  boundary conditions (section~\ref{sec:choose-div}) will in general not enforce
  $\div \tilde B = 0$ at the boundary. Contrary, the least norm solution fulfils
  $\div \tilde B = 0$ everywhere. On the other hand, the linear systems that has
  to be solved using the method of section~\ref{sec:choose-div} is symmetric and
  positive definite whereas the linear system arising in the approach described
  in this section is only positive semidefinite. Thus, solving the system via
  iterative methods for a given right hand side, the convergence behaviour might
  be different. Nevertheless, the conjugate gradient methods does still converge.
\end{remark}

\subsection{Solution of the Discrete Linear System}

There are many iterative methods that can be used to solve discrete linear
systems of the form $A x = y$, where $A$ is a discretisation of $-\Delta$ and
$y$ is the discrete version of $\div B$. The \emph{conjugate gradient} (CG) method
can be motivated by minimising $(x_* - x)^T A (x_* - x)$, where $x_*$ is the solution
of the linear system. This corresponds to minimising the error of $\grad \phi$,
i.e. of the correction to the magnetic field, since $A$ is a discretisation of
the Laplace operator with homogeneous boundary conditions.

Preconditioning can in general be very useful to accelerate the convergence of
Krylov subspace methods such as the CG method. For systems of the form described
above, multigrid methods have been very successful in the last decades. However,
while these can provide huge improvements for general right-hand sides, the
divergence errors occurring during a few timesteps of a simulation are relatively small. Therefore,
multigrid methods did not yield significant improvements in our numerical experiments
due to their overhead. Thus, no preconditioning is used.

\section{Numerical Results}
\label{sec:numerical-results}

In this section, some experiments using the numerical methods described hitherto
will be conducted. The numerical solutions are advanced in time using the fourth
order, five stage, low-storage Runge-Kutta scheme of \cite{carpenter1994fourth}
with time step $\Delta t = \cfl \frac{\min_i \Delta x_i}{\max \abs{u}}$, where
the CFL number is chosen as $\cfl = 0.95$ if not mentioned otherwise. Errors and
energies of numerical solutions are computed using the SBP mass matrices.
Derivative operators with interior order of accuracy $2$, $4$, and $6$ are used.
The coefficients for the second order scheme are given in Example~\ref{ex:SBP-2}
and corresponding operators used for divergence cleaning are described in
section~\ref{sec:div-constraint}. The corresponding operators for the fourth order
scheme are given in \appendixref{sec:SBP-HO} and the ones for the other scheme can be
obtained similarly using the coefficients of the first and second derivative
operators of \cite{mattsson2004summation}.

Having investigated all combinations of parameters given in
Tables~\ref{tab:uiBj-forms}--\ref{tab:ujBi-forms}, the parameter combinations
given in Table~\ref{tab:parameter-choices} have been chosen for detailed convergence
experiments. These combinations are representative and have been made based on
results presented in sections~\ref{sec:rotation_3D} and \ref{sec:confined_domain}.

\begin{table}[!ht]
\centering
  \caption{Parameter choices of the different forms used in the numerical experiments.}
  \label{tab:parameter-choices}
  \begin{tabular*}{\linewidth}{@{\extracolsep{\fill}}*7c@{}}
    \toprule
    & \mynewtag{1}{itm:central-zero-central} & \mynewtag{2}{itm:central-central-central}
    & \mynewtag{3}{itm:split-central-split} & \mynewtag{4}{itm:product-central-product}
    & \mynewtag{5}{itm:product-central-split} & \mynewtag{6}{itm:product-central-central}
    \\
    \midrule
    $\partial_j(u_i B_j)$
    & central & central
    & split & product
    & product & product
    \\
    source term
    & zero & central
    & central & central
    & central & central
    \\
    $-\partial_j (u_j B_i)$
    & central & central
    & split & product
    & split & central
    \\
    \bottomrule
  \end{tabular*}
\end{table}

Combination~\ref{itm:central-zero-central} might be the most obvious choice if
no energy investigation of the induction equations has been performed. However,
no energy estimate can be obtained for this scheme.
The parameters~\ref{itm:central-central-central}--\ref{itm:product-central-product}
use a source term but maintain the anti-symmetry of $\nabla \times (u \times B)$
otherwise. In \cite{koley2009higher}, the scheme~\ref{itm:product-central-product}
has been used. The choice~\ref{itm:product-central-split} corresponds to the form
for which an energy estimate can be obtained at the continuous level without further
application of the product rule. Finally, the method~\ref{itm:product-central-central}
has been used in \cite{mishra2010stability}.

The numerical schemes have been implemented in OpenCL using 64 bit floating point
numbers (\texttt{double}). It can be expected that the implementation can be
improved, in particular the one for higher order schemes. If runtimes are given,
they are given in seconds and have been obtained on an
Intel Xeon CPU E5-2620 v3 @ 2.40GHz
unless mentioned otherwise. These runtimes are single experiment measurements
and should only be considered as a rough guideline. Performance of optimised
implementations on different hardware will be considered in future work.

\subsection{Linear Induction Equation: Order of Convergence}
\label{sec:rotation_3D}

In this section, a convergence study using an exact solution of the linear magnetic
induction equation \eqref{eq:induction-transport-conservative} is performed.
The analytical solution
\begin{equation}
\label{eq:rotation_3D}
\begin{gathered}
  B(t,x,y,z) = R(t) \cdot B^0\bigl( R(-t) \cdot (x,y,z)^T \bigr),
  \quad
  u(x,y,z)
  =
  \frac{1}{\sqrt{3}}
  \begin{pmatrix}
    z-y \\
    x-z \\
    y-x
  \end{pmatrix},
  \\
  R(t)
  =
  \frac{1}{3}
  \begin{pmatrix}
    1 + 2 \cos(t) & 1 - \cos(t) - \sqrt{3} \sin(t) & 1 - \cos(t) + \sqrt{3} \sin(t) \\
    1 - \cos(t) + \sqrt{3} \sin(t) & 1 + 2 \cos(t) & 1 - \cos(t) - \sqrt{3} \sin(t) \\
    1 - \cos(t) - \sqrt{3} \sin(t) & 1 - \cos(t) + \sqrt{3} \sin(t) & 1 + 2 \cos(t)
  \end{pmatrix},
  \\
  B^0(x,y,z) = \alpha(x,y,z)
  \begin{pmatrix}
    \frac{1}{48} \left( 3 - \sqrt{3} - 4 \sqrt{3} y + 4 \sqrt{3} z \right) \\
    \frac{1}{48} \left( -3 - \sqrt{3} + 4 \sqrt{3} x - 4 \sqrt{3} z \right) \\
    \frac{1}{8 \sqrt{3}} \left( 1 - 2x + 2y \right)
  \end{pmatrix},
  \\
  \alpha(x,y,z) = \exp\Bigl( -\frac{5}{3} \left(
    3 - 2 (3 + \sqrt{3}) x + 12 x^2 - 2 (-3 +  \sqrt{3}) y + 12 y^2 + 4 \sqrt{3} z + 12 z^2
  \right)\Bigr),
\end{gathered}
\end{equation}
is inspired by the ones in two space dimensions used in \cite{torrilhon2004constraint,
fuchs2009stable, koley2009higher}. However, this solution is not aligned with the
Cartesian grid in three space dimensions.

The domain is chosen as $\Omega = [-1,1]^3$ and both initial and boundary conditions
are given by the analytical solution at $t = 0$ and $\partial\Omega$, respectively.
Errors of the numerical solutions using $N$ points per space direction are computed
at the final time $T = 2 \pi$.

Results using $N=40$ and the SBP operator of interior order of accuracy 4 are given
in Table~\ref{tab:rotation_3D_N_40_order_4} in \appendixref{sec:numerical-results-appendix}.
There, the errors
\begin{equation}
  \epsilon_{B} = \norm{B_\mathrm{num} - B_\mathrm{ana}}_M,
  \qquad
  \epsilon_{\div B} = \norm{\div B_\mathrm{num}}_M,
\end{equation}
of the magnetic field and its divergence are computed using the SBP mass matrix.

The following observations can be made.
Firstly, using no source term yields non-acceptable results (at least four orders
of magnitude larger errors) if $\partial_j(u_i B_j)$ is discretised using the
product or split form. Secondly, the error of the magnetic field is nearly the
same for all other configurations. Moreover, the divergence error is independent
of the discretisation of $-\partial_j(u_j B_i)$ in this case. The smallest divergence
error is obtained for the central forms of $\partial_j(u_i B_j)$ and the source
term.

Results of convergence experiments using the parameter choices given in
Table~\ref{tab:parameter-choices} can be found in
Tables~\ref{tab:rotation_3D-central-zero-central}--\ref{tab:rotation_3D-product-central-central} in \appendixref{sec:numerical-results-appendix}. There, the errors $\epsilon_{B}$
of the magnetic field and $\epsilon_{\div B}$ of the divergence are computed
using the SBP mass matrix. Additionally, the \emph{experimental order of convergence}
(EOC) for these quantities is given.

All schemes converge at least with the expected order of accuracy, i.e. $p+1$
for diagonal norm SBP operators with interior order $2p$. The schemes with
interior order of accuracy four show even an EOC of four which is better than
expected.

As is well-known in the literature \cite{kreiss1972comparison}, high order
schemes can be beneficial for the smooth solutions considered here. Indeed,
in order to obtain an error of the magnetic field with order of magnitude
$10^{-3}$, the second order schemes need ca. \SI{3e3}{s}, the fourth order
schemes need ca. \SI{2e1}{s}, and the sixth order ones need approximately
\SI{1e0}{s}.

\subsection{Linear Induction Equation: Energy Growth}
\label{sec:confined_domain}

In this section, the energy growth of numerical solutions using different parameters
will be compared. The stationary solution
\begin{equation}
\label{eq:confined_domain}
  u(x,y,z)
  =
  \begin{pmatrix}
    \sin(\pi x) \cos(\pi y) \cos(\pi z) \\
    \cos(\pi x) \sin(\pi y) \cos(\pi z) \\
    - 2 \cos(\pi x) \cos(\pi y) \sin(\pi z)
  \end{pmatrix},
  \quad
  B(t,x,y,z) = u(x,y,z),
\end{equation}
is considered in the domain $\Omega = [0,1]^3$. Thus, the velocity $u$ vanishes
at the boundary $\partial\Omega$ and no energy is transported out of the domain.
Using $N$ points per space direction, errors of the numerical solutions at the
final time $T = 2$ are compared.

Using again $N=40$ and the fourth order SBP operator, the results given in
Table~\ref{tab:confined_domain_N_40_order_4} in \appendixref{sec:numerical-results-appendix}
have been obtained.
For this test case, it is extremely important to preserve the anti-symmetry of
$\nabla \times (u \times B)$ discretely by choosing the same form for
$\partial_j(u_i B_j)$ and $-\partial_j(u_j B_i)$. Note that the application of
the split forms of $-\partial_j(u_j B_i)$ and the source term is equivalent to
the product form of $-\partial_j(u_j B_i)$ and the central form of the source term.
If the anti-symmetry is not preserved discretely, the errors can be several
orders of magnitude larger. This corresponds to a discrete preservation of the
steady state given by $u \parallel B$ and is linked to so-called well-balanced
schemes that are designed to preserve such steady states, e.g. for the shallow
water equations \cite{bouchut2004nonlinear}.

In particular, these observations show that having obtained an energy estimate
is not enough. Although the constants appearing in the discrete energy estimates
are larger for some forms, they can perform better on a finite grid for this test
case. Since the constants are obtained via worst case estimates, they do not
necessarily describe the behaviour of the schemes for every test case on a
realistic grid.

Results of convergence studies for this setup are given in
Tables~\ref{tab:confined_domain-central-zero-central}--\ref{tab:confined_domain-product-central-central} in \appendixref{sec:numerical-results-appendix}.
The parameter choice \ref{itm:central-zero-central} (central, zero, central)
preserves the steady state for all orders of accuracy. The other discretisations
using the same form for $\partial_j(u_i B_j)$ and $-\partial_j(u_j B_i)$ preserve
the steady state to machine accuracy if the second order SBP operator is used.
Otherwise, they perform reasonably well and converge approximately with the
expected order.

The other two parameter combinations --- \ref{itm:product-central-split} (product,
central, split) and \ref{itm:product-central-central} (product, central, central)
--- perform worse. The second and sixth order schemes do not seem to be in the
asymptotic regime, based on the low experimental orders of convergence. The choice
\ref{itm:product-central-split} (product, central, split) performs better than
the other one.

\subsection{Nonlinear Induction Equation: Order of Convergence}
\label{sec:hall_periodic}

Here, the nonlinear magnetic induction equation \eqref{eq:induction-full} with
transport and Hall term is considered. Since no energy stable inflow boundary
conditions have been derived, a periodic domain is chosen to test the order
of convergence.
The analytical solutions are given by exact solutions of the incompressible
Hall MHD equations with constant particle density $\rho \equiv 1$ that have been
computed in \cite{mahajan2005exact}. They are
\begin{equation}
\label{eq:hall_periodic}
  B(t,x,y,z) = \alpha u(t,x,y,z) + n,
  \quad
  u(t,x,y,z)
  =
  \begin{pmatrix}
    a \cos(k y + \alpha k t n_2) + b \sin(k z + \alpha k t n_3) \\
    b \cos(k z + \alpha k t n_3) + c \sin(k x + \alpha k t n_1) \\
    c \cos(k x + \alpha k t n_1) + a \sin(k y + \alpha k t n_2)
  \end{pmatrix},
\end{equation}
where $k = (1 - \alpha^2) / \alpha$, and $n = (n_1,n_2,n_3)$, $a$, $b$, $c$,
as well as $\alpha$ are constants. Choosing these as $n_1 = n_2 = n_3 = 1 / \sqrt{3}$,
$a = b = c = 1$, $\alpha = 1/2$ yields $k = 3/2$. Thus the solution is smooth in
the domain $\Omega = [0, 4\pi/3]^3$ with periodic boundary conditions.
The discretisations use $N$ points per space direction. Due to the second derivatives
appearing in the Hall term, the CFL number is chosen as $0.95 / N$ and the numerical
solutions are advanced up to the final time $T = 1$.

Results of convergence experiments can be found in
Tables~\ref{tab:hall_periodic-central-zero-central}--\ref{tab:hall_periodic-product-central-central} in \appendixref{sec:numerical-results-appendix}.
All schemes converge with the expected order $2p$. The schemes using the central
discretisation for both $\partial_j(u_i B_j)$ and $-\partial_j(u_j B_i)$ keep
the divergence norm near the initial error due to the projection onto the numerical
mesh. For the other schemes, the divergence norm converges with an order between
$2p$ and $2p+\frac{1}{2}$.

The good performance of the central discretisations can be explained as follows.
Since $D_i D_j = D_j D_i$ holds discretely, these schemes satisfy
\begin{equation}
  \partial_t D_i B_i
  =
  D_i D_j (u_i B_j - u_j B_i)
  =
  0
\end{equation}
if no source term is added and
\begin{equation}
  \partial_t D_i B_i
  =
  D_i D_j (u_i B_j - u_j B_i) - D_i (u_i D_j B_j)
  =
  - D_j (u_j D_i B_i)
\end{equation}
if a source term is added, similarly to the continuous case \eqref{eq:transport-eq-divB}.
Thus, if the initial condition is (nearly) divergence free, these schemes preserve
this property. However, this is in general not the case if (nonperiodic) boundary
conditions are added, cf. section~\ref{sec:bounds-on-divB}.

\subsection{Nonlinear Induction Equation: Outflow Boundary Conditions}
\label{sec:hall_outflow}

In this section, stability properties of the new outflow boundary condition
\eqref{eq:outflow-BCs-weak} will be studied. Therefore, the setup given in
section~\ref{sec:hall_periodic} will be used but the outflow boundary conditions
are chosen instead of periodic ones.

Since the Hall term is not negligible at the boundaries, this test case is relatively
demanding. Indeed, ignoring the Hall term in the surface terms by using the linear
boundary conditions with either homogeneous boundary data or using the analytical
solution \eqref{eq:hall_periodic} results in a blow-up of the numerical solutions
(\texttt{NaN}).

Using instead the outflow boundary condition \eqref{eq:outflow-BCs-weak}, the
numerical solutions do not blow up in most cases. The only exception is given
by the ``naive'' parameter choice \ref{itm:central-zero-central} (central, zero,
central) for which no energy estimate has been obtained. The energy and divergence
errors of the numerical solutions for the other parameter choices can be found in
Tables~\ref{tab:hall_outflow-central-central-central}%
--\ref{tab:hall_outflow-product-central-central} in
\appendixref{sec:numerical-results-appendix}.

It can be observed that the magnetic energy at the final time decreases with
increasing resolution (number of grid nodes $N$ or order of accuracy). This
could be expected since the outflow boundary condition \eqref{eq:outflow-BCs-weak}
has been designed to result in a decreasing energy.
Secondly, for fixed order of accuracy and spatial resolution, the values of
the magnetic energy and the divergence norm are of the same order of magnitude
for all five schemes. However, it is unknown whether a unique and smooth solution
with this choice of boundary conditions exists and whether such a solution has a
vanishing divergence, cf. the discussion in section~\ref{sec:bounds-on-divB}.
Nevertheless, the schemes \ref{itm:central-central-central} (central, central, central)
and \ref{itm:split-central-split} (split, central, split) yield a smaller divergence
norm than the other schemes (approximately between \SI{15}{\percent} and
\SI{20}{\percent}).

The magnetic energy and divergence norm of numerical solutions for the representative
parameter choices \ref{itm:split-central-split} (split, central, split) and
\ref{itm:product-central-split} (product, central, split) are visualised in
Figure~\ref{fig:hall_outflow} up to the final time $T = 5$.
As can be seen there, the magnetic energy decays over time for most cases and is
smaller for higher order of accuracy. The only exception is given by the choice
\ref{itm:product-central-split} (product, central, split) with order 2; in that
case, the energy decays at first but starts to increase at $t \approx 2$.
For the same parameters, the norm of the divergence of $B$ grows fastest.
Similarly, the divergence norm increases in time for all orders with the choice
\ref{itm:product-central-split} (product, central, split) but remains bounded
for the parameter set \ref{itm:split-central-split} (split, central, split).

\begin{figure}[th]
\centering
  \begin{subfigure}{0.48\textwidth}
  \captionsetup{skip=0pt}
  \centering
    \includegraphics[width=\textwidth]{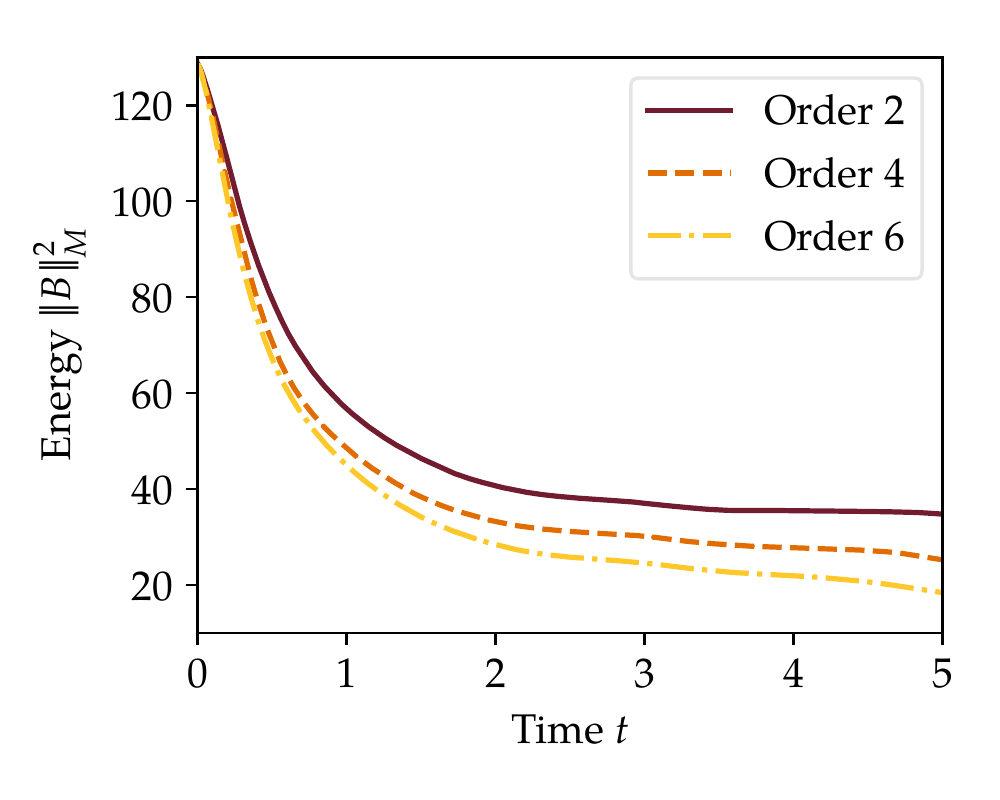}
    \caption{Form \ref{itm:split-central-split} (split, central, split), energy.}
  \end{subfigure}%
  ~
  \begin{subfigure}{0.48\textwidth}
  \captionsetup{skip=0pt}
  \centering
    \includegraphics[width=\textwidth]{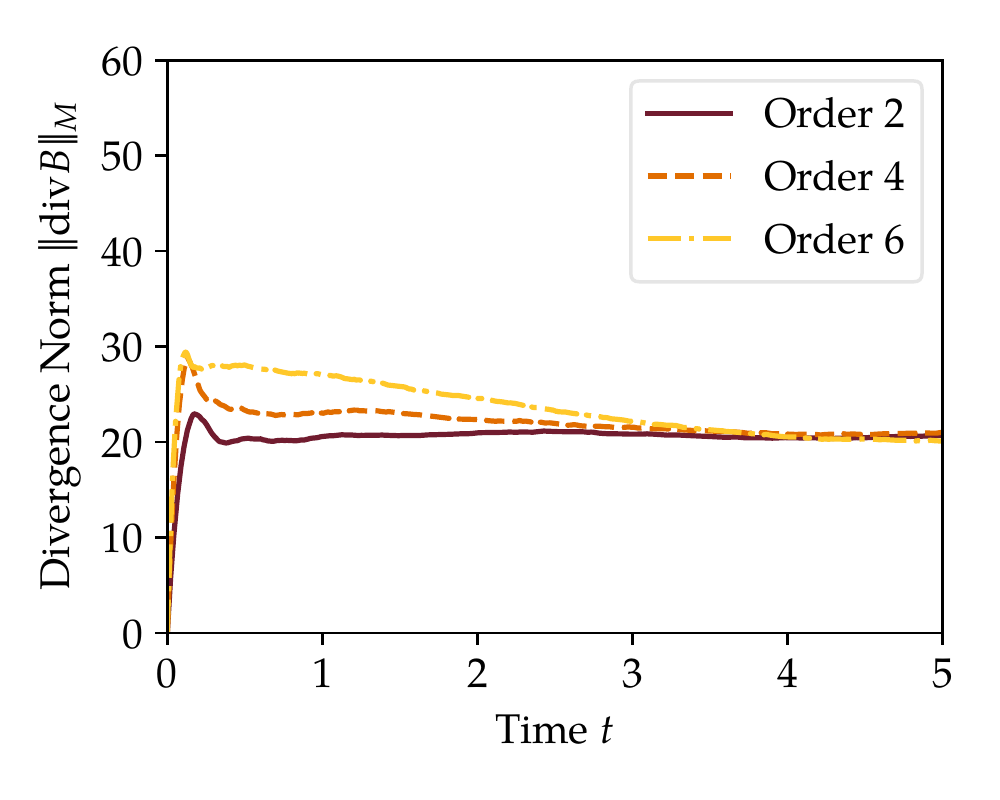}
    \caption{Form \ref{itm:split-central-split} (split, central, split), $\norm{\div B}_M$.}
  \end{subfigure}%
  \\
  \begin{subfigure}{0.49\textwidth}
  \captionsetup{skip=0pt}
  \centering
    \includegraphics[width=\textwidth]{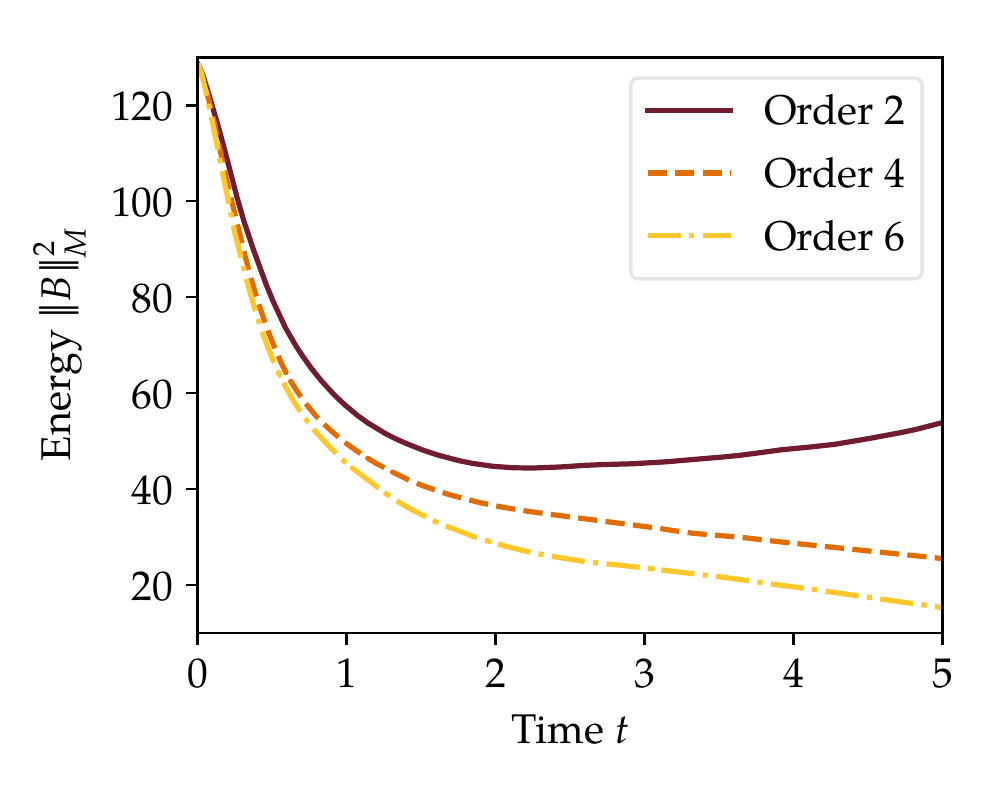}
    \caption{Form \ref{itm:product-central-split} (product, central, split), energy.}
  \end{subfigure}%
  ~
  \begin{subfigure}{0.49\textwidth}
  \captionsetup{skip=0pt}
  \centering
    \includegraphics[width=\textwidth]{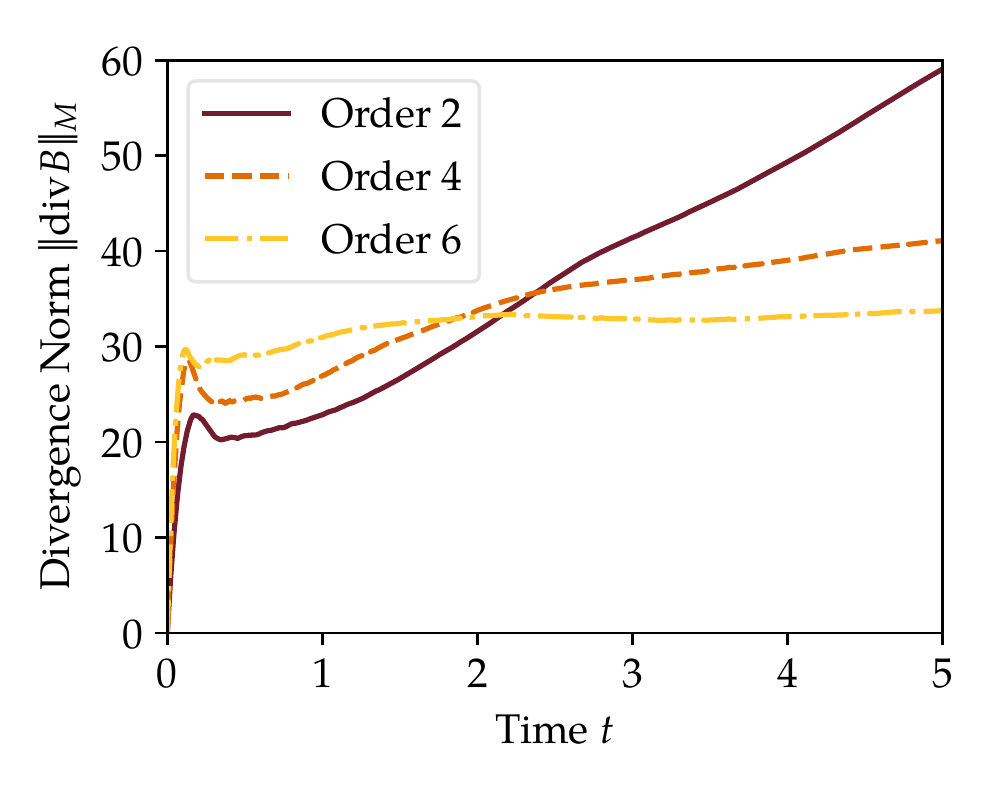}
    \caption{Form \ref{itm:product-central-split} (product, central, split), $\norm{\div B}_M$.}
  \end{subfigure}%
  \caption{Magnetic energy and divergence norms of numerical solutions of the
           nonlinear induction equation with Hall effect
           \eqref{eq:induction-transport-conservative-Hall}
           using SBP operators of different order and two choices of forms given
           in Table~\ref{tab:parameter-choices}.}
  \label{fig:hall_outflow}
\end{figure}

As mentioned in Remark~\ref{rem:hall_outflow_u_vs_u2}, the appearance of $u/2$
instead of $u$ in the proposed outflow boundary condition \eqref{eq:outflow-BCs-weak}
might be irritating. However, simply replacing $u/2$ with $u$ results in schemes
with worse performance concerning, e.g. the maximal stable time step. Indeed,
maximal CFL numbers such that the numerical solutions do not blow up till the
final time $T=1$ are given in Table~\ref{tab:hall_outflow_u_vs_u2}. There, stable
time steps are between two and three times as big for the proposed outflow boundary
condition compared to the altered one. Note that no energy estimate has been obtained
for the latter while the energy remains bounded if the proposed condition is used.

\begin{table}[!ht]
\centering
  \caption{Maximal CFL numbers such that numerical solutions of the nonlinear
           induction equation with Hall effect \eqref{eq:induction-transport-conservative-Hall}
           using SBP operators of different orders and the choice
           \ref{itm:central-central-central} (central, central, central) do not
           blow up for $N=40$.}
  \label{tab:hall_outflow_u_vs_u2}
  \begin{tabular*}{\linewidth}{@{\extracolsep{\fill}}*4c@{}}
    \toprule
    & Interior Order 2
    & Interior Order 4
    & Interior Order 6
    \\
    \midrule
    $\frac{u}{2} - \frac{\nabla \times B}{\rho} \dots$
    & $\cfl = 1.9 / N$ & $\cfl = 1.9 / N$ & $\cfl = 1.5 / N$
    \\
    $u - \frac{\nabla \times B}{\rho} \dots$
    & $\cfl = 0.7 / N$ & $\cfl = 0.6 / N$ & $\cfl = 0.6 / N$
    \\
    \bottomrule
  \end{tabular*}
\end{table}

\subsection{Divergence Cleaning}
\label{sec:div-constraint-results}

In order to test the influence of the divergence cleaning schemes via different
projection methods, setups of numerical experiments presented before will be used
to compare properties of numerical solutions obtained with or without divergence
cleaning.

The six parameter combinations of Table~\ref{tab:parameter-choices} have been
used to compute numerical solutions for the test case described in
section~\ref{sec:confined_domain}. The errors in $B$ and $\div B$ at the final
time are given in Table~\ref{tab:confined_domain_divcleaning_N_40_order_2} and
Table~\ref{tab:confined_domain_divcleaning_N_40_order_4} for the SBP operator
with order of accuracy two and four, respectively. Both results have been obtained
using $N = 40$ nodes per space direction. Either no divergence cleaning procedure
has been applied or the projection using
\begin{itemize}
  \item
  the wide stencil operator with homogeneous Dirichlet boundary conditions
  (WS, D$_0$; section~\ref{sec:choose-div}),

  \item
  the narrow stencil operator with zero Dirichlet boundary conditions
  (NS, D$_0$; section~\ref{sec:choose-div-and-Delta}),

  \item
  the wide stencil operator and the least norm solution
  (WS, LN; section~\ref{sec:least-norm-solution}).
\end{itemize}
The absolute error threshold for the divergence has been set to $10^{-3}$ and
up to $50$ iterations of the CG method have been performed after each time step.

The second order schemes using the same discretisation forms for $\partial_j(u_i B_j)$
and $-\partial_j(u_j B_i)$ perform already very well for this test case. Since the
divergence norm is already negligible, the divergence cleaning procedure does not
influence the results in Table~\ref{tab:confined_domain_divcleaning_N_40_order_2}.

For the other two parameter choices, the wide and narrow stencil discretisation
of the Laplace operator with homogeneous Dirichlet boundary conditions yield
results that are very similar to the ones without any divergence cleaning procedure
at all. The wide stencil operator reduces the error a bit more.
Contrary, the least norm solution yields a significant reduction of both the
divergence errors and the magnetic energy. Nevertheless, the energy is still ca.
3--4 times larger than for the well-performing parameter combinations and the error
in the magnetic field is of course not of the order of machine accuracy.

For the schemes with interior order of accuracy four, there is an initial divergence
error of the magnetic field due to the projection of the initial condition onto
the grid. Hence, the steady state can be left if divergence cleaning procedures
area applied, as can be seen, e.g. in the first and fourth row of
Table~\ref{tab:confined_domain_divcleaning_N_40_order_4}. As before, the least
norm solutions performs better than the other divergence projection methods for
this test case and the latter two ones perform similar.
The trend of the results of the sixth order scheme is similar to the one of the
fourth order scheme. Thus, these results are not presented here in greater detail.

Additional tests using the Hall term and outflow boundary conditions as in
section~\ref{sec:hall_outflow} up to the final time $T=5$ have been performed.
The results for the second and fourth order SBP operators are given in
Table~\ref{tab:hall_outflow_divcleaning_N_40_order_2} and
Table~\ref{tab:hall_outflow_divcleaning_N_40_order_4}, respectively.
As before, the wide stencil least norm approach is the only cleaning procedure
resulting in a significant reduction of the divergence norm. The operators with
homogeneous boundary conditions can give some stabilisation, e.g. for the (not
recommended) choice \ref{itm:central-zero-central} (central, zero, central) for
the second order operator, cf. Table~\ref{tab:hall_outflow_divcleaning_N_40_order_2}.
For the other cases, they do not yield results that are significantly better than
the ones without divergence cleaning procedure.

The results for the fourth order SBP operator are bit different. There, the
wide stencil least norm approach results in a blow-up for three parameter
combinations while the other approaches yield a blow-up of the numerical solution
only for the choice \ref{itm:central-zero-central} (central, zero, central).
However, the divergence of the numerical solutions is reduced less than an order
of magnitude by the operators using homogeneous Dirichlet boundary conditions.

Since the divergence cleaning via projection approach is relatively costly, it
might be questionable whether it is worth the effort. In numerical experiments
presented here, a desired stabilisation could also be provided by certain choices
of the discrete forms given in Table~\ref{tab:parameter-choices}.
Of course, more detailed investigations studying the influence of the chosen
error thresholds and maximum number of iterations could be carried out. However,
the general results are not very sensitive to variations of these parameters and
other means to control the divergence and stability of numerical solutions will
be studied in the future.

\section{Summary and Discussion}
\label{sec:summary}

Building on the approach to stable problems in computational physics provided
recently by \citet{nordstrom2017roadmap}, initial boundary value problems for
the magnetic induction equation have been investigated at first at the continuous
level. Using the common approach to add a non-conservative source term involving
the divergence of the magnetic field, energy estimates have been obtained at first
for the linear induction equation. By applying summation by parts operators and
simultaneous approximation terms to impose boundary conditions weakly, these results
have been transferred to the semidiscrete level. Thus, several different
semidiscretisations of the induction equation have been shown to be energy stable
(section~\ref{sec:transport-term}). Additionally, the importance of boundary
conditions for the divergence constraint has been demonstrated in
section~\ref{sec:bounds-on-divB}. Moreover, novel outflow boundary conditions for
the nonlinear induction equation with Hall effect have been proposed, resulting
in an energy estimate (section~\ref{sec:Hall-term}).
Thereafter, divergence cleaning techniques using projections of the magnetic field
have been studied. Using SBP operators and paying special attention to the boundaries,
the energy of several known approaches has been investigated and a novel scheme
with some improved properties has been proposed (section~\ref{sec:div-constraint}).

Finally, all schemes have been compared using several numerical test cases
(section~\ref{sec:numerical-results}).
In general, schemes using a nonconservative source term allow an energy estimate
and perform better than the other ones in most test cases. While there might be
some circumstances where the other schemes yield better results, it seems to be
preferable to use methods allowing an energy estimate.
However, having an energy estimate is not enough to predict the performance of
a scheme on a finite grid. Indeed, the choice of the discrete form can influence
the results considerably for certain test cases. In particular, preserving the
anti-symmetry of $\nabla \times (u \times B)$ with respect to $u$ and $B$ can
be very important, as has been demonstrated in numerical experiments.
The novel outflow boundary condition for the nonlinear magnetic induction equation
with Hall effect results in stable schemes and a decaying magnetic energy, as
expected. Together with the other ingredients discussed and developed in this
article, it will be tested in more demanding and realistic applications in the
future.

The proposed schemes have been implemented in OpenCL and have been published
as open source software \cite{ranocha2018induction} with Matlab interface via
MatCL \cite{heinisch2018MatCL}.
Further research will involve the optimisation of the implementation, other
kinds of boundary conditions for the nonlinear induction equation with Hall
effect, and other means to control divergence errors of the magnetic field.

\section*{Acknowledgements}

The first author was supported by the German Research Foundation (DFG, Deutsche
Forschungsgemeinschaft) under Grant SO~363/14-1.
The authors acknowledge the North-German Supercomputing Alliance (HLRN, Norddeutscher
Verbund für Hoch- und Höchstleistungsrechnen) for providing HPC resources that
have contributed to the research results reported in this article.

\appendix

\section{High-Order SBP Operators}
\label{sec:SBP-HO}

Similarly to Example~\ref{ex:SBP-2}, there are higher order SBP operators. The
following fourth order accurate SBP operators are given in \cite{mattsson2004summation}.
The lower right corner can be obtained from the upper right corner via
$D_{N+1-i,N+1-j} = -D_{i,j}$.
\begin{equation}
\label{eq:SBP-4}
\begin{gathered}
  D
  =
  \frac{1}{\Delta x}
  \begin{pmatrix}
    \nicefrac{-24}{17} & \nicefrac{59}{34} & \nicefrac{-4}{17} & \nicefrac{-3}{34} \\
    \nicefrac{-1}{2} & 0 & \nicefrac{1}{2} \\
    \nicefrac{4}{43} & \nicefrac{-59}{86} & 0 & \nicefrac{59}{86} & \nicefrac{-4}{43} \\
    \nicefrac{3}{98} & 0 & \nicefrac{-59}{98} & 0 & \nicefrac{32}{49} & \nicefrac{-4}{49} \\
    && \nicefrac{1}{12} & \nicefrac{-2}{3} & 0 & \nicefrac{2}{3} & \nicefrac{-1}{12} \\
    &&& \ddots & \ddots & \ddots & \ddots & \ddots \\
  \end{pmatrix},
  \\
  M
  =
  \Delta x
  \diag{\frac{17}{48}, \frac{59}{48}, \frac{43}{48}, \frac{49}{48}, 1, \dots, 1,
        \frac{49}{48}, \frac{43}{48}, \frac{59}{48}, \frac{17}{48}}.
\end{gathered}
\end{equation}

Using the approach of section~\ref{sec:least-norm-solution} for divergence cleaning,
the adjoint operator $D^*$ and $D D^*$ have to be used.
Similarly to Example~\ref{ex:least-norm-solution}, they are
\begin{equation}
  D^*
  =
  M^{-1} D^T M
  =
  - D + M^{-1} E
  =
  \frac{1}{\Delta x}
  \begin{pmatrix}
    \nicefrac{-24}{17} & \nicefrac{-59}{34} & \nicefrac{4}{17} & \nicefrac{3}{34} \\
    \nicefrac{1}{2} & 0 & \nicefrac{-1}{2} \\
    \nicefrac{-4}{43} & \nicefrac{59}{86} & 0 & \nicefrac{-59}{86} & \nicefrac{4}{43} \\
    \nicefrac{-3}{98} & 0 & \nicefrac{59}{98} & 0 & \nicefrac{-32}{49} & \nicefrac{4}{49} \\
    && \nicefrac{-1}{12} & \nicefrac{2}{3} & 0 & \nicefrac{-2}{3} & \nicefrac{1}{12} \\
    &&& \ddots & \ddots & \ddots & \ddots & \ddots \\
  \end{pmatrix}
\end{equation}
and
\begin{footnotesize}
\setcounter{MaxMatrixCols}{12}
\setlength\arraycolsep{2pt}
\begin{multline}
  \Delta x^2 \; D D^*
  =
  \\
  \begin{pmatrix}
    \nicefrac{1756935}{608923} & \nicefrac{28438}{12427} & \nicefrac{-17743}{14161}
    & \nicefrac{458}{12427} & \nicefrac{1280}{35819} & \nicefrac{-6}{833}
    \\
    \nicefrac{482}{731} & \nicefrac{885}{731} & \nicefrac{-2}{17} & \nicefrac{-283}{731}
    & \nicefrac{2}{43}
    \\
    \nicefrac{-17743}{35819} & \nicefrac{-118}{731} & \nicefrac{84428}{107457}
    & \nicefrac{-118}{2193} & \nicefrac{-944}{2107} & \nicefrac{746}{6321} & \nicefrac{-1}{129}
    \\
    \nicefrac{458}{35819} & \nicefrac{-16697}{35819} & \nicefrac{-118}{2499}
    & \nicefrac{92188}{107457} & \nicefrac{-698}{6321} & \nicefrac{-64}{147}
    & \nicefrac{16}{147} & \nicefrac{-1}{147}
    \\
    \nicefrac{80}{6321} & \nicefrac{59}{1032} & \nicefrac{-59}{147} & \nicefrac{-349}{3096}
    & \nicefrac{271403}{303408} & \nicefrac{-97}{882} & \nicefrac{-4}{9} & \nicefrac{1}{9}
    & \nicefrac{-1}{144}
    \\
    \nicefrac{-1}{392} & 0 & \nicefrac{373}{3528} & \nicefrac{-4}{9} & \nicefrac{-97}{882}
    & \nicefrac{2123}{2352} & \nicefrac{-1}{9} & \nicefrac{-4}{9} & \nicefrac{1}{9}
    & \nicefrac{-1}{144}
    \\
    0 & 0 & \nicefrac{-1}{144} & \nicefrac{1}{9} & \nicefrac{-4}{9} & \nicefrac{-1}{9}
    & \nicefrac{65}{72} & \nicefrac{-1}{9} & \nicefrac{-4}{9} & \nicefrac{1}{9}
    & \nicefrac{-1}{144}
    \\
    &&& \ddots & \ddots & \ddots & \ddots & \ddots & \ddots & \ddots & \ddots
  \end{pmatrix}.
\end{multline}
\end{footnotesize}
This operator is symmetric and positive semidefinite.
Coefficients for the derivative operators with interior order of accuracy six
can be found in \cite{mattsson2004summation}. All coefficients can also be found
in \cite{ranocha2018induction}.

\section{Results of Numerical Experiments}
\label{sec:numerical-results-appendix}

Here, additional data from numerical experiments described in
section~\ref{sec:numerical-results} are given.

\begin{table}[!ht]
\centering
  \caption{Errors and divergence norms of numerical solutions of the linear
           induction equation \eqref{eq:induction-transport-conservative} with
           analytical solution \eqref{eq:rotation_3D} using $N=40$ nodes per
           direction and the fourth order SBP operator.}
  \label{tab:rotation_3D_N_40_order_4}
  \begin{tabular*}{\linewidth}{@{\extracolsep{\fill}}*5c@{}}
    \toprule
    $\partial_j(u_i B_j)$ & source term & $-\partial_j(u_j B_i)$ & $\epsilon_{B}$ & $\epsilon_{\div B}$
    \\
    \midrule
    product &  zero    &  product & 1.10e+06 & 9.18e+06
    \\
    product &  zero    &  split   & 1.10e+06 & 9.18e+06
    \\
    product &  zero    &  central & 1.10e+06 & 9.18e+06
    \\
    product &  split   &  product & 2.04e-02 & 8.42e-02
    \\
    product &  split   &  split   & 2.04e-02 & 8.42e-02
    \\
    product &  split   &  central & 2.04e-02 & 8.42e-02
    \\
    product &  central &  product & 2.01e-02 & 5.66e-02
    \\
    product &  central &  split   & 2.01e-02 & 5.66e-02
    \\
    product &  central &  central & 2.01e-02 & 5.66e-02
    \\
    split   &  zero    &  product & 2.26e+02 & 1.43e+03
    \\
    split   &  zero    &  split   & 2.26e+02 & 1.43e+03
    \\
    split   &  zero    &  central & 2.26e+02 & 1.43e+03
    \\
    split   &  split   &  product & 2.01e-02 & 5.66e-02
    \\
    split   &  split   &  split   & 2.01e-02 & 5.66e-02
    \\
    split   &  split   &  central & 2.01e-02 & 5.66e-02
    \\
    split   &  central &  product & 1.99e-02 & 2.86e-02
    \\
    split   &  central &  split   & 1.99e-02 & 2.86e-02
    \\
    split   &  central &  central & 1.99e-02 & 2.86e-02
    \\
    central &  zero    &  product & 1.98e-02 & 4.87e-03
    \\
    central &  zero    &  split   & 1.98e-02 & 4.87e-03
    \\
    central &  zero    &  central & 1.98e-02 & 4.87e-03
    \\
    central &  split   &  product & 1.99e-02 & 2.86e-02
    \\
    central &  split   &  split   & 1.99e-02 & 2.86e-02
    \\
    central &  split   &  central & 1.99e-02 & 2.86e-02
    \\
    central &  central &  product & 1.98e-02 & 3.04e-03
    \\
    central &  central &  split   & 1.98e-02 & 3.04e-03
    \\
    central &  central &  central & 1.98e-02 & 3.04e-03
    \\
    \bottomrule
  \end{tabular*}
\end{table}

\begin{sidewaystable}
\centering
\addtolength\tabcolsep{-1.5pt}
  \caption{Results of convergence experiments for the linear induction equation
           \eqref{eq:induction-transport-conservative} with analytical solution
           \eqref{eq:rotation_3D}, SBP operators of different order, and the
           choice of discrete forms as in \ref{itm:central-zero-central}
           (central, zero, central).}
  \label{tab:rotation_3D-central-zero-central}
  \begin{tabular*}{\linewidth}{@{\extracolsep{\fill}}c|*5c|*5c|*5c@{}}
    \toprule
    & \multicolumn{5}{c}{Interior Order $2$}
    & \multicolumn{5}{c}{Interior Order $4$}
    & \multicolumn{5}{c}{Interior Order $6$}
    \\
    $N$
    & $\epsilon_{B}$ & EOC & $\epsilon_{\div B}$ & EOC & Runtime
    & $\epsilon_{B}$ & EOC & $\epsilon_{\div B}$ & EOC & Runtime
    & $\epsilon_{B}$ & EOC & $\epsilon_{\div B}$ & EOC & Runtime
    \\
    \midrule
    $ 40$
    & 1.72e-01 &       & 2.77e-02 &       & 7.18e-01
    & 1.98e-02 &       & 4.87e-03 &       & 1.02e+00
    & 3.79e-03 &       & 8.51e-03 &       & 1.34e+00
    \\
    $ 80$
    & 5.37e-02 &  1.68 & 7.14e-03 &  1.95 & 1.10e+01
    & 1.27e-03 &  3.96 & 5.82e-04 &  3.07 & 1.62e+01
    & 1.18e-04 &  5.01 & 6.87e-04 &  3.63 & 2.36e+01
    \\
    $160$
    & 1.36e-02 &  1.98 & 2.02e-03 &  1.82 & 1.69e+02
    & 7.88e-05 &  4.01 & 8.37e-05 &  2.80 & 3.03e+02
    & 4.63e-06 &  4.67 & 5.01e-05 &  3.78 & 4.20e+02
    \\
    $320$
    & 3.39e-03 &  2.01 & 6.23e-04 &  1.70 & 2.92e+03
    & 4.90e-06 &  4.01 & 1.35e-05 &  2.63 & 5.69e+03
    & 2.39e-07 &  4.27 & 3.88e-06 &  3.69 & 9.47e+03
    \\
    \bottomrule
  \end{tabular*}

  \bigskip\medskip

  \caption{Results of convergence experiments for the linear induction equation
           \eqref{eq:induction-transport-conservative} with analytical solution
           \eqref{eq:rotation_3D}, SBP operators of different order, and the
           choice of discrete forms as in \ref{itm:central-central-central}
           (central, central, central).}
  \label{tab:rotation_3D-central-central-central}
  \begin{tabular*}{\linewidth}{@{\extracolsep{\fill}}c|*5c|*5c|*5c@{}}
    \toprule
    & \multicolumn{5}{c}{Interior Order $2$}
    & \multicolumn{5}{c}{Interior Order $4$}
    & \multicolumn{5}{c}{Interior Order $6$}
    \\
    $N$
    & $\epsilon_{B}$ & EOC & $\epsilon_{\div B}$ & EOC & Runtime
    & $\epsilon_{B}$ & EOC & $\epsilon_{\div B}$ & EOC & Runtime
    & $\epsilon_{B}$ & EOC & $\epsilon_{\div B}$ & EOC & Runtime
    \\
    \midrule
    $ 40$
    & 1.72e-01 &       & 3.74e-02 &       & 7.31e-01
    & 1.98e-02 &       & 3.04e-03 &       & 1.06e+00
    & 3.70e-03 &       & 9.83e-03 &       & 1.45e+00
    \\
    $ 80$
    & 5.37e-02 &  1.68 & 7.13e-03 &  2.39 & 1.13e+01
    & 1.27e-03 &  3.96 & 2.80e-04 &  3.44 & 1.67e+01
    & 1.13e-04 &  5.04 & 6.56e-04 &  3.90 & 2.48e+01
    \\
    $160$
    & 1.36e-02 &  1.98 & 1.99e-03 &  1.84 & 1.93e+02
    & 7.89e-05 &  4.01 & 3.81e-05 &  2.88 & 3.21e+02
    & 4.91e-06 &  4.52 & 4.40e-05 &  3.90 & 4.40e+02
    \\
    $320$
    & 3.39e-03 &  2.01 & 5.97e-04 &  1.73 & 2.95e+03
    & 4.93e-06 &  4.00 & 5.42e-06 &  2.81 & 5.75e+03
    & 2.68e-07 &  4.19 & 3.33e-06 &  3.73 & 8.74e+03
    \\
    \bottomrule
  \end{tabular*}

  \bigskip\medskip

  \caption{Results of convergence experiments for the linear induction equation
           \eqref{eq:induction-transport-conservative} with analytical solution
           \eqref{eq:rotation_3D}, SBP operators of different order, and the
           choice of discrete forms as in \ref{itm:split-central-split}
           (split, central, split).}
  \label{tab:rotation_3D-split-central-split}
  \begin{tabular*}{\linewidth}{@{\extracolsep{\fill}}c|*5c|*5c|*5c@{}}
    \toprule
    & \multicolumn{5}{c}{Interior Order $2$}
    & \multicolumn{5}{c}{Interior Order $4$}
    & \multicolumn{5}{c}{Interior Order $6$}
    \\
    $N$
    & $\epsilon_{B}$ & EOC & $\epsilon_{\div B}$ & EOC & Runtime
    & $\epsilon_{B}$ & EOC & $\epsilon_{\div B}$ & EOC & Runtime
    & $\epsilon_{B}$ & EOC & $\epsilon_{\div B}$ & EOC & Runtime
    \\
    \midrule
    $ 40$
    & 1.75e-01 &       & 2.99e-01 &       & 7.42e-01
    & 1.99e-02 &       & 2.86e-02 &       & 1.12e+00
    & 3.30e-03 &       & 2.03e-02 &       & 1.55e+00
    \\
    $ 80$
    & 5.41e-02 &  1.69 & 7.72e-02 &  1.95 & 1.14e+01
    & 1.28e-03 &  3.96 & 1.82e-03 &  3.97 & 1.74e+01
    & 8.86e-05 &  5.22 & 1.20e-03 &  4.08 & 2.66e+01
    \\
    $160$
    & 1.37e-02 &  1.98 & 1.94e-02 &  2.00 & 1.79e+02
    & 7.91e-05 &  4.01 & 1.17e-04 &  3.96 & 3.15e+02
    & 3.73e-06 &  4.57 & 6.73e-05 &  4.15 & 4.69e+02
    \\
    $320$
    & 3.41e-03 &  2.01 & 4.85e-03 &  2.00 & 3.12e+03
    & 4.93e-06 &  4.00 & 8.82e-06 &  3.73 & 5.83e+03
    & 2.04e-07 &  4.20 & 5.02e-06 &  3.74 & 9.42e+03
    \\
    \bottomrule
  \end{tabular*}
\end{sidewaystable}

\begin{sidewaystable}
\centering
\addtolength\tabcolsep{-1.5pt}
  \caption{Results of convergence experiments for the linear induction equation
           \eqref{eq:induction-transport-conservative} with analytical solution
           \eqref{eq:rotation_3D}, SBP operators of different order, and the
           choice of discrete forms as in \ref{itm:product-central-product}
           (product, central, product).}
  \label{tab:rotation_3D-product-central-product}
  \begin{tabular*}{\linewidth}{@{\extracolsep{\fill}}c|*5c|*5c|*5c@{}}
    \toprule
    & \multicolumn{5}{c}{Interior Order $2$}
    & \multicolumn{5}{c}{Interior Order $4$}
    & \multicolumn{5}{c}{Interior Order $6$}
    \\
    $N$
    & $\epsilon_{B}$ & EOC & $\epsilon_{\div B}$ & EOC & Runtime
    & $\epsilon_{B}$ & EOC & $\epsilon_{\div B}$ & EOC & Runtime
    & $\epsilon_{B}$ & EOC & $\epsilon_{\div B}$ & EOC & Runtime
    \\
    \midrule
    $ 40$
    & 1.74e-01 &       & 5.69e-01 &       & 7.38e-01
    & 2.01e-02 &       & 5.66e-02 &       & 1.07e+00
    & 3.08e-03 &       & 2.09e-02 &       & 1.50e+00
    \\
    $ 80$
    & 5.49e-02 &  1.66 & 1.53e-01 &  1.89 & 1.14e+01
    & 1.29e-03 &  3.96 & 3.61e-03 &  3.97 & 1.72e+01
    & 6.86e-05 &  5.49 & 1.01e-03 &  4.38 & 2.56e+01
    \\
    $160$
    & 1.39e-02 &  1.98 & 3.85e-02 &  1.99 & 1.80e+02
    & 8.01e-05 &  4.01 & 2.25e-04 &  4.00 & 3.02e+02
    & 2.84e-06 &  4.59 & 4.70e-05 &  4.42 & 4.53e+02
    \\
    $320$
    & 3.47e-03 &  2.01 & 9.63e-03 &  2.00 & 3.09e+03
    & 4.98e-06 &  4.01 & 1.46e-05 &  3.94 & 5.73e+03
    & 1.62e-07 &  4.13 & 2.86e-06 &  4.04 & 9.31e+03
    \\
    \bottomrule
  \end{tabular*}

  \bigskip\medskip

  \caption{Results of convergence experiments for the linear induction equation
           \eqref{eq:induction-transport-conservative} with analytical solution
           \eqref{eq:rotation_3D}, SBP operators of different order, and the
           choice of discrete forms as in \ref{itm:product-central-split}
           (product, central, split).}
  \label{tab:rotation_3D-product-central-split}
  \begin{tabular*}{\linewidth}{@{\extracolsep{\fill}}c|*5c|*5c|*5c@{}}
    \toprule
    & \multicolumn{5}{c}{Interior Order $2$}
    & \multicolumn{5}{c}{Interior Order $4$}
    & \multicolumn{5}{c}{Interior Order $6$}
    \\
    $N$
    & $\epsilon_{B}$ & EOC & $\epsilon_{\div B}$ & EOC & Runtime
    & $\epsilon_{B}$ & EOC & $\epsilon_{\div B}$ & EOC & Runtime
    & $\epsilon_{B}$ & EOC & $\epsilon_{\div B}$ & EOC & Runtime
    \\
    \midrule
    $ 40$
    & 1.74e-01 &       & 5.69e-01 &       & 7.59e-01
    & 2.01e-02 &       & 5.66e-02 &       & 1.10e+00
    & 3.08e-03 &       & 2.09e-02 &       & 1.52e+00
    \\
    $ 80$
    & 5.49e-02 &  1.66 & 1.53e-01 &  1.89 & 1.16e+01
    & 1.29e-03 &  3.96 & 3.61e-03 &  3.97 & 1.73e+01
    & 6.86e-05 &  5.49 & 1.01e-03 &  4.38 & 2.59e+01
    \\
    $160$
    & 1.39e-02 &  1.98 & 3.85e-02 &  1.99 & 1.87e+02
    & 8.01e-05 &  4.01 & 2.25e-04 &  4.00 & 3.16e+02
    & 2.84e-06 &  4.59 & 4.70e-05 &  4.42 & 4.39e+02
    \\
    $320$
    & 3.47e-03 &  2.01 & 9.63e-03 &  2.00 & 3.03e+03
    & 4.98e-06 &  4.01 & 1.46e-05 &  3.94 & 5.93e+03
    & 1.62e-07 &  4.13 & 2.86e-06 &  4.04 & 7.55e+03
    \\
    \bottomrule
  \end{tabular*}

  \bigskip\medskip

  \caption{Results of convergence experiments for the linear induction equation
           \eqref{eq:induction-transport-conservative} with analytical solution
           \eqref{eq:rotation_3D}, SBP operators of different order, and the
           choice of discrete forms as in \ref{itm:product-central-central}
           (product, central, central).}
  \label{tab:rotation_3D-product-central-central}
  \begin{tabular*}{\linewidth}{@{\extracolsep{\fill}}c|*5c|*5c|*5c@{}}
    \toprule
    & \multicolumn{5}{c}{Interior Order $2$}
    & \multicolumn{5}{c}{Interior Order $4$}
    & \multicolumn{5}{c}{Interior Order $6$}
    \\
    $N$
    & $\epsilon_{B}$ & EOC & $\epsilon_{\div B}$ & EOC & Runtime
    & $\epsilon_{B}$ & EOC & $\epsilon_{\div B}$ & EOC & Runtime
    & $\epsilon_{B}$ & EOC & $\epsilon_{\div B}$ & EOC & Runtime
    \\
    \midrule
    $ 40$
    & 1.74e-01 &       & 5.69e-01 &       & 7.39e-01
    & 2.01e-02 &       & 5.66e-02 &       & 1.08e+00
    & 3.08e-03 &       & 2.09e-02 &       & 1.47e+00
    \\
    $ 80$
    & 5.49e-02 &  1.66 & 1.53e-01 &  1.89 & 1.14e+01
    & 1.29e-03 &  3.96 & 3.61e-03 &  3.97 & 1.72e+01
    & 6.86e-05 &  5.49 & 1.01e-03 &  4.38 & 2.54e+01
    \\
    $160$
    & 1.39e-02 &  1.98 & 3.85e-02 &  1.99 & 1.78e+02
    & 8.01e-05 &  4.01 & 2.25e-04 &  4.00 & 3.20e+02
    & 2.84e-06 &  4.59 & 4.70e-05 &  4.42 & 4.41e+02
    \\
    $320$
    & 3.47e-03 &  2.01 & 9.63e-03 &  2.00 & 2.81e+03
    & 4.98e-06 &  4.01 & 1.46e-05 &  3.94 & 5.82e+03
    & 1.62e-07 &  4.13 & 2.86e-06 &  4.04 & 8.02e+03
    \\
    \bottomrule
  \end{tabular*}
\end{sidewaystable}

\begin{table}[!ht]
\centering
  \caption{Errors and divergence norms of numerical solutions of the linear
           induction equation \eqref{eq:induction-transport-conservative} with
           analytical solution \eqref{eq:confined_domain} using $N=40$ nodes per
           direction and the fourth order SBP operator.}
  \label{tab:confined_domain_N_40_order_4}
  \begin{tabular*}{\linewidth}{@{\extracolsep{\fill}}*5c@{}}
    \toprule
    $\partial_j(u_i B_j)$ & source term & $-\partial_j(u_j B_i)$ & $\epsilon_{B}$ & $\epsilon_{\div B}$
    \\
    \midrule
    product & zero    & product & 1.55e-12 & 1.77e-03
    \\
    product & zero    & split   & 9.32e+03 & 4.88e+05
    \\
    product & zero    & central & 4.63e+09 & 1.98e+11
    \\
    product & split   & product & 2.74e+02 & 8.81e+02
    \\
    product & split   & split   & 1.39e+04 & 9.19e+05
    \\
    product & split   & central & 3.07e+09 & 2.30e+11
    \\
    product & central & product & 6.65e-02 & 5.66e-01
    \\
    product & central & split   & 1.33e+00 & 4.85e+00
    \\
    product & central & central & 1.64e+03 & 1.09e+05
    \\
    split   & zero    & product & 2.88e+05 & 4.90e+06
    \\
    split   & zero    & split   & 2.98e-17 & 1.77e-03
    \\
    split   & zero    & central & 1.44e+03 & 7.53e+04
    \\
    split   & split   & product & 6.65e-02 & 5.66e-01
    \\
    split   & split   & split   & 1.33e+00 & 4.85e+00
    \\
    split   & split   & central & 1.64e+03 & 1.09e+05
    \\
    split   & central & product & 3.80e+01 & 1.17e+02
    \\
    split   & central & split   & 3.50e-03 & 2.93e-02
    \\
    split   & central & central & 3.75e-01 & 1.37e+00
    \\
    central & zero    & product & 4.03e+07 & 1.11e+09
    \\
    central & zero    & split   & 3.20e+01 & 7.43e+02
    \\
    central & zero    & central & 0.00e+00 & 1.77e-03
    \\
    central & split   & product & 3.80e+01 & 1.17e+02
    \\
    central & split   & split   & 3.50e-03 & 2.93e-02
    \\
    central & split   & central & 3.75e-01 & 1.37e+00
    \\
    central & central & product & 2.77e+04 & 5.72e+05
    \\
    central & central & split   & 2.54e-01 & 1.21e+00
    \\
    central & central & central & 4.09e-03 & 3.68e-02
    \\
    \bottomrule
  \end{tabular*}
\end{table}

\begin{sidewaystable}
\centering
\addtolength\tabcolsep{-1.5pt}
  \caption{Results of convergence experiments for the linear induction equation
           \eqref{eq:induction-transport-conservative} with analytical solution
           \eqref{eq:confined_domain}, SBP operators of different order, and the
           choice of discrete forms as in \ref{itm:central-zero-central}
           (central, zero, central).}
  \label{tab:confined_domain-central-zero-central}
  \begin{tabular*}{\linewidth}{@{\extracolsep{\fill}}c|*5c|*5c|*5c@{}}
    \toprule
    & \multicolumn{5}{c}{Interior Order $2$}
    & \multicolumn{5}{c}{Interior Order $4$}
    & \multicolumn{5}{c}{Interior Order $6$}
    \\
    $N$
    & $\epsilon_{B}$ & EOC & $\epsilon_{\div B}$ & EOC & Runtime
    & $\epsilon_{B}$ & EOC & $\epsilon_{\div B}$ & EOC & Runtime
    & $\epsilon_{B}$ & EOC & $\epsilon_{\div B}$ & EOC & Runtime
    \\
    \midrule
    $ 40$
    & 1.87e-30 &       & 4.10e-15 &       & 6.17e-01
    & 2.73e-30 &       & 1.77e-03 &       & 8.02e-01
    & 4.33e-30 &       & 7.77e-05 &       & 1.11e+00
    \\
    $ 80$
    & 8.13e-31 &  1.20 & 7.28e-15 & -0.83 & 9.33e+00
    & 9.12e-31 &  1.58 & 3.04e-04 &  2.54 & 1.33e+01
    & 1.21e-30 &  1.84 & 3.26e-06 &  4.57 & 2.02e+01
    \\
    $160$
    & 1.18e-30 & -0.53 & 9.24e-15 & -0.34 & 1.53e+02
    & 1.35e-30 & -0.56 & 5.30e-05 &  2.52 & 2.70e+02
    & 1.56e-30 & -0.37 & 1.40e-07 &  4.54 & 3.66e+02
    \\
    $320$
    & 2.11e-29 & -4.17 & 2.24e-14 & -1.28 & 2.49e+03
    & 2.58e-29 & -4.26 & 9.29e-06 &  2.51 & 4.55e+03
    & 2.76e-29 & -4.14 & 6.12e-09 &  4.52 & 8.47e+03
    \\
    \bottomrule
  \end{tabular*}

  \bigskip\medskip

  \caption{Results of convergence experiments for the linear induction equation
           \eqref{eq:induction-transport-conservative} with analytical solution
           \eqref{eq:confined_domain}, SBP operators of different order, and the
           choice of discrete forms as in \ref{itm:central-central-central}
           (central, central, central).}
  \label{tab:confined_domain-central-central-central}
  \begin{tabular*}{\linewidth}{@{\extracolsep{\fill}}c|*5c|*5c|*5c@{}}
    \toprule
    & \multicolumn{5}{c}{Interior Order $2$}
    & \multicolumn{5}{c}{Interior Order $4$}
    & \multicolumn{5}{c}{Interior Order $6$}
    \\
    $N$
    & $\epsilon_{B}$ & EOC & $\epsilon_{\div B}$ & EOC & Runtime
    & $\epsilon_{B}$ & EOC & $\epsilon_{\div B}$ & EOC & Runtime
    & $\epsilon_{B}$ & EOC & $\epsilon_{\div B}$ & EOC & Runtime
    \\
    \midrule
    $ 40$
    & 9.74e-16 &       & 1.36e-14 &       & 6.08e-01
    & 4.09e-03 &       & 3.68e-02 &       & 8.50e-01
    & 3.55e-04 &       & 4.86e-03 &       & 1.20e+00
    \\
    $ 80$
    & 1.49e-15 & -0.61 & 5.34e-14 & -1.98 & 9.41e+00
    & 6.02e-04 &  2.76 & 1.41e-02 &  1.39 & 1.39e+01
    & 1.69e-05 &  4.39 & 3.42e-04 &  3.83 & 2.12e+01
    \\
    $160$
    & 1.57e-15 & -0.08 & 1.55e-13 & -1.54 & 1.43e+02
    & 6.55e-05 &  3.20 & 5.46e-03 &  1.37 & 2.74e+02
    & 5.40e-07 &  4.97 & 2.85e-05 &  3.59 & 3.81e+02
    \\
    $320$
    & 5.13e-15 & -1.70 & 9.31e-13 & -2.59 & 2.60e+03
    & 7.71e-06 &  3.09 & 1.88e-03 &  1.54 & 4.87e+03
    & 1.30e-08 &  5.38 & 2.45e-06 &  3.54 & 8.58e+03
    \\
    \bottomrule
  \end{tabular*}

  \bigskip\medskip

  \caption{Results of convergence experiments for the linear induction equation
           \eqref{eq:induction-transport-conservative} with analytical solution
           \eqref{eq:confined_domain}, SBP operators of different order, and the
           choice of discrete forms as in \ref{itm:split-central-split}
           (split, central, split).}
  \label{tab:confined_domain-split-central-split}
  \begin{tabular*}{\linewidth}{@{\extracolsep{\fill}}c|*5c|*5c|*5c@{}}
    \toprule
    & \multicolumn{5}{c}{Interior Order $2$}
    & \multicolumn{5}{c}{Interior Order $4$}
    & \multicolumn{5}{c}{Interior Order $6$}
    \\
    $N$
    & $\epsilon_{B}$ & EOC & $\epsilon_{\div B}$ & EOC & Runtime
    & $\epsilon_{B}$ & EOC & $\epsilon_{\div B}$ & EOC & Runtime
    & $\epsilon_{B}$ & EOC & $\epsilon_{\div B}$ & EOC & Runtime
    \\
    \midrule
    $ 40$
    & 1.48e-15 &       & 2.28e-14 &       & 5.86e-01
    & 3.50e-03 &       & 2.93e-02 &       & 9.10e-01
    & 1.17e-04 &       & 1.80e-03 &       & 1.27e+00
    \\
    $ 80$
    & 2.26e-15 & -0.61 & 7.59e-14 & -1.74 & 9.90e+00
    & 5.27e-04 &  2.73 & 6.76e-03 &  2.12 & 1.47e+01
    & 6.10e-06 &  4.27 & 1.14e-04 &  3.98 & 2.26e+01
    \\
    $160$
    & 2.07e-15 &  0.12 & 1.91e-13 & -1.33 & 1.51e+02
    & 6.59e-05 &  3.00 & 1.45e-03 &  2.22 & 2.84e+02
    & 3.02e-07 &  4.33 & 7.23e-06 &  3.98 & 4.08e+02
    \\
    $320$
    & 6.84e-15 & -1.72 & 1.10e-12 & -2.52 & 2.57e+03
    & 8.96e-06 &  2.88 & 3.08e-04 &  2.23 & 5.38e+03
    & 1.19e-08 &  4.66 & 4.54e-07 &  3.99 & 8.90e+03
    \\
    \bottomrule
  \end{tabular*}
\end{sidewaystable}

\begin{sidewaystable}
\centering
\addtolength\tabcolsep{-2.5pt}
  \caption{Results of convergence experiments for the linear induction equation
           \eqref{eq:induction-transport-conservative} with analytical solution
           \eqref{eq:confined_domain}, SBP operators of different order, and the
           choice of discrete forms as in \ref{itm:product-central-product}
           (product, central, product).}
  \label{tab:confined_domain-product-central-product}
  \begin{tabular*}{\linewidth}{@{\extracolsep{\fill}}c|*5c|*5c|*5c@{}}
    \toprule
    & \multicolumn{5}{c}{Interior Order $2$}
    & \multicolumn{5}{c}{Interior Order $4$}
    & \multicolumn{5}{c}{Interior Order $6$}
    \\
    $N$
    & $\epsilon_{B}$ & EOC & $\epsilon_{\div B}$ & EOC & Runtime
    & $\epsilon_{B}$ & EOC & $\epsilon_{\div B}$ & EOC & Runtime
    & $\epsilon_{B}$ & EOC & $\epsilon_{\div B}$ & EOC & Runtime
    \\
    \midrule
    $ 40$
    & 7.35e-14 &       & 8.21e-13 &       & 6.21e-01
    & 6.65e-02 &       & 5.66e-01 &       & 8.75e-01
    & 1.59e-03 &       & 1.41e-02 &       & 1.25e+00
    \\
    $ 80$
    & 1.86e-13 & -1.34 & 1.28e-12 & -0.65 & 9.76e+00
    & 1.01e-02 &  2.72 & 1.68e-01 &  1.75 & 1.45e+01
    & 5.99e-05 &  4.73 & 8.55e-04 &  4.04 & 2.17e+01
    \\
    $160$
    & 3.89e-13 & -1.07 & 3.09e-12 & -1.27 & 1.59e+02
    & 1.58e-03 &  2.67 & 5.48e-02 &  1.62 & 2.78e+02
    & 2.76e-06 &  4.44 & 5.92e-05 &  3.85 & 3.91e+02
    \\
    $320$
    & 6.36e-13 & -0.71 & 1.34e-11 & -2.11 & 2.62e+03
    & 2.83e-04 &  2.48 & 1.75e-02 &  1.65 & 5.08e+03
    & 1.37e-07 &  4.33 & 4.67e-06 &  3.66 & 8.55e+03
    \\
    \bottomrule
  \end{tabular*}

  \bigskip\medskip

  \caption{Results of convergence experiments for the linear induction equation
           \eqref{eq:induction-transport-conservative} with analytical solution
           \eqref{eq:confined_domain}, SBP operators of different order, and the
           choice of discrete forms as in \ref{itm:product-central-split}
           (product, central, split).}
  \label{tab:confined_domain-product-central-split}
  \begin{tabular*}{\linewidth}{@{\extracolsep{\fill}}c|*5c|*5c|*5c@{}}
    \toprule
    & \multicolumn{5}{c}{Interior Order $2$}
    & \multicolumn{5}{c}{Interior Order $4$}
    & \multicolumn{5}{c}{Interior Order $6$}
    \\
    $N$
    & $\epsilon_{B}$ & EOC & $\epsilon_{\div B}$ & EOC & Runtime
    & $\epsilon_{B}$ & EOC & $\epsilon_{\div B}$ & EOC & Runtime
    & $\epsilon_{B}$ & EOC & $\epsilon_{\div B}$ & EOC & Runtime
    \\
    \midrule
    $ 40$
    & 8.49e+01 &       & 2.86e+02 &       & 6.41e-01
    & 1.33e+00 &       & 4.85e+00 &       & 8.98e-01
    & 1.73e+00 &       & 1.23e+01 &       & 1.30e+00
    \\
    $ 80$
    & 4.05e+01 &  1.07 & 1.68e+02 &  0.77 & 9.45e+00
    & 1.61e-01 &  3.04 & 6.97e-01 &  2.80 & 1.44e+01
    & 2.44e-01 &  2.82 & 1.22e+00 &  3.33 & 2.17e+01
    \\
    $160$
    & 1.85e+01 &  1.13 & 1.08e+02 &  0.63 & 1.49e+02
    & 1.95e-02 &  3.04 & 1.83e-01 &  1.93 & 2.80e+02
    & 3.00e-02 &  3.02 & 1.02e-01 &  3.58 & 4.02e+02
    \\
    $320$
    & 8.31e+00 &  1.16 & 7.43e+01 &  0.54 & 2.62e+03
    & 2.30e-03 &  3.09 & 5.61e-02 &  1.71 & 5.22e+03
    & 3.61e-03 &  3.06 & 7.30e-03 &  3.80 & 8.99e+03
    \\
    \bottomrule
  \end{tabular*}

  \bigskip\medskip

  \caption{Results of convergence experiments for the linear induction equation
           \eqref{eq:induction-transport-conservative} with analytical solution
           \eqref{eq:confined_domain}, SBP operators of different order, and the
           choice of discrete forms as in \ref{itm:product-central-central}
           (product, central, central).}
  \label{tab:confined_domain-product-central-central}
  \begin{tabular*}{\linewidth}{@{\extracolsep{\fill}}c|*5c|*5c|*5c@{}}
    \toprule
    & \multicolumn{5}{c}{Interior Order $2$}
    & \multicolumn{5}{c}{Interior Order $4$}
    & \multicolumn{5}{c}{Interior Order $6$}
    \\
    $N$
    & $\epsilon_{B}$ & EOC & $\epsilon_{\div B}$ & EOC & Runtime
    & $\epsilon_{B}$ & EOC & $\epsilon_{\div B}$ & EOC & Runtime
    & $\epsilon_{B}$ & EOC & $\epsilon_{\div B}$ & EOC & Runtime
    \\
    \midrule
    $ 40$
    & 3.52e+04 &       & 1.68e+06 &       & 6.19e-01
    & 1.64e+03 &       & 1.09e+05 &       & 8.57e-01
    & 8.33e+03 &       & 6.43e+05 &       & 1.22e+00
    \\
    $ 80$
    & 2.36e+04 &  0.58 & 2.29e+06 & -0.44 & 9.60e+00
    & 2.61e+02 &  2.65 & 3.50e+04 &  1.63 & 1.40e+01
    & 1.44e+03 &  2.53 & 2.25e+05 &  1.52 & 2.10e+01
    \\
    $160$
    & 1.51e+04 &  0.65 & 2.93e+06 & -0.36 & 1.53e+02
    & 4.06e+01 &  2.68 & 1.10e+04 &  1.68 & 2.75e+02
    & 2.30e+02 &  2.64 & 7.24e+04 &  1.64 & 4.09e+02
    \\
    $320$
    & 8.93e+03 &  0.75 & 3.49e+06 & -0.25 & 2.63e+03
    & 5.85e+00 &  2.79 & 3.17e+03 &  1.79 & 5.16e+03
    & 3.38e+01 &  2.77 & 2.13e+04 &  1.76 & 8.61e+03
    \\
    \bottomrule
  \end{tabular*}
\end{sidewaystable}

\begin{sidewaystable}
\centering
\addtolength\tabcolsep{-1.5pt}
  \caption{Results of convergence experiments for the nonlinear induction equation
           \eqref{eq:induction-transport-conservative-Hall} with analytical solution
           \eqref{eq:hall_periodic}, SBP operators of different order, and the
           choice of discrete forms as in \ref{itm:central-zero-central}
           (central, zero, central).}
  \label{tab:hall_periodic-central-zero-central}
  \begin{tabular*}{\linewidth}{@{\extracolsep{\fill}}c|*5c|*5c|*5c@{}}
    \toprule
    & \multicolumn{5}{c}{Interior Order $2$}
    & \multicolumn{5}{c}{Interior Order $4$}
    & \multicolumn{5}{c}{Interior Order $6$}
    \\
    $N$
    & $\epsilon_{B}$ & EOC & $\epsilon_{\div B}$ & EOC & Runtime
    & $\epsilon_{B}$ & EOC & $\epsilon_{\div B}$ & EOC & Runtime
    & $\epsilon_{B}$ & EOC & $\epsilon_{\div B}$ & EOC & Runtime
    \\
    \midrule
    $ 40$
    & 2.02e-02 &       & 2.55e-13 &       & 5.59e+00
    & 9.98e-05 &       & 3.44e-13 &       & 6.85e+00
    & 5.27e-07 &       & 3.94e-13 &       & 8.26e+00
    \\
    $ 60$
    & 9.00e-03 &  1.99 & 5.76e-13 & -2.01 & 4.18e+01
    & 1.97e-05 &  4.00 & 7.78e-13 & -2.01 & 5.34e+01
    & 4.64e-08 &  5.99 & 8.83e-13 & -1.99 & 6.65e+01
    \\
    $ 80$
    & 5.07e-03 &  2.00 & 1.02e-12 & -2.00 & 1.97e+02
    & 6.25e-06 &  4.00 & 1.38e-12 & -1.99 & 2.26e+02
    & 8.26e-09 &  6.00 & 1.57e-12 & -2.00 & 3.16e+02
    \\
    $100$
    & 3.24e-03 &  2.00 & 1.61e-12 & -2.02 & 5.70e+02
    & 2.56e-06 &  4.00 & 2.15e-12 & -2.00 & 7.69e+02
    & 2.17e-09 &  6.00 & 2.46e-12 & -2.01 & 9.57e+02
    \\
    \bottomrule
  \end{tabular*}

  \bigskip\medskip

  \caption{Results of convergence experiments for the nonlinear induction equation
           \eqref{eq:induction-transport-conservative-Hall} with analytical solution
           \eqref{eq:hall_periodic}, SBP operators of different order, and the
           choice of discrete forms as in \ref{itm:central-central-central}
           (central, central, central).}
  \label{tab:hall_periodic-central-central-central}
  \begin{tabular*}{\linewidth}{@{\extracolsep{\fill}}c|*5c|*5c|*5c@{}}
    \toprule
    & \multicolumn{5}{c}{Interior Order $2$}
    & \multicolumn{5}{c}{Interior Order $4$}
    & \multicolumn{5}{c}{Interior Order $6$}
    \\
    $N$
    & $\epsilon_{B}$ & EOC & $\epsilon_{\div B}$ & EOC & Runtime
    & $\epsilon_{B}$ & EOC & $\epsilon_{\div B}$ & EOC & Runtime
    & $\epsilon_{B}$ & EOC & $\epsilon_{\div B}$ & EOC & Runtime
    \\
    \midrule
    $ 40$
    & 2.02e-02 &       & 2.56e-13 &       & 5.80e+00
    & 9.98e-05 &       & 3.47e-13 &       & 6.83e+00
    & 5.27e-07 &       & 3.93e-13 &       & 8.28e+00
    \\
    $ 60$
    & 9.00e-03 &  1.99 & 5.76e-13 & -2.00 & 4.39e+01
    & 1.97e-05 &  4.00 & 7.76e-13 & -1.99 & 5.19e+01
    & 4.64e-08 &  5.99 & 8.82e-13 & -1.99 & 6.61e+01
    \\
    $ 80$
    & 5.07e-03 &  2.00 & 1.03e-12 & -2.02 & 1.94e+02
    & 6.25e-06 &  4.00 & 1.38e-12 & -2.00 & 2.28e+02
    & 8.26e-09 &  6.00 & 1.58e-12 & -2.02 & 3.14e+02
    \\
    $100$
    & 3.24e-03 &  2.00 & 1.61e-12 & -2.00 & 5.74e+02
    & 2.56e-06 &  4.00 & 2.16e-12 & -2.00 & 7.48e+02
    & 2.17e-09 &  6.00 & 2.46e-12 & -1.99 & 9.48e+02
    \\
    \bottomrule
  \end{tabular*}

  \bigskip\medskip

  \caption{Results of convergence experiments for the nonlinear induction equation
           \eqref{eq:induction-transport-conservative-Hall} with analytical solution
           \eqref{eq:hall_periodic}, SBP operators of different order, and the
           choice of discrete forms as in \ref{itm:split-central-split}
           (split, central, split).}
  \label{tab:hall_periodic-split-central-split}
  \begin{tabular*}{\linewidth}{@{\extracolsep{\fill}}c|*5c|*5c|*5c@{}}
    \toprule
    & \multicolumn{5}{c}{Interior Order $2$}
    & \multicolumn{5}{c}{Interior Order $4$}
    & \multicolumn{5}{c}{Interior Order $6$}
    \\
    $N$
    & $\epsilon_{B}$ & EOC & $\epsilon_{\div B}$ & EOC & Runtime
    & $\epsilon_{B}$ & EOC & $\epsilon_{\div B}$ & EOC & Runtime
    & $\epsilon_{B}$ & EOC & $\epsilon_{\div B}$ & EOC & Runtime
    \\
    \midrule
    $ 40$
    & 2.02e-02 &       & 1.48e-03 &       & 5.82e+00
    & 9.98e-05 &       & 1.02e-05 &       & 7.19e+00
    & 5.27e-07 &       & 6.34e-08 &       & 9.03e+00
    \\
    $ 60$
    & 9.00e-03 &  2.00 & 5.30e-04 &  2.53 & 4.41e+01
    & 1.98e-05 &  4.00 & 1.71e-06 &  4.40 & 5.75e+01
    & 4.64e-08 &  5.99 & 4.78e-09 &  6.38 & 7.50e+01
    \\
    $ 80$
    & 5.07e-03 &  2.00 & 2.66e-04 &  2.40 & 1.87e+02
    & 6.25e-06 &  4.00 & 4.76e-07 &  4.44 & 2.44e+02
    & 8.26e-09 &  6.00 & 7.55e-10 &  6.41 & 3.40e+02
    \\
    $100$
    & 3.24e-03 &  2.00 & 1.63e-04 &  2.21 & 5.65e+02
    & 2.56e-06 &  4.00 & 1.82e-07 &  4.32 & 8.63e+02
    & 2.17e-09 &  6.00 & 1.88e-10 &  6.24 & 1.04e+03
    \\
    \bottomrule
  \end{tabular*}
\end{sidewaystable}

\begin{sidewaystable}
\centering
\addtolength\tabcolsep{-1.5pt}
  \caption{Results of convergence experiments for the nonlinear induction equation
           \eqref{eq:induction-transport-conservative-Hall} with analytical solution
           \eqref{eq:hall_periodic}, SBP operators of different order, and the
           choice of discrete forms as in \ref{itm:product-central-product}
           (product, central, product).}
  \label{tab:hall_periodic-product-central-product}
  \begin{tabular*}{\linewidth}{@{\extracolsep{\fill}}c|*5c|*5c|*5c@{}}
    \toprule
    & \multicolumn{5}{c}{Interior Order $2$}
    & \multicolumn{5}{c}{Interior Order $4$}
    & \multicolumn{5}{c}{Interior Order $6$}
    \\
    $N$
    & $\epsilon_{B}$ & EOC & $\epsilon_{\div B}$ & EOC & Runtime
    & $\epsilon_{B}$ & EOC & $\epsilon_{\div B}$ & EOC & Runtime
    & $\epsilon_{B}$ & EOC & $\epsilon_{\div B}$ & EOC & Runtime
    \\
    \midrule
    $ 40$
    & 2.03e-02 &       & 3.03e-03 &       & 5.81e+00
    & 9.99e-05 &       & 2.05e-05 &       & 6.94e+00
    & 5.27e-07 &       & 1.30e-07 &       & 8.66e+00
    \\
    $ 60$
    & 9.01e-03 &  2.00 & 1.07e-03 &  2.56 & 4.37e+01
    & 1.98e-05 &  4.00 & 3.42e-06 &  4.42 & 5.24e+01
    & 4.64e-08 &  5.99 & 9.72e-09 &  6.39 & 6.90e+01
    \\
    $ 80$
    & 5.07e-03 &  2.00 & 5.39e-04 &  2.40 & 1.93e+02
    & 6.25e-06 &  4.00 & 9.58e-07 &  4.43 & 2.45e+02
    & 8.26e-09 &  6.00 & 1.55e-09 &  6.38 & 3.03e+02
    \\
    $100$
    & 3.25e-03 &  2.00 & 3.29e-04 &  2.21 & 5.63e+02
    & 2.56e-06 &  4.00 & 3.68e-07 &  4.29 & 7.64e+02
    & 2.17e-09 &  6.00 & 3.87e-10 &  6.21 & 9.93e+02
    \\
    \bottomrule
  \end{tabular*}

  \bigskip\medskip

  \caption{Results of convergence experiments for the nonlinear induction equation
           \eqref{eq:induction-transport-conservative-Hall} with analytical solution
           \eqref{eq:hall_periodic}, SBP operators of different order, and the
           choice of discrete forms as in \ref{itm:product-central-split}
           (product, central, split).}
  \label{tab:hall_periodic-product-central-split}
  \begin{tabular*}{\linewidth}{@{\extracolsep{\fill}}c|*5c|*5c|*5c@{}}
    \toprule
    & \multicolumn{5}{c}{Interior Order $2$}
    & \multicolumn{5}{c}{Interior Order $4$}
    & \multicolumn{5}{c}{Interior Order $6$}
    \\
    $N$
    & $\epsilon_{B}$ & EOC & $\epsilon_{\div B}$ & EOC & Runtime
    & $\epsilon_{B}$ & EOC & $\epsilon_{\div B}$ & EOC & Runtime
    & $\epsilon_{B}$ & EOC & $\epsilon_{\div B}$ & EOC & Runtime
    \\
    \midrule
    $ 40$
    & 2.03e-02 &       & 3.03e-03 &       & 5.87e+00
    & 9.99e-05 &       & 2.05e-05 &       & 7.30e+00
    & 5.27e-07 &       & 1.30e-07 &       & 9.35e+00
    \\
    $ 60$
    & 9.01e-03 &  2.00 & 1.07e-03 &  2.56 & 4.28e+01
    & 1.98e-05 &  4.00 & 3.42e-06 &  4.42 & 5.53e+01
    & 4.64e-08 &  5.99 & 9.72e-09 &  6.39 & 7.23e+01
    \\
    $ 80$
    & 5.07e-03 &  2.00 & 5.39e-04 &  2.40 & 1.91e+02
    & 6.25e-06 &  4.00 & 9.58e-07 &  4.43 & 2.59e+02
    & 8.26e-09 &  6.00 & 1.55e-09 &  6.38 & 3.22e+02
    \\
    $100$
    & 3.25e-03 &  2.00 & 3.29e-04 &  2.21 & 5.83e+02
    & 2.56e-06 &  4.00 & 3.68e-07 &  4.29 & 8.29e+02
    & 2.17e-09 &  6.00 & 3.87e-10 &  6.21 & 1.06e+03
    \\
    \bottomrule
  \end{tabular*}

  \bigskip\medskip

  \caption{Results of convergence experiments for the nonlinear induction equation
           \eqref{eq:induction-transport-conservative-Hall} with analytical solution
           \eqref{eq:hall_periodic}, SBP operators of different order, and the
           choice of discrete forms as in \ref{itm:product-central-central}
           (product, central, central).}
  \label{tab:hall_periodic-product-central-central}
  \begin{tabular*}{\linewidth}{@{\extracolsep{\fill}}c|*5c|*5c|*5c@{}}
    \toprule
    & \multicolumn{5}{c}{Interior Order $2$}
    & \multicolumn{5}{c}{Interior Order $4$}
    & \multicolumn{5}{c}{Interior Order $6$}
    \\
    $N$
    & $\epsilon_{B}$ & EOC & $\epsilon_{\div B}$ & EOC & Runtime
    & $\epsilon_{B}$ & EOC & $\epsilon_{\div B}$ & EOC & Runtime
    & $\epsilon_{B}$ & EOC & $\epsilon_{\div B}$ & EOC & Runtime
    \\
    \midrule
    $ 40$
    & 2.03e-02 &       & 3.03e-03 &       & 5.73e+00
    & 9.99e-05 &       & 2.05e-05 &       & 6.99e+00
    & 5.27e-07 &       & 1.30e-07 &       & 8.36e+00
    \\
    $ 60$
    & 9.01e-03 &  2.00 & 1.07e-03 &  2.56 & 4.28e+01
    & 1.98e-05 &  4.00 & 3.42e-06 &  4.42 & 5.25e+01
    & 4.64e-08 &  5.99 & 9.72e-09 &  6.39 & 6.49e+01
    \\
    $ 80$
    & 5.07e-03 &  2.00 & 5.39e-04 &  2.40 & 1.90e+02
    & 6.25e-06 &  4.00 & 9.58e-07 &  4.43 & 2.50e+02
    & 8.26e-09 &  6.00 & 1.55e-09 &  6.38 & 3.08e+02
    \\
    $100$
    & 3.25e-03 &  2.00 & 3.29e-04 &  2.21 & 5.85e+02
    & 2.56e-06 &  4.00 & 3.68e-07 &  4.29 & 7.81e+02
    & 2.17e-09 &  6.00 & 3.87e-10 &  6.21 & 9.32e+02
    \\
    \bottomrule
  \end{tabular*}
\end{sidewaystable}

\begin{table}[!ht]
\centering
\begin{small}
\addtolength\tabcolsep{-1.5pt}
  \caption{Results of numerical experiments for the nonlinear induction equation
           \eqref{eq:induction-transport-conservative-Hall} using SBP operators of
           different order and the choice of discrete forms as in
           \ref{itm:central-central-central} (central, central, central).}
  \label{tab:hall_outflow-central-central-central}
  \begin{tabular*}{\linewidth}{@{\extracolsep{\fill}}c|*3c|*3c|*3c@{}}
    \toprule
    & \multicolumn{3}{c}{Interior Order $2$}
    & \multicolumn{3}{c}{Interior Order $4$}
    & \multicolumn{3}{c}{Interior Order $6$}
    \\
    $N$
    & $\norm{B}_M^2$ & $\norm{D_i B_i}_M$ & Runtime
    & $\norm{B}_M^2$ & $\norm{D_i B_i}_M$ & Runtime
    & $\norm{B}_M^2$ & $\norm{D_i B_i}_M$ & Runtime
    \\
    \midrule
    $ 40$
    & 5.68e+01 & 2.01e+01 & 3.71e+00
    & 4.88e+01 & 2.23e+01 & 5.36e+00
    & 4.51e+01 & 2.64e+01 & 7.36e+00
    \\
    $ 80$
    & 4.61e+01 & 2.79e+01 & 6.88e+01
    & 4.23e+01 & 3.13e+01 & 1.16e+02
    & 4.03e+01 & 3.67e+01 & 1.70e+02
    \\
    \bottomrule
  \end{tabular*}

  \bigskip

  \caption{Results of numerical experiments for the nonlinear induction equation
           \eqref{eq:induction-transport-conservative-Hall} using SBP operators of
           different order and the choice of discrete forms as in
           \ref{itm:split-central-split} (split, central, split).}
  \label{tab:hall_outflow-split-central-split}
  \begin{tabular*}{\linewidth}{@{\extracolsep{\fill}}c|*3c|*3c|*3c@{}}
    \toprule
    & \multicolumn{3}{c}{Interior Order $2$}
    & \multicolumn{3}{c}{Interior Order $4$}
    & \multicolumn{3}{c}{Interior Order $6$}
    \\
    $N$
    & $\norm{B}_M^2$ & $\norm{D_i B_i}_M$ & Runtime
    & $\norm{B}_M^2$ & $\norm{D_i B_i}_M$ & Runtime
    & $\norm{B}_M^2$ & $\norm{D_i B_i}_M$ & Runtime
    \\
    \midrule
    $ 40$
    & 5.66e+01 & 2.05e+01 & 3.48e+00
    & 4.86e+01 & 2.31e+01 & 5.00e+00
    & 4.48e+01 & 2.69e+01 & 7.15e+00
    \\
    $ 80$
    & 4.60e+01 & 2.88e+01 & 6.55e+01
    & 4.21e+01 & 3.28e+01 & 1.08e+02
    & 4.04e+01 & 3.78e+01 & 1.66e+02
    \\
    \bottomrule
  \end{tabular*}

  \bigskip

  \caption{Results of numerical experiments for the nonlinear induction equation
           \eqref{eq:induction-transport-conservative-Hall} using SBP operators of
           different order and the choice of discrete forms as in
           \ref{itm:product-central-product} (product, central, product).}
  \label{tab:hall_outflow-product-central-product}
  \begin{tabular*}{\linewidth}{@{\extracolsep{\fill}}c|*3c|*3c|*3c@{}}
    \toprule
    & \multicolumn{3}{c}{Interior Order $2$}
    & \multicolumn{3}{c}{Interior Order $4$}
    & \multicolumn{3}{c}{Interior Order $6$}
    \\
    $N$
    & $\norm{B}_M^2$ & $\norm{D_i B_i}_M$ & Runtime
    & $\norm{B}_M^2$ & $\norm{D_i B_i}_M$ & Runtime
    & $\norm{B}_M^2$ & $\norm{D_i B_i}_M$ & Runtime
    \\
    \midrule
    $ 40$
    & 5.73e+01 & 2.28e+01 & 3.35e+00
    & 4.96e+01 & 2.69e+01 & 5.41e+00
    & 4.55e+01 & 3.11e+01 & 7.50e+00
    \\
    $ 80$
    & 4.64e+01 & 3.26e+01 & 6.89e+01
    & 4.27e+01 & 3.83e+01 & 1.15e+02
    & 4.08e+01 & 4.35e+01 & 1.79e+02
    \\
    \bottomrule
  \end{tabular*}

  \bigskip

  \caption{Results of numerical experiments for the nonlinear induction equation
           \eqref{eq:induction-transport-conservative-Hall} using SBP operators of
           different order and the choice of discrete forms as in
           \ref{itm:product-central-split} (product, central, split).}
  \label{tab:hall_outflow-product-central-split}
  \begin{tabular*}{\linewidth}{@{\extracolsep{\fill}}c|*3c|*3c|*3c@{}}
    \toprule
    & \multicolumn{3}{c}{Interior Order $2$}
    & \multicolumn{3}{c}{Interior Order $4$}
    & \multicolumn{3}{c}{Interior Order $6$}
    \\
    $N$
    & $\norm{B}_M^2$ & $\norm{D_i B_i}_M$ & Runtime
    & $\norm{B}_M^2$ & $\norm{D_i B_i}_M$ & Runtime
    & $\norm{B}_M^2$ & $\norm{D_i B_i}_M$ & Runtime
    \\
    \midrule
    $ 40$
    & 5.72e+01 & 2.28e+01 & 3.41e+00
    & 4.95e+01 & 2.69e+01 & 5.36e+00
    & 4.56e+01 & 3.10e+01 & 7.46e+00
    \\
    $ 80$
    & 4.65e+01 & 3.26e+01 & 6.93e+01
    & 4.27e+01 & 3.83e+01 & 1.16e+02
    & 4.08e+01 & 4.36e+01 & 1.79e+02
    \\
    \bottomrule
  \end{tabular*}

  \bigskip

  \caption{Results of numerical experiments for the nonlinear induction equation
           \eqref{eq:induction-transport-conservative-Hall} using SBP operators of
           different order and the choice of discrete forms as in
           \ref{itm:product-central-central} (product, central, central).}
  \label{tab:hall_outflow-product-central-central}
  \begin{tabular*}{\linewidth}{@{\extracolsep{\fill}}c|*3c|*3c|*3c@{}}
    \toprule
    & \multicolumn{3}{c}{Interior Order $2$}
    & \multicolumn{3}{c}{Interior Order $4$}
    & \multicolumn{3}{c}{Interior Order $6$}
    \\
    $N$
    & $\norm{B}_M^2$ & $\norm{D_i B_i}_M$ & Runtime
    & $\norm{B}_M^2$ & $\norm{D_i B_i}_M$ & Runtime
    & $\norm{B}_M^2$ & $\norm{D_i B_i}_M$ & Runtime
    \\
    \midrule
    $ 40$
    & 5.71e+01 & 2.28e+01 & 3.36e+00
    & 4.96e+01 & 2.69e+01 & 5.59e+00
    & 4.55e+01 & 3.10e+01 & 7.71e+00
    \\
    $ 80$
    & 4.65e+01 & 3.26e+01 & 6.99e+01
    & 4.27e+01 & 3.83e+01 & 1.19e+02
    & 4.08e+01 & 4.36e+01 & 1.85e+02
    \\
    \bottomrule
  \end{tabular*}
\end{small}
\end{table}

\begin{table}[!ht]
\centering
  \caption{Errors and divergence norms of numerical solutions of the linear
           induction equation \eqref{eq:induction-transport-conservative} with
           analytical solution \eqref{eq:confined_domain} using $N=40$ nodes per
           direction, the second order SBP operator, and different divergence
           cleaning procedures.}
  \label{tab:confined_domain_divcleaning_N_40_order_2}
  \begin{tabular*}{\linewidth}{@{\extracolsep{\fill}}*7c@{}}
    \toprule
    $\partial_j(u_i B_j)$ & source term & $-\partial_j(u_j B_i)$
    & $\norm{B}_M^2$ & $\epsilon_{B}$ & $\epsilon_{\div B}$ & div. cleaning
    \\
    \midrule
    central & zero    & central & 7.49e-01 & 0.00e+00 & 2.79e-15 & WS, LN
    \\
    central & zero    & central & 7.49e-01 & 0.00e+00 & 2.79e-15 & WS, D$_0$
    \\
    central & zero    & central & 7.49e-01 & 0.00e+00 & 2.79e-15 & NS, D$_0$
    \\
    central & zero    & central & 7.49e-01 & 0.00e+00 & 2.79e-15 & none
    \\
    central & central & central & 7.49e-01 & 3.97e-16 & 5.18e-15 & WS, LN
    \\
    central & central & central & 7.49e-01 & 3.97e-16 & 5.18e-15 & WS, D$_0$
    \\
    central & central & central & 7.49e-01 & 3.97e-16 & 5.18e-15 & NS, D$_0$
    \\
    central & central & central & 7.49e-01 & 3.97e-16 & 5.18e-15 & none
    \\
    split   & central & split   & 7.49e-01 & 3.67e-16 & 7.72e-15 & WS, LN
    \\
    split   & central & split   & 7.49e-01 & 3.67e-16 & 7.72e-15 & WS, D$_0$
    \\
    split   & central & split   & 7.49e-01 & 3.67e-16 & 7.72e-15 & NS, D$_0$
    \\
    split   & central & split   & 7.49e-01 & 3.67e-16 & 7.72e-15 & none
    \\
    product & central & product & 7.49e-01 & 1.79e-15 & 3.09e-14 & WS, LN
    \\
    product & central & product & 7.49e-01 & 1.79e-15 & 3.09e-14 & WS, D$_0$
    \\
    product & central & product & 7.49e-01 & 1.79e-15 & 3.09e-14 & NS, D$_0$
    \\
    product & central & product & 7.49e-01 & 1.79e-15 & 3.09e-14 & none
    \\
    product & central & split   & 2.87e+01 & 5.30e+00 & 3.51e-03 & WS, LN
    \\
    product & central & split   & 3.39e+03 & 5.82e+01 & 2.58e+02 & WS, D$_0$
    \\
    product & central & split   & 3.79e+03 & 6.15e+01 & 2.58e+02 & NS, D$_0$
    \\
    product & central & split   & 7.20e+03 & 8.49e+01 & 2.86e+02 & none
    \\
    product & central & central & 1.92e+01 & 4.33e+00 & 1.09e-04 & WS, LN
    \\
    product & central & central & 1.02e+09 & 3.20e+04 & 1.54e+06 & WS, D$_0$
    \\
    product & central & central & 1.21e+09 & 3.47e+04 & 1.67e+06 & NS, D$_0$
    \\
    product & central & central & 1.24e+09 & 3.52e+04 & 1.68e+06 & none
    \\
    \bottomrule
  \end{tabular*}
\end{table}

\begin{table}[!ht]
\centering
  \caption{Errors and divergence norms of numerical solutions of the linear
           induction equation \eqref{eq:induction-transport-conservative} with
           analytical solution \eqref{eq:confined_domain} using $N=40$ nodes per
           direction, the fourth order SBP operator, and different divergence
           cleaning procedures.}
  \label{tab:confined_domain_divcleaning_N_40_order_4}
  \begin{tabular*}{\linewidth}{@{\extracolsep{\fill}}*7c@{}}
    \toprule
    $\partial_j(u_i B_j)$ & source term & $-\partial_j(u_j B_i)$
    & $\norm{B}_M^2$ & $\epsilon_{B}$ & $\epsilon_{\div B}$ & div. cleaning
    \\
    \midrule
    central & zero    & central & 7.50e-01 & 2.84e-03 & 3.50e-06 & WS, LN
    \\
    central & zero    & central & 7.50e-01 & 2.79e-03 & 2.39e-03 & WS, D$_0$
    \\
    central & zero    & central & 7.50e-01 & 2.75e-03 & 2.38e-03 & NS, D$_0$
    \\
    central & zero    & central & 7.50e-01 & 0.00e+00 & 1.77e-03 & none
    \\
    central & central & central & 7.50e-01 & 2.82e-03 & 2.06e-06 & WS, LN
    \\
    central & central & central & 7.50e-01 & 4.62e-03 & 3.18e-02 & WS, D$_0$
    \\
    central & central & central & 7.50e-01 & 4.73e-03 & 3.35e-02 & NS, D$_0$
    \\
    central & central & central & 7.50e-01 & 4.09e-03 & 3.68e-02 & none
    \\
    split   & central & split   & 7.50e-01 & 6.98e-03 & 3.62e-05 & WS, LN
    \\
    split   & central & split   & 7.50e-01 & 1.12e-02 & 5.37e-02 & WS, D$_0$
    \\
    split   & central & split   & 7.50-01 & 1.13e-02 & 4.64e-02 & NS, D$_0$
    \\
    split   & central & split   & 7.50e-01 & 3.50e-03 & 2.93e-02 & none
    \\
    product & central & product & 5.47e+00 & 2.17e+00 & 1.36e-02 & WS, LN
    \\
    product & central & product & 1.06e+00 & 5.56e-01 & 5.70e+00 & WS, D$_0$
    \\
    product & central & product & 1.18e+00 & 6.58e-01 & 3.16e+00 & NS, D$_0$
    \\
    product & central & product & 7.49e-01 & 6.65e-02 & 5.65e-01 & none
    \\
    product & central & split   & 7.49e-01 & 2.69e-02 & 5.62e-05 & WS, LN
    \\
    product & central & split   & 1.34e+00 & 7.69e-01 & 4.44e+00 & WS, D$_0$
    \\
    product & central & split   & 1.43e+00 & 8.25e-01 & 4.04e+00 & NS, D$_0$
    \\
    product & central & split   & 2.49e+00 & 1.33e+00 & 4.85e+00 & none
    \\
    product & central & central & 7.61e-01 & 1.01e-01 & 6.05e-04 & WS, LN
    \\
    product & central & central & 2.02e+06 & 1.42e+03 & 9.46e+04 & WS, D$_0$
    \\
    product & central & central & 2.60e+06 & 1.61e+03 & 1.07e+05 & NS, D$_0$
    \\
    product & central & central & 2.68e+06 & 1.64e+03 & 1.09e+05 & none
    \\
    \bottomrule
  \end{tabular*}
\end{table}

\begin{table}[!ht]
\centering
  \caption{Errors and divergence norms of numerical solutions of the nonlinear
           induction equation \eqref{eq:induction-transport-conservative-Hall}
           with outflow boundary conditions \eqref{eq:outflow-BCs-weak} using
           $N=40$ nodes per direction, the second order SBP operator, and different
           divergence cleaning procedures.}
  \label{tab:hall_outflow_divcleaning_N_40_order_2}
  \begin{tabular*}{\linewidth}{@{\extracolsep{\fill}}*6c@{}}
    \toprule
    $\partial_j(u_i B_j)$ & source term & $-\partial_j(u_j B_i)$
    & $\norm{B}_M^2$ & $\epsilon_{\div B}$ & div. cleaning
    \\
    \midrule
    central &    zero & central & 2.35e+01 & 1.27e-04 & WS, LN
    \\
    central &    zero & central & 2.90e+01 & 8.80e+00 & WS, D$_0$
    \\
    central &    zero & central & 2.94e+01 & 9.27e+00 & NS, D$_0$
    \\
    central &    zero & central &      NaN &      NaN & none
    \\
    central & central & central & 2.41e+01 & 8.79e-05 & WS, LN
    \\
    central & central & central & 2.40e+01 & 5.70e+00 & WS, D$_0$
    \\
    central & central & central & 2.42e+01 & 5.63e+00 & NS, D$_0$
    \\
    central & central & central & 3.38e+01 & 1.52e+01 & none
    \\
      split & central &   split & 2.38e+01 & 1.38e-04 & WS, LN
    \\
      split & central &   split & 2.57e+01 & 5.79e+00 & WS, D$_0$
    \\
      split & central &   split & 2.57e+01 & 5.61e+00 & NS, D$_0$
    \\
      split & central &   split & 3.48e+01 & 2.07e+01 & none
    \\
    product & central & product & 2.59e+01 & 1.16e-04 & WS, LN
    \\
    product & central & product & 3.50e+01 & 7.42e+00 & WS, D$_0$
    \\
    product & central & product & 3.42e+01 & 7.51e+00 & NS, D$_0$
    \\
    product & central & product & 5.25e+01 & 5.71e+01 & none
    \\
    product & central &   split & 2.56e+01 & 1.74e-04 & WS, LN
    \\
    product & central &   split & 3.32e+01 & 7.25e+00 & WS, D$_0$
    \\
    product & central &   split & 3.30e+01 & 7.41e+00 & NS, D$_0$
    \\
    product & central &   split & 5.39e+01 & 5.91e+01 & none
    \\
    product & central & central & 2.56e+01 & 1.46e-04 & WS, LN
    \\
    product & central & central & 3.17e+01 & 7.01e+00 & WS, D$_0$
    \\
    product & central & central & 3.30e+01 & 7.26e+00 & NS, D$_0$
    \\
    product & central & central & 5.40e+01 & 5.90e+01 & none
    \\
    \bottomrule
  \end{tabular*}
\end{table}

\begin{table}[!ht]
\centering
  \caption{Errors and divergence norms of numerical solutions of the nonlinear
           induction equation \eqref{eq:induction-transport-conservative-Hall}
           with outflow boundary conditions \eqref{eq:outflow-BCs-weak} using
           $N=40$ nodes per direction, the fourth order SBP operator, and different
           divergence cleaning procedures.}
  \label{tab:hall_outflow_divcleaning_N_40_order_4}
  \begin{tabular*}{\linewidth}{@{\extracolsep{\fill}}*6c@{}}
    \toprule
    $\partial_j(u_i B_j)$ & source term & $-\partial_j(u_j B_i)$
    & $\norm{B}_M^2$ & $\epsilon_{\div B}$ & div. cleaning
    \\
    \midrule
    central &    zero & central &      NaN &      NaN & WS, LN
    \\
    central &    zero & central &      NaN &      NaN & WS, D$_0$
    \\
    central &    zero & central &      NaN &      NaN & NS, D$_0$
    \\
    central &    zero & central &      NaN &      NaN & none
    \\
    central & central & central &      NaN &      NaN & WS, LN
    \\
    central & central & central & 1.95e+01 & 7.08e+00 & WS, D$_0$
    \\
    central & central & central & 1.98e+01 & 6.83e+00 & NS, D$_0$
    \\
    central & central & central & 2.53e+01 & 1.63e+01 & none
    \\
      split & central &   split &      NaN &      NaN & WS, LN
    \\
      split & central &   split & 1.99e+01 & 6.64e+00 & WS, D$_0$
    \\
      split & central &   split & 2.04e+01 & 6.41e+00 & NS, D$_0$
    \\
      split & central &   split & 2.52e+01 & 2.10e+01 & none
    \\
    product & central & product & 2.10e+01 & 5.03e-04 & WS, LN
    \\
    product & central & product & 2.39e+01 & 7.60e+00 & WS, D$_0$
    \\
    product & central & product & 2.30e+01 & 6.94e+00 & NS, D$_0$
    \\
    product & central & product & 2.57e+01 & 4.09e+01 & none
    \\
    product & central &   split & 2.10e+01 & 3.76e-04 & WS, LN
    \\
    product & central &   split & 2.37e+01 & 7.70e+00 & WS, D$_0$
    \\
    product & central &   split & 2.29e+01 & 7.08e+00 & NS, D$_0$
    \\
    product & central &   split & 2.55e+01 & 4.11e+01 & none
    \\
    product & central & central & 2.11e+01 & 2.20e-04 & WS, LN
    \\
    product & central & central & 2.39e+01 & 7.63e+00 & WS, D$_0$
    \\
    product & central & central & 2.30e+01 & 6.96e+00 & NS, D$_0$
    \\
    product & central & central & 2.75e+01 & 4.12e+01 & none
    \\
    \bottomrule
  \end{tabular*}
\end{table}

\printbibliography

\end{document}